\documentclass[11pt]{article}
\usepackage{mathrsfs}
\usepackage{xcolor}
\usepackage[a4paper,margin=1.1in]{geometry}
\usepackage{tocloft}
\usepackage{latexsym}
\usepackage{epsfig}
\usepackage{amsthm}
\usepackage{amssymb}
\usepackage{amsfonts}
\usepackage{amsmath,amsbsy}
\usepackage{epstopdf}
\usepackage{setspace}
\usepackage{mhequ}
\usepackage{centernot}
\usepackage{hyperref}
\usepackage{tikz}
\usetikzlibrary{shapes.misc}
\usetikzlibrary{shapes.symbols}
\usetikzlibrary{shapes.geometric}
\usetikzlibrary{snakes}
\usetikzlibrary{decorations}
\usetikzlibrary{decorations.markings}

\colorlet{darkgreen}{green!60!black}

\def\drawx{\draw[-,solid] (-3pt,-3pt) -- (3pt,3pt);\draw[-,solid] (-3pt,3pt) -- (3pt,-3pt);}
\tikzset{
	root/.style={circle,fill=testcolor,inner sep=0pt, minimum size=2mm},
	dot/.style={circle,fill=black,inner sep=0pt, minimum size=2.5mm},
	circ/.style={circle,draw=black,inner sep=0pt, minimum size=1mm},
	var/.style={circle,fill=black!10,draw=black,inner sep=0pt, minimum size=2mm},
	dotred/.style={circle,fill=black!50,inner sep=0pt, minimum size=2mm},
	generic/.style={semithick,shorten >=1pt,shorten <=1pt},
	oddfunc/.style={generic, dotted},
	dist/.style={ultra thick,draw=testcolor,shorten >=1pt,shorten <=1pt},
	testfcn/.style={ultra thick,testcolor,shorten >=1pt,shorten <=1pt,<-},
	testfunction/.style={ultra thick,testcolor,shorten >=1pt,shorten <=1pt},
	testfcnx/.style={ultra thick,testcolor,shorten >=1pt,shorten <=1pt,<-,
		postaction={decorate,decoration={markings,mark=at position 0.6 with {\drawx}}}},
	kprime/.style={semithick,shorten >=1pt,shorten <=1pt,densely dashed,->},
	kprimex/.style={semithick,shorten >=1pt,shorten <=1pt,densely dashed,->,
		postaction={decorate,decoration={markings,mark=at position 0.4 with {\drawx}}}},
	kernel/.style={semithick,shorten >=1pt,shorten <=1pt,->,draw=black},
	multx/.style={shorten >=1pt,shorten <=1pt,
		postaction={decorate,decoration={markings,mark=at position 0.5 with {\drawx}}}},
	kernelx/.style={semithick,shorten >=1pt,shorten <=1pt,->,
		postaction={decorate,decoration={markings,mark=at position 0.4 with {\drawx}}}},
	kernel1/.style={->,semithick,shorten >=1pt,shorten <=1pt,postaction={decorate,decoration={markings,mark=at position 0.45 with {\draw[-] (0,-0.1) -- (0,0.1);}}}},
	kernel2/.style={->,semithick,shorten >=1pt,shorten <=1pt,postaction={decorate,decoration={markings,mark=at position 0.45 with {\draw[-] (0.05,-0.1) -- (0.05,0.1);\draw[-] (-0.05,-0.1) -- (-0.05,0.1);}}}},
	wave/.style={snake=coil,segment aspect=0.5},
	rho/.style={dotted,semithick,shorten >=1pt,shorten <=1pt},
	renorm/.style={shape=circle,fill=white,inner sep=1pt},
	labl/.style={shape=rectangle,fill=white,inner sep=1pt},
cumu2n/.style={inner sep=3pt},
cumu2/.style={draw=red!80,fill=red!40},
cumu2b/.style={draw=blue!80,fill=blue!40},
cumu2nv/.style={inner sep=3pt},
cumu2v/.style={draw=red!80,fill=white,very thick},
cumu3/.style={regular polygon, regular polygon sides=3,draw=red!80,rounded corners=3pt,fill=red!40,minimum size=5mm},
cumu4/.style={regular polygon, regular polygon sides=4,draw=red!80,rounded corners=3pt,fill=red!40,minimum size=7mm},
cumu5/.style={regular polygon, regular polygon sides=5,draw=red!80,rounded corners=3pt,fill=red!40,minimum size=7mm},
	xi/.style={circle,fill=symbols!10,draw=symbols,inner sep=0pt,minimum size=1.2mm},
	xix/.style={crosscircle,fill=symbols!10,draw=symbols,inner sep=0pt,minimum size=1.2mm},
	xib/.style={circle,fill=symbols!10,draw=symbols,inner sep=0pt,minimum size=1.6mm},
	xibx/.style={crosscircle,fill=symbols!10,draw=symbols,inner sep=0pt,minimum size=1.6mm},
	not/.style={circle,fill=symbols,draw=symbols,inner sep=0pt,minimum size=0.5mm},
	>=stealth,
	}

\newtheorem{theorem}{Theorem}[section]
\newtheorem{thm}[theorem]{Theorem}
\newtheorem{lemma}[theorem]{Lemma}
\newtheorem{lem}[theorem]{Lemma}
\newtheorem{proposition}[theorem]{Proposition}

\newtheorem{corollary}[theorem]{Corollary}

\theoremstyle{definition}
\newtheorem{definition}[theorem]{Definition}
\newtheorem{defn}[theorem]{Definition}

\newtheorem{remark}[theorem]{Remark}
\newtheorem{rem}[theorem]{Remark}
\numberwithin{equation}{section}

 \DeclareMathAlphabet{\mathpzc}{OT1}{pzc}{m}{it}
 \DeclareMathOperator\supp{supp}

 \newcommand{\eqdef}{\stackrel{\mbox{\tiny def}}{=}}

\newcommand{\Ell}{\ensuremath{\boldsymbol\ell}}

 \newcommand{\E}{\mathbb{E}}            
  \newcommand{\ka}{\kappa}
 \newcommand{\CZ}{\mathcal{Z}}
 
 \newcommand{\e}{\varepsilon}

 \newcommand{\N}{\mathbb{N}}
 \newcommand{\R}{\mathbb{R}}
 \newcommand{\Z}{\mathbb{Z}}

 \newcommand{\Be}{\begin{equation}}
 \newcommand{\Ee}{\end{equation}}
 \newcommand{\Bs}{\begin{split}}
 \newcommand{\Es}{\end{split}}
  \newcommand{\Bes}{\begin{equation*}}
 \newcommand{\Ees}{\end{equation*}}
 \newcommand{\BT}{\begin{thm}}
 \newcommand{\ET}{\end{thm}}
 \newcommand{\Bp}{\begin{proof}}
 \newcommand{\Ep}{\end{proof}}
 \newcommand{\BL}{\begin{lem}}
 \newcommand{\EL}{\end{lem}}
 \newcommand{\BP}{\begin{proposition}}
 \newcommand{\EP}{\end{proposition}}
 \newcommand{\BC}{\begin{corollary}}
 \newcommand{\EC}{\end{corollary}}
 \newcommand{\BR}{\begin{rem}}
 \newcommand{\ER}{\end{rem}}
 \newcommand{\BD}{\begin{defn}}
 \newcommand{\ED}{\end{defn}}
 \newcommand{\BI}{\begin{itemize}}
 \newcommand{\EI}{\end{itemize}}

  \newcommand{\CP}{\mathcal{P}}
    \newcommand{\CC}{\mathcal{C}}

  \newcommand{\Noise}{\Xi}

\definecolor{darkergreen}{rgb}{0.0, 0.5, 0.0}

\newcommand{\bP}{{\mathbf P}}
\newcommand{\bE}{{\mathbf E}}
\renewcommand{\P}{\mathbb{P}}

\newcommand{\1}{\mathbf{1}}
\def\H{\mathcal H}

\begin{document}

\title{Scaling limit of a directed polymer among a Poisson field of independent walks}

\author{Hao Shen\footnote{University of Wisconsin-Madison, Van Vleck Hall, US. Email: pkushenhao@gmail.com},
$\;$
Jian Song\footnote{Research Center for Mathematics and Interdisciplinary Sciences, Shandong University, Qingdao, Shandong, 266237, China and School of Mathematics, Shandong University, Jinan, Shandong, 250100, China.
Email: txjsong@hotmail.com},
$\;$
Rongfeng Sun\footnote{Department of Mathematics,
		National University of Singapore,
		10 Lower Kent Ridge Road, 119076 Singapore.
		Email: matsr@nus.edu.sg},
$\;$
 Lihu Xu\footnote{1.\ Department of Mathematics, Faculty of Science and Technology,
 	University of Macau, Av.\ Padre Tom\'{a}s Pereira, Taipa Macao, China; 2.\ UMacau Zhuhai Research Institute, Zhuhai, China.  Email: lihuxu@um.edu.mo}
}

\maketitle

\vspace{-5ex}

\begin{abstract}
We consider a directed polymer model in dimension $1+1$, where the disorder is given by the occupation field of
a Poisson system of independent random walks on $\Z$. In a suitable continuum and weak disorder limit, we show
that the family of quenched  partition functions of the directed polymer converges to the Stratonovich solution of
a multiplicative stochastic heat equation (SHE) with a Gaussian noise, whose space-time covariance is given
by the heat kernel.

In contrast to the case with space-time white noise where the solution of the
SHE admits a Wiener-It\^{o} chaos expansion, we establish an $L^1$-convergent
chaos expansions of iterated  integrals generated by Picard iterations. Using this expansion
and its discrete counterpart for the polymer partition functions, the convergence of the terms in the expansion is proved
via functional analytic arguments and heat kernel estimates. The Poisson random walk system is amenable to careful moment analysis, which
is an important input to our arguments.
\end{abstract}

\setcounter{tocdepth}{1}
\tableofcontents

\section{Introduction}

The directed polymer in random environment (disorder) is a classic model in the study of disordered systems. See e.g.\ the lecture notes by Comets~\cite{Comets17}. It is also intimately connected to the stochastic heat equation with multiplicative noise. In particular, Alberts, Khanin and Quastel~\cite{AKQ} showed that in dimension $1+1$, when the space-time random environment consists of i.i.d.\ random variables, in a suitable weak disorder and diffusive space-time scaling limit, the family of polymer partition functions (indexed by the starting point of the polymer) converges in distribution to the solution of the multiplicative stochastic heat equation (SHE) with {\it space-time white noise}. More generally, for so-called disorder relevant systems,
such an intermediate disorder regime is often expected to exist and leads to non-trivial continuum disordered systems, such as the SHE for the directed polymer. See \cite{CSZ17a} for more details and other examples such as the random field Ising model. In the language of singular SPDE's, such results are also known as weak universality results because the limit does not depend on the details of the disorder on the microscopic scale, and they even hold for a general class of discrete models, see for instance
\cite{Bertini1997,Amir11} for SHE convergence of asymmetric simple exclusion processes (ASEP) and
\cite{Dembo2016,CSL16,Labbe16b,corwin2016open,G17,MR3907953} for variants of ASEP, and \cite{CT15,CGST} for six-vertex model or its higher-spin generalizations,
and \cite{jara2019stationary} for convergence of another directed polymer model introduced by O'Connell-Yor \cite{MR1865759} (though the proof is via connection to stochastic Burgers equation).
Let us also mention the work
\cite{MR3785598} in which a singular polymer measure defined on continuous paths is constructed in $d=2,3$.
This SPDE in connection with the KPZ equation via Hopf-Cole transformation leads to a vast literature;  see \cite[Section~6]{corwin2019Bulletin} for further references.

When the random environment in the directed polymer model consists of i.i.d.\ random variables indexed by space that do no vary with time, a similar
continuum and weak disorder limit exists in dimensions $d<4$, which is an SHE driven by {\it spatial white noise}, known as the Parabolic Anderson model (PAM). In dimensions $d=2,3$, we need the theory of regularity structure/paracontrolled distributions \cite{GIP,Hairer14,MR3779690} to define the solution of PAM.
Weak universality results are also proved \cite{MR3736653,MR4029148}; and in $d=2$, a series expansion solution for PAM for small time is recently defined in \cite{MR3798237}, which is to some extent related to the method exploited in the present paper.
Note that of particular interest is the recent papers \cite{CTInhomogeneous} and \cite{perkowski2019rough} in which
(generalized) SHEs driven by a {\it mix} of spacetime and spatial multiplicative noises are derived from
ASEP in a spatially inhomogeneous environment and branching random walk in static random environment.

It is natural to ask whether similar weak universality results hold for directed polymer in more general dependent space-time random environment, and what limits do we obtain. This motivates us to consider a random environment defined by the occupation field of a Poisson system of random walks, which has been studied extensively and used as a dynamic random environment in various contexts, see e.g.~\cite{CG84, VS05, GdH06, DGRS12, MdHS+}.  There are several reasons for working with Poisson random walks.  They exhibit non-trivial space-time fluctuations which are representative of a wide class of models known as the Edwards-Anderson universality class (see e.g.~\cite{Sep05, Ravishankar, PS08}). On the other hand, the independence of the walks allow explicit calculations not possible for other models. In the context of the directed polymer model, which is closely linked to the parabolic Anderson model, the Poisson field of random walks can naturally be interpreted as a catalytic potential where the catalysts undergo independent motion (see e.g.~\cite{GdH06} for further background). With the occupation field of Poisson random walks as our random potential, we will show that in $d=1$, the polymer partition functions converge to the solution to a new SPDE, which is an SHE driven by a multiplicative Gaussian noise with a {\it long-range but singular} covariance, more precisely, the heat kernel. We also conjecturally expect that a similar convergence would hold for this new SPDE in $d<4$.

We will consider both the polymer and the Poisson walks in continuous time.

\begin{definition}(Poisson field of independent walks.)
At time $t=0$, we start with $\xi(0,x)$ particles for each $x \in \Z^d$ such that $\xi(0,x)$ are i.i.d.\ Poisson with mean $\lambda>0$. Each particle then performs rate $1$ independent simple symmetric random walk on $\Z^d$. Denote by $\xi(t,x)$ the number of particles at position $x$ and time $t$. More precisely,
\begin{equ}\label{e:xi}
\xi(t,x) = \sum_{y\in\Z^d} \sum_{i=1}^{\xi(0,y)} 1_{\{Y^{y,i}_t=x\}},
\end{equ}
where $Y^{y,i}$ is the $i$-th random walk starting from $y$ at time $0$. It is easily seen that i.i.d.\ Poisson product measure with mean $\lambda$ is an invariant measure of the dynamics $\xi(t,\cdot)$. We will denote probability and expectation with respect to $\xi$ by $\P$ and $\E$, respectively.
\end{definition}

The directed polymer will be modeled by a rate $1$ simple symmetric random walk $S$. If $S$ starts from position $x$ at time $t$, then we denote
its probability and expectation by $\bP_{(t,x)}$ and $\bE_{(t,x)}$ respectively, and by $\bP$ and $\bE$ when $(t,x)=(0,0)$. We will denote
the transition kernel for $S$ by $P_t(x) := \bP(S_t=x)$.

\begin{definition}(Directed polymer in the environment $\xi$.)
Let $\{\xi(t,x)\}_{t \ge 0, x \in \Z^d}$ be defined from the mean $\lambda$ Poisson field of random walks as above, and let $S$ be a simple symmetric random walk on $\Z^d$. Given $T>0$ and coupling constant $\beta\in \R$, define
\begin{align}
\Lambda_\beta(\xi, S):=\exp\left[\beta \int_0^T \tilde \xi(s,S_s) {\rm d}s\right], \qquad \tilde \xi(s,x):= \xi(s,x)-\lambda.
\end{align}
The (quenched) polymer measure in the random environment $\xi$ is then defined via a Gibbs change of measure for $S$ with weight given by $\Lambda_\beta(\xi, S)/Z^\xi_{T, \beta}$, where $Z^\xi_{T, \beta} := \bE[\Lambda_\beta(\xi, S)]$ is known as the quenched partition function.
Furthermore, we will allow the starting space-time point of the polymer $S$ to vary and consider the family of partition functions
\begin{equ} [e:Z]
Z^\xi_{T, \beta} (t,x) := \bE_{(t,x)}\left[\exp\left[\beta \int_t^T \tilde \xi(s,S_s) {\rm d}s\right]\right], \quad t\in [0, T], x\in \Z^d.
\end{equ}
For $x=(x_{1},...,x_{d}) \not\in \Z^d$, we define $Z^\xi_{T, \beta} (t,x):= Z^\xi_{T, \beta} (t,\lceil x\rceil)$  with $\lceil x\rceil=(\lceil x_{1}\rceil,...,\lceil x_{d}\rceil)$.
\end{definition}

It turns out that in dimension $d=1$, in order for the family of partition functions $(Z^\xi_{T, \beta}(t,x))_{t\in [0,T], x\in\R}$ to have non-trivial random limits, we need to scale space-time and $\beta$ as follows:

{\bf Scaling:} Let $\e=1/\sqrt{T}$. We will rescale $\beta$, $t$ and $x$ as follows:
\begin{equ} [e:scaling]
\beta \to \e^{\frac32} \beta\;,
\qquad
x \to  \e^{-1} x \;,
\qquad
t \to \e^{-2} t \;, \quad \quad \beta\in\R,\ x\in \R,\ t\geq 0.
\end{equ}

Our main result is then the following.

\BT[Convergence of polymer partition functions] \label{T:main} Given $\e>0$ and $\beta\in\R$, let
$(Z^\xi_{\e^{-2},\, \e^{3/2}\beta}(\e^{-2}t, \e^{-1}x))_{t\in [0,1], x\in \R}$ be the family of polymer partition functions defined as in \eqref{e:Z}. Then as $\e\downarrow 0$, it converges in finite-dimensional distribution to $(u(1-t,x))_{t\in [0,1], x\in\R}$, where $u$ is the Stratonovich solution
of the SPDE
\begin{equation}\label{spde}
\frac{\partial }{\partial t}u(t, x)=\frac12 \Delta  u(t,x)+ \beta \sqrt{\lambda} u \, \Noise(t, x), \qquad u(0, \cdot)\equiv 1,
\end{equation}
and $\Noise$ is a generalized Gaussian field with covariance
\begin{equ} [Xi]
\E [\Noise (t,x) \Noise(s,y)]=p_{|t-s|}(x-y) := \frac{1}{\sqrt{2\pi |t-s|}} e^{-\frac{|x-y|^2}{2|t-s|}}.
\end{equ}
\ET

\begin{remark}\label{R:EW}
We note that  $\Xi$ is in fact the stationary solution of the SPDE
\begin{equation}
\frac{\partial }{\partial t} \Xi = \frac12 \Delta \Xi + \partial_x \digamma(t,x),
\end{equation}
where $\digamma$ is a space-time white noise on $\R_+\times \R$. In particular, $\Xi= \partial_x u$, where $u$ is the solution of the additive stochastic heat (a.k.a.\ Edwards-Wilkinson) equation 
\begin{equation}\label{eq:EW}
\frac{\partial }{\partial t} u = \frac12 \Delta u + \digamma,
\end{equation}
with $u(0,\cdot)$ being a two-sided Brownian motion. This is consistent with the fact that the occupation field $\xi$ of the Poisson random walks is the spatial increment of a random growth model, whose fluctuation field converges to the solution of the Edwards-Wilkinson equation~\cite{Sep05}. 
\end{remark}

Note that when the polymer disorder $\xi(t,x)$ is white noise on $[0,\infty)\times \Z$, it is the continuous time analogue of the directed polymer considered in \cite{AKQ} and the correct scaling is $\beta \to \e^{\frac12} \beta$ (see also \cite{CSZ17a}). The critical dimension is $d=2$, at and above which,  it is generally believed that there is no non-trivial SPDE limit (see \cite{CSZ17b}). When $\xi(t,x)=\xi(0,x)$ is white noise on $\Z^d$, the correct scaling turns out to be $\beta \to \e^{2-\frac{d}{2}} \beta$. The critical dimension is instead $d=4$. Theorem~\ref{T:main} shows that our disorder $\xi$ requires the same scaling as the case of a spatial white noise. In other words, the power law decay (in time) of the covariance is not fast enough to change the critical spatial dimension when compared to the case of having no decay in time (white noise in space only). In dimensions $d=2,3$, similar to the parabolic Anderson model, we expect that the partition function requires renormalization in order to converge to the solution of a singular SPDE. It would also be interesting to find a class of models that interpolate between the PAM and SHE, and see what speed of decay for the noise covariance will alter the critical dimension of the model.

\subsection{Outline of proof}\label{S:outline}

In this section, we outline the main steps in our proof of Theorem \ref{T:main}. The basic strategy is to expand the polymer partition function $Z^\xi_{T, \beta}$ in terms of a series (chaos) expansion with respect to the disorder $\tilde \xi$. Unlike the case with i.i.d.\
noise as considered in \cite{AKQ, CSZ17a}, the terms of the expansion are not $L^2$ orthogonal. Instead, we will show that the series
is $L^1$-convergent and bound the tail of the series uniformly as $\e\downarrow 0$. We then show that the first $m$ terms of the series converge to a sum of iterated Stratonovich integrals with respect to the Gaussian noise $\Xi$, which is in fact a series representation of the solution for the SPDE \eqref{spde}.  Stratonovich integrals arise naturally in this context because in the expansion \eqref{eq:Zexp} below, we have the products of $\tilde \xi$. Due to the non-trivial covariance structure of $\tilde \xi$, their products cannot be replaced by Wick products via a simple normalisation of $Z^\xi_{T,\beta}$, unlike the case when $\tilde \xi$ is a field of independent random variables as in~\cite{AKQ}.

First we expand the partition function as follows:
\begin{align}
Z^\xi_{T,\beta}(t_0,x_0)
& = 1+\sum_{k=1}^\infty \frac{\beta^k}{k!} \bE_{(t_0,x_0)}\left[\Big(\int_{t_0}^T \tilde \xi(s, S_s) {\rm d}s\Big)^k  \right] \nonumber\\
& = 1 + \sum_{k=1}^\infty \beta^k \!\!\!\!\!\! \sum\limits_{x_1, \ldots, x_k\in\Z}\ \ \idotsint\limits_{t_0<t_1<\cdots<t_k<T}  \prod_{i=1}^k P_{t_i-t_{i-1}}(x_i-x_{i-1})  \prod_{i=1}^k \tilde \xi(t_i, x_i) {\rm d}t_i
\nonumber\\
& =:
1 +\sum_{k=1}^\infty Z_{T,\beta}^{(k)}(t_0,x_0) \;.
\label{eq:Zexp}
\end{align}
Note that this expansion is reminiscent of Wiener chaos expansion with respect to $\tilde \xi$, except that $\tilde \xi$ is correlated and non-Gaussian.
However, $\tilde \xi$ does converge to a generalised Gaussian field in the diffusive scaling limit.

\BP[Convergence to Gaussian field]\label{P:xilim}
As $\e\downarrow 0$, $\e^{-\frac12}\tilde \xi(\e^{-2}\cdot, \e^{-1}\cdot)$ converges in distribution to $\sqrt{\lambda}\Xi$
in the weighted Besov-H\"older space $\mathcal C^\alpha_\ka$ (defined in Appendix~\ref{S:Besov})  for every $\alpha<-1/2$ and $\ka> 0$,  where $\Xi$ is the generalised Gaussian field
defined in \eqref{Xi}.
\EP 

\begin{remark}
Our proof of Proposition \ref{P:xilim} (in Section \ref{S:fieldconv}) relies on direct calculations for the Poisson random walks. More results of this type can be found in the literature on hydrodynamic limits \cite[Section 11]{KL99}. As pointed out by a referee, an alternative approach is to regard $\xi$ as the spatial increment of a random growth model, as mentioned in Remark
\ref{R:EW}, and first show that the fluctuation field of the associated height function converges to the solution of the Edwards-Wilkinson equation \eqref{eq:EW}. One can then use the continuity of $f\to \partial_x f$ from ${\mathcal C}^{\alpha+1}_\ka \to {\mathcal C}^\alpha_\ka$ to deduce the convergence of $\xi$. This approach should also be applicable to $\xi$ with more general initial conditions, as well as to other models in the EW universality class, such as the symmetric simple exclusion processes~\cite{Ravishankar, KL99}.
\end{remark}

We will show that the solution of the SPDE \eqref{spde} has a similar expansion as in \eqref{eq:Zexp}, where $\tilde \xi$ is replaced by its Gaussian limit $\sqrt{\lambda}\Xi$.

\BP[Solution of SPDE] \label{P:uexp}
The SPDE \eqref{spde} has a well-defined mild Stratonovich solution $u$. For all $(t,x) \in [0,1]\times \R$,  $u$ admits the series representation
\begin{equ} [e:uexp]
u(t, x) = \sum_{k=0}^\infty u^{(k)}(t,x) =
1 +
 \sum_{k=1}^\infty (\beta\sqrt{\lambda})^k  \idotsint\limits_{0<s_1<\cdots<s_k<t \atop x_1, \ldots, x_k\in\R}
 \prod_{i=1}^k p_{s_{i+1}-s_{i}}(x_{i+1}-x_{i})  \prod_{i=1}^k \Xi(s_i, x_i) {\rm d}x_i\, {\rm d}s_i,
\end{equ}
where $(s_{k+1}, x_{k+1}):=(t,x)$, and the series converges in $L^1$.
\EP

\begin{remark}
Note that \eqref{e:uexp} formally resembles the discrete series \eqref{eq:Zexp} if we reverse time and identify the sequence
$t_0<t_1<\cdots <t_k<T$ with $t>s_k>\cdots >s_1>0$, and identify $(\tilde \xi(s, \cdot))_{s\in [0,T]}$ with $(\sqrt{\lambda}\Xi(1-s, \cdot))_{s\in [0,1]}$. The fact that  the iterated integrals in \eqref{e:uexp} are well-defined, and that each term in the series   \eqref{eq:Zexp} converges to the corresponding iterated integral in \eqref{e:uexp} (for which the precise statement is Lemma~\ref{L:termconv} below), will be clear in Section~\ref{S:conv}.
\end{remark}

To prove Theorem \ref{T:main}, it then suffices to show firstly that the series expansions for $Z^\xi_{\e^{-2}, \e^{3/2}\beta} (\e^{-2}t, \e^{-1}x)$ can be truncated to a fixed order with an arbitrarily small error that is uniform as $\e\downarrow 0$; secondly, show that terms of the series
in the expansion for $Z^\xi_{\e^{-2}, \e^{3/2}\beta} (\e^{-2} t, \e^{-1}x)$ converge in joint distribution (possibly with different choices of $(t,x)$) to terms of the series for $u$ in \eqref{e:uexp}. The following two lemmas address these two issues.

\BL[Bounds on partition functions]\label{L:Zbound}
Given $(t,x) \in [0,1]\times \R$, let
\[
\CZ_{\e,\beta}(t,x):= Z^\xi_{\e^{-2}, \e^{3/2}\beta} (\e^{-2}t,\e^{-1}x)
\]
 where $Z^\xi$ is  the polymer partition function defined as in \eqref{e:Z}, with terms of its expansion in \eqref{eq:Zexp} denoted by
 \begin{equ}[e:expand-rescaled-Z]
 \CZ_{\e, \beta}(t,x) = 1 +\sum_{k=1}^\infty \CZ_{\e,\beta}^{(k)}(t,x)\;.
 \end{equ}
  Then for any $\beta\in\R$,
\begin{equ} [e:Zbound]
\limsup_{\e\downarrow 0} \E[\CZ_{\e,\beta} (t,x)] <\infty.
\end{equ}
Furthermore,
\begin{equ} [e:Ztail]
\lim_{m\to\infty} \limsup_{\e\downarrow 0} \E\left[\Big| \CZ_{\e, \beta}(t,x)-1-\sum_{k=1}^m \CZ_{\e, \beta}^{(k)}(t,x) \Big| \right] = 0.
\end{equ}
\EL

\BL[Convergence of finite order chaos]\label{L:termconv}
Let $\CZ_{\e,\beta}(t,x)$ and $(\CZ_{\e,\beta}^{(k)}(t,x))_{k\in\N}$ be as above. Then for any $l\in\N$, and $(t_i, x_i) \in [0,1]\times \R$, $k_i\in\N$, for $1\leq i\leq l$, $(\CZ_{\e,\beta}^{(k_i)}(t_i, x_i))_{1\leq i\leq l}$ converge in joint distribution to $(u^{(k_i)}(1-t_i, x_i))_{1\leq i\leq l}$, where $u^{(k)}(1-t,x)$ is the $k$-th order term in the expansion for $u(1-t,x)$ in \eqref{e:uexp}.
\EL

The proof of Lemma \ref{L:Zbound} will be based on Poisson random walk calculations. The proof of Lemma \ref{L:termconv} will be based on functional analytic arguments,  where terms of the series are treated as ($\epsilon$-dependent) continuous functionals of the underlying noise $\e^{-\frac12}\tilde \xi(\e^{-2}\cdot, \e^{-1}\cdot)$ in weighted Besov-H\"older spaces. To control the convergence of these functionals, we will rely on local limit theorem type of estimates for the random walk transition kernel and its space-time gradients, which are of independent interest.

\medskip

\begin{proof}[Proof of Theorem \ref{T:main}]
To show that $(Z^\xi_{\e^{-2}, \e^{3/2}\beta} (\e^{-2}t_i, \e^{-1}x_i))_{1\leq i\leq l}$ converge in joint distribution to $(u(1-t_i, x_i))_{1\leq i\leq l}$, first note that by Proposition \ref{P:uexp} and \eqref{e:Ztail}, we can truncate the series for $\CZ_{\e,\beta}(t_i,x_i)$ and $u(1-t_i, x_i)$ up to the $m$-th order with a small error in $L^1$. The convergence in joint distribution of the truncated series then follows from Lemma~\ref{L:termconv}. Lastly, sending $m\to\infty$ proves Theorem \ref{T:main}.
\end{proof}

The rest of the paper is organized as follows. In Section \ref{S:SPDE}, we prove Proposition \ref{P:uexp}. In Section \ref{S:Poisson}, we prove some basic properties for the Poisson field $\xi$ and its correlation functions, which are then used in Section \ref{S:fieldconv} to prove Proposition \ref{P:xilim}. Sections \ref{S:polymer} and \ref{S:conv} are then devoted to the proof of Lemmas \ref{L:Zbound}, \ref{L:termconv}, respectively. In Appendix \ref{S:Besov}, we recall the basics of weighted Besov-H\"older spaces, in which the noise and the family of partition functions take their values. Lastly, in  Appendix \ref{sec:HK}, we prove some heat kernel and gradient estimates for the random walk which are essential for the proof.

Unless otherwise specified, we will assume $d=1$ in the rest of the paper, and we will write $f(x) \lesssim g(x)$ if $f(x)\leq Cg(x)$ for
some $C\in (0,\infty)$ uniformly in $x$.

\section{The limiting SPDE}\label{S:SPDE}

In this section, we will prove Proposition \ref{P:uexp}, namely the SPDE \eqref{spde} has a mild Stratonovich solution  in dimension $d=1$, following the approach used in \cite{hunuso11, song17}. For general $d$, this SPDE can be formally written as
\begin{equation}\label{spde2}
\frac{\partial }{\partial t}u(t,x)=\frac12 \Delta  u(t,x)+ \beta \sqrt\lambda u \,  \Noise(t,x)
\end{equation}
with
$$
\E [\Noise (t,x) \Noise(s,y)]=p_{|t-s|}(x-y) := (2\pi |t-s|)^{-d/2}e^{-\frac{|x-y|^2}{2|t-s|}}.
$$
 Here $\Xi$ is a  generalised space-time Gaussian field, the rigorous meaning of which is given below.

Let $\mathcal H$ be the completion of the space $C_c^\infty ([0, \infty)\times \R^d)$ of smooth functions with compact support endowed with the inner product
\[\langle f, g\rangle_{\mathcal H}=\int_{\R_+^2\times \R^2} f(s,x)g(t,y) p_{|s-t|}(x-y) dxdydsdt,\]
$\|\cdot\|_\H$ be the induced norm,  and $\{\Noise(g), g\in \H\}$ be an isonormal Gaussian process with covariance
\[\E[\Noise(f)\Noise(g)]=\langle f, g\rangle_\H.\]
 Note that $\H$ contains distributions which are not  classical measurable functions. For instance, it is straightforward  to verify that $\1_{[0,T]}\delta(x)\in \H$,  where $
\delta(x)$ is the Dirac delta function; when $d=1$, noting that $p_{|t-s|(x-y)}\le (2\pi |t-s|)^{-1/2}$, by \cite{pt00} it is possible that $f(s,x) \in \H$ is a distribution in terms of $s$.

 For $(t,x)\in\R_+\times\R^d$, define $W(t,x):=\Xi\left(\1_{[0,t]}(\cdot)\prod_{k=1}^d\1_{[0,x_k]}(\cdot)\right)$, where $x=(x_1,\dots, x_d)$ and $\1_{[0, a]}:=-\1_{[-a,0]}$ for $a<0$. Then $(W(t,x))_{(t,x)\in\R_+\times \R^d}$ is a classical Gaussian field. Moreover, $\Xi(t,x)=\frac{\partial^{1+d}}{
\partial t \partial x_1\dots\partial x_d}W(t,x)$ where the partial derivative is in the distribution sense.

To define the Stratonovich integral, we introduce some notation.  For positive numbers $\varepsilon$ and $\e'$, define
\begin{equation} \label{appr}
\Noise^{\e,{\e'}}(t,x):=\int_{0}^{t}\int_{\mathbb{R}} \psi_{{\e'}}(t-s)p_{\e}(x-y)\Noise(s,y){\rm d}y{\rm d}s=\Noise(\phi_{t,x}^{\varepsilon, {\e'}}),
\end{equation}
where
\begin{equation}\label{e:psip}
p_{\e}(x)=(2\pi \varepsilon )^{-\frac{d}{2}}e^{-\frac{|x|^{2}}{%
2\e }}, ~  x\in \R^d;~ ~\psi _{\e'}(t)=\dfrac{1}{\e' }I_{[0,\e']}(t),~ t\in \R,
\end{equation}
and
 \begin{equation}\label{e:phi}
 \phi_{t,x}^{\varepsilon, {\e'}}(s,y)=\psi_{{\e'}
}(t-s)p_{\e}(x-y).
\end{equation}
 Note that $\phi_{t,x}^{\varepsilon, \varepsilon'}$ belongs to $\mathcal H$, and hence $\Noise^{\e ,{\e'}
}(t,x)$ exists in the classical sense and is an approximation of $\Noise(t,x)$.

We define  the Stratonovich integral following \cite[Definition 4.1]{hunuso11}, which was also used in \cite{hunu09,hhnt15,song17}.

\begin{definition}[Stratonovich integral]\label{defstra}
\label{def2} Suppose that  $v=\{v(t,x),t\geq 0,x\in
\mathbb{R}^d\}$ is a random field satisfying
\begin{equation*}
\int_{0}^{T}\int_{\mathbb{R}^d}|v(t,x)|dxdt<\infty, \, \text{ a.s.},
\end{equation*}
and that the limit in probability $\lim\limits_{\e,{\e'} \downarrow 0}\int_{0}^{T}\int_{\mathbb{R}^d}v(t,x)
\Noise^{\e ,{\e'} }(t,x)dxdt$  exists. Then we denote the limit by
\begin{equation*}
\int_{0}^{T}\int_{\mathbb{R}^d}v(t,x)
\Noise(t,x)dxdt:=\lim_{\e,{\e'} \downarrow 0}\int_{0}^{T}\int_{\mathbb{R}^d}v(t,x)
\Noise^{\e,{\e'} }(t,x)dxdt.
\end{equation*}
and call it {\it Stratonovich integral}.
\end{definition}

Let $\mathcal F_t$ be the $\sigma$-algebra generated by  $\{W(s,x), 0\le s\le t, x\in \R^d\}$, and  a random field $\{F(t,x),t\ge0,x\in \R^d\}$ is said to be adapted  if $\{F(t,x), t\ge0\}$ is adapted to the filtration $\{\mathcal F_t\}_{t\ge 0}$ for all $x\in\R^d$.

The mild Stratonovich solution of \eqref{spde} is defined as follows.
\begin{definition}[Mild solution] \label{defmildstr} An adapted random field $u=\{u(t,x), t\ge0, x\in \R^d\}$ is a mild solution to (\ref{spde}) with initial condition $u_0\in C_b(\R^d)$, if for all $t\ge0$ and $x\in \R^d$ the following integral equation holds  a.s.:
\begin{equation}\label{mildstra}
u(t,x)= u_0(x)+\beta \sqrt \lambda \int_0^t\int_{\R^d}p_{t-s}(x-y)u(s,y)\Noise(s,y)dy ds,
\end{equation}
where the stochastic integral on the right-hand side is in the Stratonovich sense.
\end{definition}

Let $B$ be a standard $d$-dimensional Brownian motion independent of $\Noise$, and $B_t^x:=B_t+x$.  Following the discrete notation, we use $\P$ and $\E$ to denote respectively the probability and expectation for the noise $\Noise$, and use $\bP$  and $\bE$ for the Brownian motion $B$.

Fix $x\in \mathbb{R}^{d}$ and $t>0$. Recalling the notation in \eqref{e:psip}, we define \begin{equation}
I_{t,x}^{\varepsilon ,{\e'} }=\ \int_{0}^{t}\int_{\mathbb{R}%
^{d}}\delta^{\e,{\e'} }(B_{t-r}^x-y)\Noise(r,y)dydr:=\Noise(\delta^{\e,{\e'} }(B_{t-\cdot}^x-\cdot))\,  \label{e3}
\end{equation}
and
\begin{equation}
\delta^{\e,{\e'} }(B_{t-r}^x-y)=\int_{0}^{t}\psi _{{\e'}
}(t-s-r)p_{\varepsilon }(B_{s}^{x}-y)ds,  \label{e2}
\end{equation}
where $I_{t,x}^{\e, {\e'}}$ is a Wiener integral (conditional on $B$), of which the well-definedness will be justified in the following result.  Here,  by ``Wiener integral''  we mean $
\Xi(\phi)$ for $\phi\in\mathcal H$.  We would like to point out that $\delta^{\e,{\e'} }(B_{t-r}^x-y)$ is in fact a function of $(t-r, y, (B_s^x)_{s\in[0, t-r]})$, rather than a function of $B_{t-r}^x-y$. We use this notation because
formally $\delta^{\e,{\e'} }(B_{t-r}^x-y)$ converges to $\delta(B_{t-r}^x-y)$ as $(\varepsilon,\varepsilon')$ goes to zero. 
\begin{proposition}
\label{T:appr}
If $d=1$, then for all $
\e, {\e'} >0$, the term $\delta^{\e,{\e'} }(B_{t-\cdot}^x-\cdot)$ defined in
(\ref{e2}) belongs to $\mathcal{H}$ a.s.\ and the family of random variables $%
I_{t,x}^{\e,{\e'} }$ defined in (\ref{e3}) converges in $L^{p}$ for all $p\ge 2$ to a
limit denoted by
\begin{equation}\label{e:def-Itx}
I_{t,x}=\int_{0}^{t}\int_{\mathbb{R}}\delta (B_{t-r}^{x}-y)\Noise(r,y)dydr:= \Noise(\delta (B_{t-\cdot}^{x}-\cdot))\,,
\end{equation}
 where $\delta(B_{t-\cdot}^x-\cdot)$ is an $\H$-valued random variable given by the $L^2$-limit of $\delta^{\e,\e'}(B_{t-\cdot}^x-\cdot)$.
Conditional on $B$, $I_{t,x}$ is a Gaussian random variable with mean $0$
and variance%
\begin{equation}
\mathrm{Var}[I_{t,x}|B] =\int_{0}^{t}\int_{0}^{t}p_{|r-s|}(B_r-B_s)dr ds\,.
\label{e.2.14}
\end{equation}
\end{proposition}

\begin{proof}
For $\varepsilon $, $\varepsilon ^{\prime }$, ${\sigma} $ and $\sigma'>0$,
\begin{align}
&\left\langle \delta^{\e,{\e'} }(B_{t-\cdot}^x-\cdot),\delta^{\sigma ,{\sigma'} }(B_{t-\cdot}^x-\cdot)\right\rangle _{\mathcal{H}} \notag\\
=&\int_{[0,t]^{4}}\int_{\mathbb{R}^{2}}p_{\e
}(B_{s}^{x}-y)p_{\sigma}(B_{r}^{x}-z)  \psi _{\e'}(t-s-u)\psi _{{\sigma'}}(t-r-v)\notag\\
&\qquad \qquad \qquad \times p_{|u-v|}(y-z)dydzdudvdsdr.\label{e4}
\end{align}
Noting that $p_{|u-v|}(y-z)\le (2\pi)^{-1/2}|u-v|^{-1/2},$ we have
\begin{align}
\int_{[0,t]^{2}}& \int_{\mathbb{R}^{2}}  p_{\e}(B_{s}^{x}-y)p_{\sigma}(B_{r}^{x}-z)  
\cdot \psi _{\e'}(t-s-u)\psi _{\sigma'}(t-r-v)p_{|u-v|}(y-z)\,dy\,dz\,du\,dv
\notag \\
&\leq  C \int_{[0,t]^2} \psi _{{\e'} }(t-s-u)\psi _{\sigma'}(t-r-v) |u-v|^{-1/2} dudv 
\quad \le C |s-r|^{-1/2}\label{e5}
\end{align}%
for some constant $C>0$, where the last step follows from \cite[Lemma A.3]{hunuso11}.
As a consequence, letting $\e=\e'$ and $\sigma=\sigma'$, by \eqref{e4} and (\ref{e5}), we have
\begin{equation}\label{e:ubd}
\sup_{\e,{\e'}>0}\left\|\delta^{\e,{\e'} }(B_{t-\cdot}^x-\cdot)\right\| _{\mathcal{H}}^{2} \le Ct^{3/2}<\infty,
\end{equation}
where $\|\cdot\|_\H$ is the norm induced by the inner product $\langle \cdot, \cdot \rangle_\H$,
which implies that almost surely $\delta^{\e,{\e'} }(B_{t-\cdot}^x-\cdot)$ belongs to
the Hilbert space $\mathcal{H}$ for all $\varepsilon $ and ${\e'} >0$. Therefore,
the random variables $I_{t,x}^{\varepsilon ,{\e'} }$ are well defined and we have
\begin{align}\label{e:2.13}
\bE\,\E[I_{t,x}^{\e ,{\e'} }I_{t,x}^{\sigma,\sigma'}]=&\bE\left\langle \delta^{\e,{\e'} }(B_{t-\cdot}^x-\cdot),\delta^{\sigma,{\sigma'} }(B_{t-\cdot}^x-\cdot)\right\rangle _{\mathcal{%
H}}.
\end{align}%
For any $s\neq r$ and $B_{s}\neq B_{r}$, as $\varepsilon $, $\varepsilon
^{\prime }$, $\sigma$ and $\sigma'$ tend to zero, the left-hand
side of the inequality \ (\ref{e5}) converges to $p_{|s-r|}(B_s-B_r)$. Therefore, by
dominated convergence theorem we obtain that, as $\varepsilon $, $\varepsilon ^{\prime }$, $\sigma $ and $\sigma'$ tend to zero,
  $\bE \E[I_{t,x}^{\e
,{\e'} }I_{t,x}^{\sigma, \sigma'}]$ converges to $\bE\int_{[0,t]^2}p_{|s-r|}(B_s-B_r) dsdr$ which is finite by Lemma \ref{exp-int} below. As a consequence,
\begin{equation}\label{e:2.14}
\bE\,\E \Big[\Big( I_{t,x}^{{\varepsilon },{{\e'} }}-I_{t,x}^{\sigma,\sigma' }\Big) ^{2}\Big]
=\bE\,\E \Big[\Big( I_{t,x}^{{\varepsilon },{{\e'} }\ }\Big) ^{2}\Big]
-2\bE\,\E \Big[\Big( I_{t,x}^{{\varepsilon },{{\e'} }}I_{t,x}^{\sigma,\sigma' }\Big)\Big]
+\bE\,\E \Big[\Big(
I_{t,x}^{\sigma,\sigma' }\Big)
^{2}\Big]\rightarrow 0\,.
\end{equation}%

 It also follows from \eqref{e:2.13} and \eqref{e:2.14} that,  for all  sequences $\e_n$ and $\e_n'$ converging to zero, $\delta^{\e_n,\e_n'}(B_{t-\cdot}^x-\cdot)$ is a Cauchy sequence in $L^2(\Omega,\mathbf P; \H)$ with $(\Omega,\mathcal G, \mathbf P)$ being the probability space generated by $B$. We denote the limit in $L^2(\Omega, \mathbf P; \H)$ of $\delta^{\e,\e'}(B_{t-\cdot}^x-\cdot)$ by $\delta(B_{t-\cdot}^x-\cdot)$.

Now, noting that $I_{t,x}^{\e_n,\e_n'}$ is a centered Gaussian random variable conditional on $B$  and hence  $\E[(I_{t,x}^{\e_n,\e_n'})^{2n}] =(2n-1)!! \,\left(\E[(I_{t,x}^{\e_n,\e_n'})^2]\right)^n$ for  $n\in \mathbb N$,  we have  that \ $I_{t,x}^{\e_n,\e'_n}$ is a Cauchy
sequence in $L^{p}$ for $p\ge 2$ for all  sequences $\e_n$ and $\e_n'$ converging to zero. As a consequence, $I_{t,x}^{\e_n,\e_n'}$
converges in  $L^{p}$ to a limit denoted by $I_{t,x}$, which does not depend
on the choice of the sequences $\e_n$ and $\e_n'$. \
Finally, \ (\ref{e.2.14}) can be shown by a similar argument.
\end{proof}

\begin{lemma} \label{exp-int}
The following bound holds  if and only if $d=1$:
\[\bE  \left[\int_0^t\int_0^t p_{|r-s|}(B_r-B_s)drds \right] <\infty \;.
\]
Furthermore, when $d=1$, one has
\begin{equ}[e:exp-int]
\bE\left[\exp\left(\lambda \int_0^t \int_0^t p_{|r-s|}(B_r-B_s)drds \right)\right] <\infty \text{ for all } \lambda >0 \;.
\end{equ}
\end{lemma}

\begin{proof}
First note that
\begin{equs}
\bE & \left[\int_0^t\int_0^t p_{|r-s|}(B_r-B_s)drds\right]
=2 \int_0^t\int_0^r  \int_{\R^{d}}   p_{r-s}(z)^2 dzdrds
\\
&=2\int_0^t\int_0^r p_{2(r-s)}(0) dsdr
= 2\int_0^t\int_0^r (4\pi(r-s))^{-d/2}dsdr,
\end{equs}
which is finite if and only if $d=1$. When $d=1$,
one has the deterministic bound
$$\int_0^t\int_0^t p_{|r-s|}(B_r-B_s)drds\le \int_0^t\int_0^t (2\pi |r-s|)^{-1/2}drds<\infty,$$ and hence \eqref{e:exp-int} follows.
\end{proof}

From now on, we only consider the case $d=1$.  We shall show that a mild solution to \eqref{spde} can be obtained by the following Picard iteration. Let $u_0(t,x)\equiv 1$ and
\begin{equation}\label{e:u-n}
u_{n}(t,x)= 1 +\beta\sqrt \lambda\int_0^t \int_{\mathbb R} p_{t-s}(x-y) u_{n-1}(s,y)\Noise(s,y)dy ds, ~n=1,2,\dots.
\end{equation}
Then we have  $$u_n(t,x)-u_{n-1}(t,x)= u^{(n)}(t,x), ~n=1, 2, \dots, $$
where $u^{(n)}$ is as in \eqref{e:uexp}, namely
\begin{equ}[e:def-Jn]
u^{(n)}(t,x)
=(\beta\sqrt\lambda)^n\int_{[0,t]^n_< \times \R^n}
\prod_{k=1}^{n}
p_{s_{k+1}-s_{k}}(y_{k+1}-y_{k})
 \prod_{k=1}^n\Noise(s_k,y_k)dy_kds_k
\end{equ}
where $[0,t]^n_<= \{ (s_1,\cdots,s_n)\in \R^n\;:\; 0< s_{1}<\cdots<s_{n}<t \}$ with the convention $ y_{n+1}=x, s_{n+1}=t$.
We remark that the multiple integral on the RHS of \eqref{e:def-Jn} is classically meaningful, which we explain in Section~\ref{S:conv}
(Remark~\ref{rem:multi-int}) when we discuss  the convergence of the discrete multiple integrals to them.

The above Picard iteration suggests that formally
\[
u(t,x)=1+\sum_{n=1}^\infty u^{(n)}(t,x)
\] is a mild solution to \eqref{spde}, if it is well-defined.

 Indeed, in the following we shall show that $u_n(t,x)=1+\sum_{k=1}^n u^{(k)}(t,x)$ converges in $L^1$, as $n\to \infty$, to the Feynman-Kac type expression
\begin{equ}[e:FC-u]
 \bE\left[\exp\left(\beta\sqrt\lambda\int_0^t \int_{\mathbb R} {\delta}(B^x_{t-s}-y)\Noise(s,y)dyds\right)\right],
 \end{equ}
which will be proven to be a mild Stratonovich solution to \eqref{spde}.

\begin{lemma}\label{lem:2.5}
\eqref{e:FC-u} is $L^1$ integrable with respect to $\Noise$.
\end{lemma}

\begin{proof}
By Proposition~\ref{T:appr},
\begin{align*}
&\E \bE \left[\exp\left(\beta\sqrt\lambda\int_0^t \int_{\mathbb R} {\delta}(B_{t-s}^x-y)\Noise(s,y) dyds\right) \right]\\
&=\bE \left[\exp\left(\frac{\beta^2\lambda}{2} \int_0^t \int_0^t p_{|r-s|}(B_r-B_s)drds \right) \right],
\end{align*}
and combining with Lemma \ref{exp-int}, we know that the Feynman-Kac type expression \eqref{e:FC-u}
 is well-defined (i.e. integrable with respect to $\Noise$) when $d=1$.
 \end{proof}

Now we show that the series generated by the Picard iteration in \eqref{e:def-Jn} and a Taylor expansion of the Feynman-Kac expression \eqref{e:FC-u} coincide term by term:
\begin{lemma}\label{lem-Jn}
For each $n\in \Z_+$ one has
\begin{equ}[e:series=TaylorFC]
u^{(n)}(t,x)=\frac{(\beta\sqrt \lambda)^n}{n!}\bE \left[(I_{t,x})^n\right]\;.
\end{equ}
where $I_{t,x}$ is defined in \eqref{e:def-Itx}.
\end{lemma}

\begin{proof}
 By Proposition \ref{T:appr}, we have that $I_{t,x}= \Noise(\delta (B_{t-\cdot}^{x}-\cdot))$ is $L^p$-integrable for all $p\ge 2$. Recall that  $\delta (B_{t-\cdot}^{x}-\cdot)$ belongs to $\H$ a.s. and is given by the limit of $\delta^{\e, \e'} (B_{t-\cdot}^{x}-\cdot)$. Hence, 
$$
\|\bE [\delta (B_{t-\cdot}^{x}-\cdot)-\delta^{\e,\e'} (B_{t-\cdot}^{x}-\cdot)] \|_{\H}^2\le \bE \| \delta (B_{t-\cdot}^{x}-\cdot)-\delta^{\e,\e'} (B_{t-\cdot}^{x}-\cdot) \|_{\H}^2 \to 0
$$ 
as $(\e, \e')$ goes to zero. Here we used the following Jensen's inequality on $\H$: denoting $f:=\delta (B_{t-\cdot}^{x}-\cdot)-\delta^{\e,\e'} (B_{t-\cdot}^{x}-\cdot)$, 
 \[\|\mathbf Ef \|_{\mathcal H}^2 = \sum_{k=1}^\infty |  \langle \mathbf E f, e_k\rangle_{\mathcal H}|^2 \le   \sum_{k=1}^\infty\mathbf E\left[|\langle f, e_k\rangle_{\mathcal H}|^2\right]=\mathbf E[\|f\|_{\mathcal H}^2],\] 
where $\{e_k, k\in \mathbb N\}$ is an orthonormal basis of $\H$. Therefore, 
$$
\bE [\delta (B_{t-\cdot}^{x}-\cdot)]=\lim_{(\e,\e')\to0}\bE[\delta^{\e,\e'} (B_{t-\cdot}^{x}-\cdot)],
$$ 
where the limit is taken in the space $\H$. Following the proof of  Proposition \ref{T:appr}, one can prove that $(\delta^{\e,\e'} (B_{t-\cdot}^{x}-\cdot))^{\otimes n}$ converges to $(\delta (B_{t-\cdot}^{x}-\cdot))^{\otimes n}$ in $L^2(\Omega;\H^{\otimes n})$, and similarly,
\begin{equation}\label{e:2.19}
 \bE[(\delta (B_{t-\cdot}^{x}-\cdot))^{\otimes n}]=\lim_{(\e,\e')\to0} \bE[(\delta^{\e,\e'} (B_{t-\cdot}^{x}-\cdot))^{\otimes n}]
 \end{equation}
 in the space $\H^{\otimes n}$.

Now, we have
\begin{align*}\label{fkhn}
&\frac1{n!}\bE\left[ (I_{t,x})^n\right]\notag\\
&= \frac{1}{n!}\bE \left[\int_{[0,t]^n}\int_{\R^{n}}{{\delta}}(B_{t-s_1}^x-y_1) \cdots {{\delta}}
(B_{t-s_n}^x-y_n)\Noise(s_1,y_1)ds_1dy_1\dots\Noise(s_n, y_n)ds_ndy_n\right]\notag
\\
&= \frac{1}{n!}\int_{[0,t]^n}\int_{\R^{n}}\bE \Big[{{\delta}}(B_{t-s_1}^x-y_1) \cdots {{\delta}}
(B_{t-s_n}^x-y_n)\Big]\Noise(s_1,y_1)dy_1ds_1\dots\Noise(s_n, y_n)dy_nds_n\notag
\\
&=\int_{[0,t]_<^n}\int_{\R^{n}}\lim_{(\e, \e')\to 0}\bE \left[{{\delta}^{\e,\e'}}(B_{t-s_1}^x-y_1) \cdots {{\delta}^{\e,\e'}}
(B_{t-s_n}^x-y_n) \right] \Noise(s_1,y_1)dy_1ds_1\dots\Noise(s_n, y_n)dy_nds_n \notag \\
&= \int_{[0,t]_<^n}\int_{\R^{n}} p_{s_{2}-s_{1}}(y_{2}-y_{1})\cdots p_{t-s_{n}}(x-y_{n})
\Noise(s_1,y_1)dy_1ds_1\dots\Noise(s_n, y_n)dy_nds_n, ~~a.s.,
\end{align*}
where we have interchanged $\bE[\cdot]$ with the stochastic integration by stochastic Fubini theorem for Stratonovich integrals, which follows from stochastic Fubini Theorem for Skorohod integrals (\cite{DaPrato14,KRT07}) and Hu-Meyer formula (\cite{HM88}), and the last second step follows from \eqref{e:2.19}. Comparing with
\eqref{e:def-Jn}, we obtain the desired result.
\end{proof}

\begin{proof}[Proof of Proposition~\ref{P:uexp}]
Firstly, we note that
$u_n(t,x)=1+\sum_{k=1}^n u^{(k)}(t,x)$  converges to $u(t,x)$ as defined in \eqref{e:uexp} in $L^1$ as $n\to\infty$, since by Lemma \ref{lem-Jn} and Lemma \ref{lem:2.5},
\begin{align*}
&\E\left[\sum_{n=0}^\infty |u^{(n)}(t,x)|\right] \le \sum_{n=0}^\infty\frac{(\beta\sqrt\lambda)^n}{n!}\E\bE \left[ |I_{t,x}|^n\right]\\
 & =\E\bE\left[e^{\beta\sqrt\lambda |I_{t,x}|}\right]  \le \E\bE\left[ e^{\beta\sqrt\lambda I_{t,x}}+e^{-\beta\sqrt\lambda I_{t,x}}\right]<\infty.
\end{align*}
Thus \eqref{e:uexp} equals the Feynman-Kac representation \eqref{e:FC-u}.

To conclude the proof,  we shall verify that \eqref{e:FC-u} is a mild Stratonovich solution to \eqref{spde},  following the approach used in \cite{song17}.

Consider the following approximation of \eqref{spde}:
\begin{equation}\label{appspde}
\begin{cases}
\dfrac{\partial}{\partial t}u^{\varepsilon, {\e'}}(t,x)=\dfrac12 {\Delta} u^{\varepsilon, {\e'}}(t,x)+\beta\sqrt\lambda u^{\varepsilon, {\e'}}(t,x)\Noise^{\varepsilon, {\e'}}(t,x),\\
u^{\varepsilon,{\e'}}(0,x)\equiv1,
\end{cases}
\end{equation}
where $\Noise^{\e, {\e'}}(t,x)$ is defined in \eqref{appr}.

By the classical Feynman-Kac formula, we have
\begin{align*}
u^{\varepsilon, {\e'}}(t,x)=\bE\left[\exp\left(\beta\sqrt\lambda\int_0^t \Noise^{\varepsilon, {\e'}}(r,B_{t-r}^x)dr \right)\right]=\bE\left[\exp\left(\beta\sqrt\lambda I_{t,x}^{\e,\e'}\right)\right]
\end{align*}
where $I_{t,x}^{\e, {\e'}}=\Noise(\delta^{\e,\e'}(B_{t-\cdot}^x-\cdot))$ is defined in \eqref{e3} and the last equality follows from  Fubini's theorem. It is a mild solution to (\ref{appspde}), i.e.,
\begin{equation}\label{mildapp}
u^{\varepsilon, {\e'}}(t,x)=1+\beta\sqrt\lambda\int_0^t\int_{\R}p_{t-s}(x-y)u^{\varepsilon, {\e'}}(s,y)\Noise^{\varepsilon, {\e'}}(s,y)dyds.
\end{equation}

To prove the result, it suffices to show that as $(\varepsilon,{\e'})$ tends to zero, both sides of  (\ref{mildapp}) converge in probability to those of (\ref{mildstra})   respectively, with $u(t,x)$ given by (\ref{e:FC-u}).

First, we show some convergence results for $u^{\e,\e'}(t,x)$.  By Proposition~\ref{T:appr}, as $(\varepsilon, {\e'})\to 0$, $I_{t,x}^{\e,{\e'}}=\Noise(\delta^{\e,\e'}(B_{t-\cdot}^x-\cdot))$ converges to  $\Noise(\delta(B_{t-\cdot}^x-\cdot))$ in probability, and
$$\E[|u^{\varepsilon, {\e'}}(t,x)|^p]\le \E\,\bE\left[\exp\left(p\beta\sqrt\lambda I_{t,x}^{\e,\e'})\right)\right]=\bE\left[\exp\left(\frac{1}{2}p^2\beta^2\lambda \|\delta^{\e,\e'}(B_{t-\cdot}^x-\cdot)\|_{\H}^2\right)\right],$$
 which, by \eqref{e:ubd}, is bounded by $Ce^{Ct^{3/2}}$ for some constant $C$ independent of $(\e,{\e'})$.  This yields the $L^p$-convergence of $u^{\varepsilon, {\e'}}(t,x)$ to $u(t,x)$ for all $p>1.$

Furthermore, we show that $u^{\e,\e'}(t,x)$ converges to $u(t,x)$ in $\mathbb D^{1,2}$, i.e.,
  \begin{equation}\label{e:D12}
  \lim_{\e,\e'\downarrow 0} \E[|u^{\e,\e
'}(t, x)-u(t,x)|^2]+\E[\|Du^{\e,\e
'}(t, x)-Du(t,x)\|_\H^2]=0,
\end{equation}
 where $D$ is the Malliavin derivative (see \cite[Section 1.2]{nualart06} for details). To prove this, noting that the Malliavin derivative is closable (\cite[Proposition 1.2.1]{nualart06}), it suffices to prove
\begin{equation}\label{e:Du}
\lim_{\e,\e',\sigma,\sigma'\downarrow 0}\E\|Du^{\e,\e'}(t,x)-Du^{\sigma,\sigma'}(t,x)\|_{\H}^2=0.
\end{equation}
  By the chain rule (\cite[Proposition 1.2.3]{nualart06}),
\begin{align}\label{e:Due}
Du^{\e,\e'}(t,x)&=\beta\sqrt\lambda\bE\left[\exp\left(\beta\sqrt\lambda I_{t,x}^{\e,\e'}\right)DI_{t,x}^{\e,\e'}\right]\notag\\
&=\beta\sqrt\lambda\bE\left[\exp\left(\beta\sqrt\lambda I_{t,x}^{\e,\e'}\right)\delta^{\e,\e'}(B_{t-\cdot}^x-\cdot)\right].
\end{align}
Thus, letting $B^1$ and $B^2$ be two independent 1-dimensional Brownian motions, and $I_{t,x}^{\e,\e'}(B^i)$ be defined as in \eqref{e3} with $B$ replaced by $B^i$ for $i=1,2$, we have
\begin{align*}
&\E\left\langle Du^{\e,\e'}(t,x),Du^{\sigma,\sigma'}(t,x)\right\rangle_{\H}^2\\
=&(\beta^2\lambda)\E\,\bE\left[\exp\left(\beta\sqrt\lambda\left(I_{t,x}^{\e,\e'}(B^1)+I_{t,x}^{\sigma,\sigma'}(B^2)\right)\right) \left\langle \delta^{\e,\e'}(B_{t-\cdot}^1+x-\cdot), \delta^{\sigma,\sigma'}(B_{t-\cdot}^2+x-\cdot)\right\rangle_\H\right]\\
=&(\beta^2\lambda)\E\,\bE\bigg[\exp\left(\beta\sqrt\lambda\left(I_{t,x}^{\e,\e'}(B^1)+I_{t,x}^{\sigma,\sigma'}(B^2)\right)\right) \\
&\qquad\quad  \times\bigg( \int_{[0,t]^{4}}\int_{\mathbb{R}^{2}}p_{\e}(B_{s}^{1}+x-y)p_{\sigma}(B_{r}^{2}+ x-z)\\
&\qquad\qquad \qquad \times  \psi _{\e'}(t-s-u)\psi _{{\sigma'}}(t-r-v)p_{|u-v|}(y-z)\,dy\,dz\,du\,dv\,ds\,dr\bigg)
 \bigg]\,.
\end{align*}
Using similar analysis as in the proof of Proposition~\ref{T:appr}, we can show that
\begin{align*}
&\lim_{\e,\e',\sigma,\sigma'\downarrow0}\E\left\langle Du^{\e,\e'}(t,x),Du^{\sigma,\sigma'}(t,x)\right\rangle_{\H}^2\\
=& (\beta^2\lambda)\E\,\bE\bigg[\exp\bigg(\beta\sqrt\lambda\left(I_{t,x}(B^1)+I_{t,x}(B^2)\right)\bigg) \int_0^t\int_0^t p_{|s-r|}(B_s^1-B_r^2)dsdr\bigg]\\
=&(\beta^2\lambda)\bE\bigg[ \exp\bigg(\beta^2\lambda \sum_{i,j=1}^2\int_0^t\int_0^t p_{|s-r|}(B_s^i-B_r^j)dsdr\bigg)\int_0^t\int_0^t p_{|s-r|}(B_s^1-B_r^2)dsdr\bigg],
\end{align*}
which is finite by Lemma \ref{exp-int}. This proves \eqref{e:Du}, and hence \eqref{e:D12} holds.

We are now ready to prove the convergence to the right-hand side of \eqref{mildstra} of the right-hand side of (\ref{mildapp}). Noting that the Stratonovich integral on the right hand-side of \eqref{mildstra} is defined by  Definition \ref{defstra},  it suffices to show that, as $(\e,\e')\to 0$,
\begin{equation}\label{e:Mee'}
M^{\varepsilon, {\e'}}:=\int_0^t\int_{\R}p_{t-s}(x-y)(u^{\varepsilon, {\e'}}(s,y)-u(s,y))\Noise^{\varepsilon, {\e'}}(s,y)dyds\to 0 \text{ in }  L^2.
\end{equation}
 Denoting $v^{\varepsilon,{\e'}}(s,y):=u^{\varepsilon, {\e'}}(s,y)-u(s,y)$, recalling that $\Noise^{\varepsilon, {\e'}}(s,y)=\Noise(\phi_{s,y}^{\varepsilon,{\e'}})=\delta(\phi_{s,y}^{\varepsilon,{\e'}})$, where $\delta$ is the divergence operator (also known as Skorohod integral, see \cite[Section 1.3]{nualart06}), by \eqref{appr},  and applying the integration by parts  formula (\cite[Proposition 1.3.3]{nualart06})
 $$
 \delta(Fu)=F\delta(u)-\langle DF, u\rangle_\H
$$
to $v^{\e,\e'}(s,y)\Noise^{\e,\e'}(s,y)$,  we have
\begin{align*}
M^{\e, \e'}&=\int_0^t\int_\R p_{t-s}(x-y) \left[\delta(v^{\e,\e'}(s,y) \phi_{s,y}^{\e, \e'})+\langle Dv^{\e,\e'}(s,y),  \phi_{s,y}^{\e, \e'}\rangle_\H \right] dyds.
\end{align*}
We take the integral form $\int_0^T\int_\R u(s,y) \diamond \Noise(s,y)dyds$ to denote the Skorohod integral $\delta(u)$. Thus,
\begin{align}
M^{\e, \e'}=&\int_0^t\int_\R p_{t-s}(x-y) \left[\delta(v^{\e,\e'}(s,y) \phi_{s,y}^{\e, \e'})+\langle Dv^{\e,\e'}(s,y),  \phi_{s,y}^{\e, \e'}\rangle_\H \right] dyds \nonumber\\
=&\int_0^t\int_\R ~\int_0^t\int_\R p_{t-s}(x-y) v^{\e,\e'}(s,y) \phi_{s,y}^{\e, \e'} (r, z) dyds\diamond \Noise(r,z)dzdr \nonumber \\
&\qquad +\int_0^t\int_\R p_{t-s}(x-y)\langle Dv^{\e,\e'}(s,y),  \phi_{s,y}^{\e, \e'}\rangle_\H dyds \nonumber\\
=:&M_1^{\e,\e'}+M_2^{\e,\e'}. \label{M2}
\end{align}
For $M_1^{\e,\e'}$, by the estimate for the Skorohod integral $\E[|\delta(u)|^2]\le \E[\|u\|_\H^2]+\E[\|Du\|^2_{\H\otimes \H}]$ (\cite[Proposition 1.3.1]{nualart06}), we have
\begin{align}\label{e:M1}
\E[|M_1^{\e,\e'}|^2]\le \E [\|\Phi_{t,x}^{\e,\e'}\|_\H^2]+\E[\|D\Phi_{t,x}^{\e,\e'}\|^2_{\H\otimes\H}],
\end{align}
where
\[\Phi^{\e,\e'}_{t,x}(r,z)=\int_0^t\int_\R p_{t-s}(x-y) v^{\e,\e'}(s,y) \phi_{s,y}^{\e, \e'} (r, z) dyds.\]
For the first term on the right hand side of \eqref{e:M1},
\begin{equation}\label{M1bd}
\begin{aligned}
\E [\|\Phi_{t,x}^{\e,\e'}\|_\H^2]=& \int_0^t\int_{\R}\int_0^t\int_\R p_{t-s_1}(x-y_1)p_{t-s_2}(x-y_2)\\
&\quad \times  \E\left[v^{\e,\e'}(s_1, y_1)v^{\e,\e'}(s_2, y_2)\right]\left\langle \phi^{\e,\e'}_{s_1,y_1}, \phi^{\e,\e'}_{s_2,y_2}\right\rangle_\H ~dy_1ds_1dy_2ds_2
\end{aligned}
\end{equation}
Noting that  $\lim_{\varepsilon,{\e'}\downarrow0}\E[|v_{s,y}^{\varepsilon,{\e'}}|^2]=0$, $\E[|v_{s,y}^{\varepsilon,{\e'}}|^2]$ is  bounded by  $Ce^{Cs^{3/2}}$ for some constant $C$ independent of $(\e,\e')$, and that $\left\langle \phi^{\e,\e'}_{s_1,y_1}, \phi^{\e,\e'}_{s_2,y_2}\right\rangle_\H$ is uniformly bounded by $C|s_1-s_2|^{-\frac12}$ in $(\e,\e')$ by Lemma \ref{lem:phi} below,  we thus have $\lim_{\varepsilon,{\e'}\downarrow0}\E [\|\Phi_{t,x}^{\e,\e'}\|_\H^2]=0$ by the dominated convergence theorem.

Similarly, for the second term on the right hand side of \eqref{e:M1}, by \eqref{e:Du},
\begin{align*}
\E[\|D\Phi_{t,x}^{\e,\e'}\|^2_{\H\otimes\H}]=&\int_0^t\int_{\R}\int_0^t\int_\R p_{t-s_1}(x-y_1)p_{t-s_2}(x-y_2) \\
&\quad \times \E\left[\left\langle Dv^{\e,\e'}(s_1, y_1), Dv^{\e,\e'}(s_2, y_2)\right\rangle_\H\right]\left\langle \phi^{\e,\e'}_{s_1,y_1}, \phi^{\e,\e'}_{s_2,y_2}\right\rangle_\H ~dy_1ds_1dy_2ds_2
\end{align*}
also converges to $0$ as $(\e,\e')\to0$. Hence, we have
\begin{equation}\label{e:M10}
\lim_{\e,\e'\downarrow0} \E[|M_1^{\e,\e'}|^2]=0.
\end{equation}

Finally, we bound $M_2^{\e,\e'}$ in \eqref{M2}.
\begin{align}\label{e:M2}
M_2^{\e,\e'}=& \int_0^t\int_\R p_{t-s}(x-y)\langle Du^{\e,\e'}(s,y),  \phi_{s,y}^{\e, \e'}\rangle_\H dyds\notag\\
&\quad -\int_0^t\int_\R p_{t-s}(x-y)\langle Du(s,y),  \phi_{s,y}^{\e, \e'}\rangle_\H dyds.
\end{align}
We shall show that the two terms on the right-hand side of \eqref{e:M2} converges in $L^2$ to the same limit, as $(\e,\e')\to0$. For the first term, by \eqref{e:Due},
\begin{align}
&\int_0^t\int_\R p_{t-s}(x-y)\langle Du^{\e,\e'}(s,y),  \phi_{s,y}^{\e, \e'}\rangle_\H dyds\notag\\
=&\beta\sqrt\lambda\int_0^t\int_\R p_{t-s}(x-y)\bE \left[\exp\left(\beta\sqrt \lambda I_{s,y}^{\e,\e'}\right)\left\langle \delta^{\e,\e'}(B_{s-\cdot}^{y}-\cdot),  \phi_{s,y}^{\e, \e'}\right\rangle_\H\right] dyds.\label{e:M21}
\end{align}
For the inner product
\begin{align*}
&\left\langle \delta^{\e,\e'}(B_{s-\cdot}^{y}-\cdot),  \phi_{s,y}^{\e, \e'}\right\rangle_\H\\
=&\int_{[0,s]^3}\int_{\R^2} p_{|r_1-r_2|}(z_1-z_2) \psi_{\e'}(s-\tau-r_1)p_\e(B_\tau^y-z_1)\\
&\qquad \qquad \times\psi_{\e'}(s-r_2)p_\e(y-z_2) dz_1dz_2 d\tau dr_1dr_2,
\end{align*}
using the fact $p_{|r_1-r_2|}\le (2\pi|r_1-r_2|)^{-\frac12}$, integrating in space variables
$(z_1, z_2)$, and then applying \cite[Lemma A.3]{hunuso11},   we have that it is bounded by $C\int_0^s\tau^{-\frac12}d\tau<\infty$,  and hence as $(\e,\e')\to0$, it converges to $\int_0^s p_{\tau}(B_\tau)d\tau$ in $L^p$ for all $p\ge1$ by dominated convergence theorem.  Thus, as $(\e,\e')\to0$, noting that $u^{\e,\e'}(s,y)$ converges to $u(s,y)$ in $L^p$ for all $p\ge1$, we have that \eqref{e:M21} converges in $L^2$ to
\[\beta\sqrt\lambda\int_0^t\int_\R p_{t-s}(x-y)\bE \bigg[\exp\left(\beta\sqrt \lambda I_{s,y}\right)\int_0^s p_\tau(B_\tau)d\tau \bigg] dy ds.\]
In a similar way, we can show that the second term on the right-hand side \eqref{e:M2} converges in $L^2$ to the same limit as $(\e,\e')\to0.$

  Thus we have shown the convergence to 0 of \eqref{e:M2}. This together with \eqref{e:M10} yields  \eqref{e:Mee'}. The proof of Proposition \ref{P:uexp} is concluded.

\end{proof}

In the bound for \eqref{M1bd}, we used the following result.
 \begin{lemma}\label{lem:phi}
 \[\sup_{\e,\e'>0}\langle \phi_{s_1,y_1}^{\e,{\e'}}, \phi_{s_2,y_2}^{\sigma,\sigma'}  \rangle_\H \le C|s_1-s_2|^{-1/2},\]
 where $C$ is a constant independent of $(\e,\e',\sigma, \sigma', s_1,y_1, s_2,y_2)$.
 \end{lemma}
 \begin{proof} By the definition of $\phi^{\e, \e'}_{t,x}$ in \eqref{e:phi},
 \begin{align*}
\langle \phi_{s_1,y_1}^{\e,{\e'}}, \phi_{s_2,y_2}^{\sigma,\sigma'}  \rangle_\H &=\left\langle \psi_{\e'}(s_1-\cdot)p_\e(y_1-\cdot), \psi_{\delta'}(s_2-\cdot)p_\delta(y_2-\cdot)\right\rangle_\H\\
& =\int_0^{s_1} dr_1 \int_0^{s_2} dr_2 ~~\psi_{\e'}(s_1-r_1)\psi_{\delta'}(s_2-r_2)\\
 &\qquad \int_{\R^2}p_\e(y_1-z_1) p_\delta(y_2-z_2) p_{|r_1-r_2|}(z_1-z_2) dz_1dz_2\\
& \le  C|s_1-s_2|^{-1/2}.
 \end{align*}
where the last step follows from \cite[Lemma A.3]{hunuso11} and the fact $\sup_{z\in\R}p_{|r_1-r_2|}(z)\le (2\pi |r_1-r_2|)^{-1/2}$.
 \end{proof}

 \begin{remark}
Strictly speaking Proposition, \ref{P:uexp} provides a series solution, but does not address the uniqueness of the mild Stratonovich solution. Nevertheless, the convergence established in Theorem~\ref{T:main} identifies the unique limit of the partition functions as the ``natural'' solution of the SPDE. One may prove uniqueness of the mild Stratonovich solution following the approach used in \cite[Setion 5]{hhnt15}. We omit the details here since we are mainly interested in the scaling limit of the polymer partition functions in this article.
\end{remark}

\section{Correlations for the Poisson random walks} \label{S:Poisson}

To prove Proposition \ref{P:xilim}, we need to bound moments in order to establish tightness for the rescaled and recentered random field $\xi$ in a weighted Besov-H\"older space. It turns out that the model of Poisson random walks admits exact correlation calculations, based on which one can perform careful moment analysis for the occupation variables $\xi$. In this section we provide some formulas and bounds for these correlations, which will be useful in Section~\ref{S:fieldconv}.

Let $(\xi(t,x))_{t>0, x\in\Z^d}$ be the mean $\lambda$ Poisson field of random walks as defined  in \eqref{e:xi}, and let $\tilde \xi(\cdot, \cdot):= \xi(\cdot, \cdot)-\lambda$ be its centered version. In this section, we give bounds on the $k$-point correlation functions
of $\xi$.

Note that $\xi(\cdot, \cdot)$ can be defined from a {\em Poisson point process} $\Gamma$ on the space of simple random walk trajectories on $\Z^d$ with intensity measure
\begin{equ} [e:mu]
\mu(\cdot)= \lambda \sum_{x\in \Z^d}  \bP_{(0,x)}(S\in \cdot),
\end{equ}
where $\bP_{(0,x)}$ is the law of a simple symmetric random walk starting from $x$ at time $0$. Given the random set of trajectories $\Gamma$, $\xi(t,x)$ simply counts the number of paths in $\Gamma$ that visit the space-time point $(t,x)$.
Given $A$, a set of random walk trajectories on $\Z^d$, let
\[
\xi_A \eqdef
\sum_{y\in\Z^d} \sum_{i=1}^{\xi(0,y)} 1_{ (Y_t^{y,i})_{t\ge0}\in A}
\]
denote the number of paths in the Poisson point process $\xi$ that belong to A, and denote $\tilde \xi_A: = \xi_A-\E[\xi_A]$.
Clearly $ \xi_A$ and $\tilde \xi_A$ have the additive property:
if $A\cap B=\emptyset$ then $\xi_{A\cup B}=\xi_A + \xi_B $ and the same holds for $\tilde \xi_A$.

Let $T>0$ and $j\in \N$ be  fixed. Given a collection of space-time points $(t_i, x_i)_{1\le i\le j}$ with $0\le t_1\le t_2\le \cdots\le t_j\le T$ and $x_1,\cdots, x_j\in \Z^d$, for each $1\leq i\leq j$, denote $\tilde\xi_i \eqdef \tilde \xi(t_i, x_i)$.  Denote by $A_i$ the set of trajectories that pass through the point $(t_i, x_i)$, then $\tilde \xi_i=\tilde \xi_{A_i}$.
For two sets $A$ and $B$, we denote $A\cap B$ by $AB$, and denote $A\cup B$ by $A+B$ if $A$ and $B$ are disjoint. As an illustration of the main idea useful for the general case,
 we compute the two-point correlation $\E[\tilde \xi_1\tilde \xi_2]$:
\begin{equs}[e:2point-ex]
\E & [\tilde \xi_1\tilde \xi_2]=\E[\tilde\xi_{A_1}\tilde\xi_{A_2}]
=\E[\tilde\xi_{A_1A_2+ A_1\backslash A_2}\,\tilde\xi_{A_1A_2+A_2\backslash A_1}]\\
&=\E\left[(\tilde\xi_{A_1A_2}+\tilde\xi_ {A_1\backslash A_2})(\tilde\xi_{A_1A_2}+\tilde\xi_{A_2\backslash A_1})\right]\\
&=\E[(\tilde \xi_{A_1A_2})^2] \;.
\end{equs}
Here, in the last step, we used the fact that $\tilde \xi_A$ and $\tilde\xi_B$ are independent if $A\cap B=\emptyset$,
which is why we cut the sets $A_1,A_2$ into non-intersecting subsets.

To compute the RHS of \eqref{e:2point-ex}, note that $\xi_{A_1A_2}$ is a Poisson random variable with mean
\begin{equ} [e:wk]
\lambda \sum_{x\in \Z^d} P_{t_1}(x_1-x) P_{t_2-t_1}(x_2-x_1) = \lambda P_{t_2-t_1}(x_2-x_1),
\end{equ}
where $P_t(x) := \bP_{(0,0)}(S_t=x)$ is the random walk transition kernel. Therefore
\begin{equ} [e:2pt]
\E [\tilde \xi_1\tilde \xi_2]= \E[(\tilde \xi_{A_1A_2})^2] = \lambda P_{t_2-t_1}(x_2-x_1).
\end{equ}

More generally, by the same calculation as in \eqref{e:wk}, we have:
\begin{lemma}\label{lem:xiAAAPoi}
With $0< t_1 < t_2 <\dots< t_j< T$, $x_i\in \mathbb Z^d$ and the sets $A_i$ defined above, the random variable
$\xi_{A_1A_2\dots A_j}$ has Poisson distribution with parameter
\[
\lambda \prod_{k=1}^{j-1}
P_{t_{k+1}-t_k}(x_{k+1}-x_k) .
\]
\end{lemma}

\begin{proof}
$\xi_{A_1A_2\dots A_j}$ clearly has Poisson distribution, with mean
\begin{equ}
\lambda \sum_{x\in \Z^d} P_{t_1}(x_1-x) P_{t_2-t_1}(x_2-x_1) \cdots P_{t_j-t_{j-1}}(x_j-x_{j-1}).
\end{equ}
 The claimed result follows from $\sum_{x\in \Z^d} P_{t_1}(x_1-x)=1$.
\end{proof}

With Lemma~\ref{lem:xiAAAPoi}, we can compute $\E\big[\big(\tilde\xi_{A_1A_2\dots A_j}\big)^k\big]$ by the formulas for the moments of centered Poisson random variables. For instance,
$$
\E[\tilde \xi_1\tilde \xi_2\tilde\xi_3]=\E[(\tilde\xi_{A_1A_2A_3})^3]=\lambda P_{t_2-t_1}(x_2-x_1)P_{t_3-t_2}(x_3-x_2).
$$

The next lemma provides a useful expression for correlations  of the centered Poisson variables $\tilde \xi_{A_i}$ for a given collection  $(A_i)_{1\le i \le m}$ of sets of trajectories,
thus generalizing \eqref{e:2point-ex}.
%
In the next lemma, with a slight abuse of notation, we will assume each $A_i$ to be an (arbitrary) set of random walk trajectories, instead of  the set of trajectories that pass through a given point as above.

\begin{lemma}\label{lem:corr-exact}
Let $\xi$ be the Poisson field of independent random walks as before, and let $(A_i)_{1\leq i\leq m}$ be sets of trajectories
with associated centered Poisson random variables $(\tilde \xi_{A_i})_{1\leq i\leq m}$. For $J\subset {[m]}:=\{1, \ldots, m\}$, let
\begin{equ}[e:defBJ]
B_J\eqdef \big(\cap_{i\in J} A_i\big) \cap \big( \cap_{i\in J^c} A_i^c\big) \;.
\end{equ}
Then
\begin{equ}\label{e:corr-exact}
\E\Big[ \prod_{i=1}^m \tilde \xi_{A_i} \Big] =
\sum_{I_1, \ldots, I_k \, partition\, {[m]} \atop |I_i|\geq 2 \,\forall\, 1\leq i\leq k}
\sum_{J_i\supset I_i \, \forall\, 1\leq i\leq k \atop J_i\neq J_{i'} \,\forall\, i\neq i'}
\prod_{i=1}^k \,
\E[\tilde \xi_{B_{J_i}}^{|I_i|}] \;.
\end{equ}
Here, $\tilde \xi_{B_{J_i}}$
is a centered Poisson variable with mean $\lambda \sum_{x\in \Z^d} \bP_{(0,x)}(S\in B_{J_i})$.
\end{lemma}

\begin{proof}
As in \eqref{e:2point-ex}, we can partition each $A_i$ into the $B_J$'s to write
\begin{equ}
\E  \Big[ \prod_{i=1}^m \tilde \xi_{A_i} \Big]
=\E \Big[ \prod_{i=1}^m
\sum_{J_i\subset [m]}  1_{\{i\in J_i\}} \tilde \xi_{B_{J_i}} \Big]
=  \sum_{J_1, \ldots, J_m\subset {[m]} \atop i\in J_i \, \forall\, 1\leq i\leq m}
\E\Big[\prod_{i=1}^m \tilde \xi_{B_{J_i}}\Big]  \;.
\end{equ}
Note that when $J_i\neq J_{i'}$,
one has $B_{J_i}\cap B_{J_{i'}}=\emptyset$
and thus $\tilde\xi_{B_{J_i}}$ and $\tilde\xi_{B_{J_{i'}}}$
are independent. Thus $(\tilde \xi_{B_J})_{J\subset {[m]}}$ is a family of independent centered Poisson random variables.
The collection $(J_i)_{i=1}^m$
determines a partition $\{I_1, \ldots, I_k \}$ of $[m]$
(for some $k\le m$)
where $i$ and $i'$ are in the same block
if and only if $J_i=J_{i'}$. Note that each block $I\in \{I_1, \ldots, I_k \}$ is associated with some $J\in \{J_1, \ldots, J_m\}$ with $J\supset I$.

Re-summing first over the partitions of $[m]$ and then over the $J$'s,
using the aforementioned independence of $\tilde\xi$ for disjoint sets,
and recalling that $\tilde\xi$ are centered, one obtains \eqref{e:corr-exact}.
\end{proof}

The next lemma provides bounds on the moments $\E[\tilde \xi_{B_{J_i}}^{|I_i|}] $  appearing on the right hand side of \eqref{e:corr-exact}.
We first introduce some notation. For a generic collection of space-time points
$((t_i,x_i))_{i\in J}$ indexed by some set $J$ with $|J|=j$, we can order the time variables
by $t_{i_1} < \cdots < t_{i_{j}}$ for some rearrangement of the set $J=\{i_1,\cdots,i_{j}\}$
and then we write
\begin{equ}[e:defCP]
\CP(J)
 := P_{t_{i_2}-t_{i_1}}(x_{i_2}-x_{i_1})
 	\cdots P_{t_{i_{j}}-t_{i_{j-1}}}(x_{i_{j}}-x_{i_{j-1}}) \;.
\end{equ}

\begin{lemma}\label{lem:mom-xiBJ}
In the same setting as above, for every $n\in \N$ and $J\subset [m]$
one has
\begin{equ}[e:boundExi-n]
\big| \E[\tilde \xi_{B_{J}}^{n}]\big| \leq C(\lambda, n) \CP (J)
\end{equ}
for some $C(\lambda, n)$ depending only on $\lambda$ and $n$.
\end{lemma}

\begin{proof}
Recall from  the definition \eqref{e:defBJ}
that $ \xi_{B_J}$ counts the number of trajectories in the Poisson point process $\xi$
that pass $(t_r,x_r)$ for each $r\in J$ but do not pass $(t_r,x_r)$ for each $r\notin J$,
which is a Poisson random variable with mean
$$
\E[\xi_{B_J}] = \lambda \sum_{x\in \Z^d} \bP_{(0,x)}\big(S_{t_r}=x_r \, \forall\, r\in J, \ S_{t_r}\neq x_r \, \forall\, r\in [m]\backslash J\big)
\leq \lambda \CP(J),
$$
where the bound is obtained by removing the constraint that $S_{t_r}\neq x_r \, \forall\, r\in [m]\backslash J$.

The $n$-th moment $\E[\tilde \xi_{B_J}^n]$ of the centered Poisson random variable $\tilde  \xi_{B_J}$ is given by the centered Touchard
polynomial $\tilde T_n(x)$ (see for instance \cite[Prop.~3.3.4]{PeccatiTaqqu}) of degree at most $n$ with $x=\E[\xi_{B_J}]$. In particular,
for all $n\geq 2$,
\begin{equ} [e:Touchard]
\tilde T_n (x) =x+ x^2 q_n(x) \ \mbox{for some polynomial } q_n, \mbox{ with } q_2(x)=q_3(x)=0.
\end{equ}
Since $\E[\xi_{B_J}]^k \leq \lambda^k \CP(J)^k \leq \lambda^k \CP(J)$ for all $k\in\N$, \eqref{e:boundExi-n} follows easily.
\end{proof}

Although the next lemma is not needed in our proof, it provides a more refined estimate on the correlation functions of a general Poisson field, which is of independent interest. In particular, we see in \eqref{ximom} below that the correlation function can be split into two parts: the correlation function of a Gaussian field with the same covariance as $\tilde \xi$, plus the remainder for which we give a bound.
\begin{lemma}\label{lemma-correlation}
Let $\xi$ be the Poisson point process on the space of random walk trajectories with intensity measure $\mu$ given by \eqref{e:mu}.
Let $(A_i)_{1\leq i\leq m}$ be sets of trajectories, with $A_J:=\cap_{i\in J} A_i$ for $J\subset {[m]}=\{1, \ldots, m\}$. Then
\begin{equation}\label{ximom}
\E\Big[ \prod_{i=1}^m \tilde \xi_{A_i} \Big] = \sum_{I_1, \ldots, I_k \, partition\, {[m]} \atop |I_i|= 2 \,\forall\, 1\leq i\leq k} \prod_{i=1}^k \mu(A_{I_i}) + R(A_1, \ldots, A_m),
\end{equation}
where the first term equals $0$ when $m$ is odd. If $\alpha:=\max\limits_{I\subset {[m]}, |I|=2} \mu(A_I)\leq 1$, then
\begin{equation}\label{ximomR}
|R(A_1, \ldots, A_m)| \leq  \sum_{I_1, \ldots, I_k \, partition\, {[m]} \atop \min_i |I_i|\geq 2; \ \max_i|I_i|>2} 3(1+C_m\alpha)^k\prod_{i=1}^k \mu(A_{I_i}),
\end{equation}
where $C_m$ depends only on $m$.
\end{lemma}
\begin{remark}
We note that Lemmas \ref{lem:corr-exact} and \ref{lemma-correlation} in fact hold for a general Poisson point process $\xi$ on a Polish space $X$ with $\sigma$-finite intensity measure $\mu$, and $(A_i)_{1\leq i\leq m}$ are subsets of $X$ with $\mu(A_i)<\infty$ for each $i$. The proof is exactly the same.
\end{remark}

\begin{proof} By \eqref{e:Touchard}, for any Poisson random variable $\xi_B$ with mean $\tilde \alpha\leq 1$, its centered moment satisfies
$$
\max_{1\leq k\leq m} |\E[\tilde \xi^k_B]|\leq  \tilde \alpha + C_m \tilde \alpha^2,
$$
where $C_m$ depends only on $m$.

If $\alpha\leq 1$, then we can bound the contributions to the sum in
\eqref{e:corr-exact} with $|I_i|>2$ for some $i$, by
\begin{equation}\label{ximom2}
\begin{aligned}
& \Big|\!\!\!\!\sum_{I_1, \ldots, I_k \, partition\, {[m]} \atop \min_i |I_i|\geq 2; \ \max_i|I_i|>2} \sum_{J_i\supset I_i \, \forall\, 1\leq i\leq k \atop J_i\neq J_{i'} \,\forall\, i\neq i'} \prod_{i=1}^k \E[\tilde \xi_{B_{J_i}}^{|I_i|}]\Big| \\
\leq& \!\! \sum_{I_1, \ldots, I_k \, partition\, {[m]} \atop \min_i |I_i|\geq 2; \ \max_i|I_i|>2}
\sum_{J_i\supset I_i \, \forall\, 1\leq i\leq k} \prod_{i=1}^k \E[|\tilde \xi_{B_{J_i}}^{|I_i|}|] \\
\leq& \sum_{I_1, \ldots, I_k \, partition\, {[m]} \atop \min_i |I_i|\geq 2; \ \max_i|I_i|>2}
\prod_{i=1}^k \sum_{J_i\supset I_i} (1+C_m\alpha)\mu(B_{J_i}) = \!\!\!\!\!\!\!\!\!\!\!\!\!\! \sum_{I_1, \ldots, I_k \, partition\, {[m]} \atop \min_i |I_i|\geq 2; \ \max_i|I_i|>2}
\!\!\!\!\!\!\!\!\!\!\!\!\! (1+C_m\alpha)^k\prod_{i=1}^k \mu(A_{I_i}).
\end{aligned}
\end{equation}
In the sum in \eqref{e:corr-exact}, when $|I_i|=2$ for all $i$ and we remove the constraint that $J_i\neq J_{i'}$, then we get precisely the first term in \eqref{ximom} since $\mu(A_{I_i}) = \sum_{J_i \supset I_i} \E[\tilde \xi_{B_{J_i}}^2]$. The difference can be bounded by
\begin{equation}\label{ximom3}
\begin{aligned}
& \sum_{I_1, \ldots, I_k \, partition\, {[m]} \atop |I_i|=2 \,\forall\, 1\leq i\leq k} \sum_{J_i\supset I_i \, \forall\, 1\leq i\leq k \atop J_a= J_b \, for \, some\, a\neq b} \prod_{i=1}^k \mu(B_{J_i}) \\
\leq & \sum_{I_1, \ldots, I_k \, partition\, {[m]} \atop |I_i|= 2 \,\forall\, 1\leq i\leq k} \sum_{1\leq a\neq b\leq k}
\sum_{J_i\supset I_i \, \forall\, 1\leq i\leq k \atop J_a= J_b} \prod_{i=1}^k \mu(B_{J_i})
\\
\leq & \sum_{I_1, \ldots, I_k \, partition\, {[m]} \atop |I_i|= 2 \,\forall\, 1\leq i\leq k} \sum_{1\leq a\neq b\leq k}
\mu(A_{I_a\cup I_b})\prod_{1\leq i\leq k\atop i\neq a, b} \mu(A_{I_i}) \\
\leq & \sum_{J_1, \ldots, J_{k-1}\, partition\, {[m]} \atop \min_i |J_i|= 2, \max_i |J_i|=4} 3 \prod_{i=1}^{k-1} \mu(A_{J_i}),
\end{aligned}
\end{equation}
where in the last step, we replaced $I_a$ and $I_b$ by a single partition element $I_a\cup I_b$. Combining the above estimates then gives \eqref{ximom} with the error bound \eqref{ximomR}.
\end{proof}

%

\section{Convergence of environment to Gaussian field}\label{S:fieldconv}

In this section, we prove Proposition \ref{P:xilim}. We start with a weaker version.  Recall that $(\xi(t,x))_{t\geq 0, x\in\Z^d}$ is the  Poisson field of random walks with mean $\lambda$ as defined  in \eqref{e:xi}, and $\tilde \xi(\cdot, \cdot):= \xi(\cdot, \cdot)-\lambda$.

\begin{lemma}\label{lem:distri-converge}
As $\e\downarrow 0$,
$\tilde\xi_\e \eqdef \e^{-\frac12}\tilde \xi(\e^{-2}\cdot, \e^{-1}\cdot)$ converges in distribution to the Gaussian field $\sqrt{\lambda}\Xi$
defined as in \eqref{Xi} in the sense that for any $\phi \in C_c([0,\infty)\times \R)$,
\begin{equ}[e:xiphi]
\tilde \xi_\e(\phi):=  \e\!\sum_{x\in \e\Z}\int_0^\infty \phi(t, x) \e^{-\frac12} \tilde \xi(\e^{-2}t, \e^{-1}x) {\rm d}t
\end{equ}
converges in distribution to a centered normal random variable $\sqrt{\lambda}\Xi(\phi)$ with variance
\begin{equ}[e:xiphivar]
2 \lambda \int_{s<t} \int_{\R^2} \phi(s,x) p_{t-s}(y-x) \phi(t,y) dxdydsdt
\end{equ}
\end{lemma}

\begin{proof}
Let $\phi \in C_c([0,\infty)\times \R)$, and
\begin{equ}[e:tilde-xi-phi]
\tilde \xi_\e(\phi) :=  \e^{\frac12}\sum_{x\in \e\Z}\int_0^\infty \phi(t, x)  \tilde \xi(\e^{-2}t, \e^{-1}x) {\rm d}t = \e^{\frac52}\sum_{x\in \Z}\int_0^\infty \phi(\e^2 t, \e x)  (\xi(t, x) -\lambda) {\rm d}t.
\end{equ}
To prove the convergence, it suffices to show that the Laplace transform of $\tilde \xi_\e(\phi)$ converges  to that of the correct normal. By \eqref{e:xi}, for any $\kappa\in\R$, we can integrate out the i.i.d.\ Poisson random variables $\xi(0,\cdot)$ to obtain
\begin{align} \label{eq:Laplace}
\E\left[ \exp\left\{\kappa \tilde \xi_\e(\phi) \right\}\right]
&= \Big(\prod_{y\in \Z}  \exp\left\{\lambda\left(w_\e(y) -1\right) \right\}\Big)
\\
&\qquad \cdot
\exp \Big\{ -\kappa \lambda \e^{\frac52} \sum_{x\in\Z} \int_0^\infty \phi(\e^2t, \e x) {\rm d}t \Big\}, \notag
\end{align}
where, denoting the expectation w.r.t a random walk $Y$ starting from $y$ by $\E^Y_y$,
$$
w_\e(y) := \E^Y_y\left[
\exp\left\{\kappa \e^{\frac 52}\int_0^\infty \phi(\e^2 t, \e Y_t){\rm d}t \right\}\right].
$$
Note that to obtain \eqref{eq:Laplace}, we have
used  the independence of the walks $Y^{y,i}$  and  $\xi(0,y)$ which allows us to write
\[
\E  \left[\exp\Big\{\sum_{i=1}^{\xi(0,y)}\kappa \e^{\frac 52}\int_0^\infty \phi(\e^2 t, \e Y_t^{y,i}){\rm d}t \Big\} \right]
=
\E [ w_\e(y)^{\xi(0,y)}  ]
\]
and then apply the formula of moment generating function of Poisson variables to $\xi(0,y)$.

Since $\phi$ has compact support in time, by Taylor expansion, we have
\begin{align*}
 w_\e(y)-1
& =  \kappa \e^{\frac 52} \! \int_0^\infty \!\! \E^Y_y[\phi(\e^2 t, \e Y_t)] {\rm d}t
\\ &\qquad
+ (1+o(1)) \kappa^2\e^5 \! \iint\limits_{0<t_1<t_2} \!\E^Y_y[\phi(\e^2t_1, \e Y_{t_1})\phi(\e^2t_2, \e Y_{t_2})] {\rm d}t_1 {\rm d}t_2.
\end{align*}
Here $o(1)$ stands for a quantity which vanishes as $\e\to 0$ uniformly in $y$ (due to our assumption on $\phi$).
Summing over $y$, we use the translation invariance of the random walk to obtain
\begin{align}
& \lambda \sum_y (w_\e(y)-1)  = \kappa \lambda \e^{\frac52} \sum_{x\in\Z} \int_0^\infty \phi(\e^2t, \e x) {\rm d}t  \\
& \quad  + (1+o(1)) \kappa^2 \lambda \e^5\!\!\! \sum_{x_1, x_2\in\Z} \ \ \iint\limits_{0<t_1<t_2} \phi(\e^2 t_1, \e x_1) \phi(\e^2 t_2, \e x_2) P_{t_2-t_1}(x_2-x_1) {\rm d}t_1 {\rm d}t_2. \nonumber
\end{align}
Substituting into \eqref{eq:Laplace}, we find that the Laplace transform equals
\begin{equation}\label{eq:expt}
\exp\Big\{(1+o(1)) \kappa^2 \lambda \e^5\!\!\! \sum_{x_1, x_2\in\Z} \ \iint\limits_{0<t_1<t_2} \!\!\! \phi(\e^2 t_1, \e x_1) \phi(\e^2 t_2, \e x_2) P_{t_2-t_1}(x_2-x_1) {\rm d}t_1 {\rm d}t_2  \Big\}.
\end{equation}
Fix $\delta>0$. Using the fact that $\phi$ has compact space-time support and bounding $\phi(\e^2 t_2, \e x_2)$ by $|\phi|_\infty$, it is easily seen that the contribution to the exponent from $t_2-t_1\le\delta \e^{-2}$ can be made arbitrarily small by choosing $\delta>0$ small. On the other hand, by the local central limit theorem, with $\tilde t_i=\e^{2} t_i$, $\tilde x_i= \e x_i$,
as $\e\downarrow 0$, we have
$$
P_{t_2-t_1}(x_2-x_1)= \frac{\e(1+o(1))}{\sqrt{2\pi (\tilde t_2- \tilde t_1)}} \exp\left\{-\frac{(\tilde x_2- \tilde x_1)^2}{2(\tilde t_2-\tilde t_1)}\right\}
$$
uniformly in $\tilde t_2-\tilde t_1>\delta$ and $\tilde x_1, \tilde x_2$ in the support of $\phi$. By a Riemann sum approximation, \eqref{eq:expt} is then seen to converge to
$$
\exp\Big\{\frac{1}{2} \kappa^2\cdot  2\lambda  \idotsint\limits_{0<\tilde t_1<\tilde t_2 \atop \tilde x_1, \tilde x_2\in\R} \phi(\tilde t_1, \tilde x_1) \phi(\tilde t_2, \tilde x_2) p_{\tilde t_2- \tilde t_1}(\tilde x_2- \tilde x_1){\rm d} \tilde x_1 {\rm d} \tilde x_2 {\rm d}\tilde t_1 {\rm d} \tilde t_2 \Big\},
$$
which is exactly the Laplace transform of a centered Gaussian with variance specified in \eqref{e:xiphivar}, which proves the lemma.
\end{proof}
\medskip

Lemma \ref{lem:distri-converge} identifies the distributional limit of the centered and rescaled random field $\tilde \xi_\e$.  We now boost it to show convergence
in the weighted Besov-H\"older space, defined in Appendix \ref{S:Besov}.

\begin{proof}[Proof of Proposition~\ref{P:xilim}]
Given Lemma~\ref{lem:distri-converge}, which identifies the limit as $\sqrt{\lambda}\Noise$, we only need to prove tightness of $(\tilde \xi_\e)_{\e>0}$ in the weighted Besov-H\"older space
$ \CC^\alpha_\ka$ defined in Appendix \ref{S:Besov}, and more specifically, verify the conditions in Proposition \ref{prop:tightness-cri}.
Since $\tilde\xi_\e$ is stationary in time, and we only prove convergence in law, we can assume that the field  $\tilde\xi_\e$ is defined over all time $t\in\R$.

We extend $\tilde \xi_\e $ to the continuum by  piecewise constant interpolation as follows.
For each $x\in \R$, we write
$\langle x \rangle = \e \lfloor \e^{-1}x \rfloor$.
We then define the extension
$\tilde \xi_\e (t,x)\eqdef \tilde\xi_\e (t,\langle x \rangle)$ for $(t,x)\in \R^2$.
For $z=(t,x)\in \R^2$ we also write $\langle z \rangle = (t,\langle x\rangle)$.
We note that by Proposition~\ref{prop:tightness-cri} with $d=1$, $\alpha'=-\frac12$,  to prove that $(\tilde \xi_\e)_{\e>0}$ is tight in $\CC^\alpha$ for any $\alpha<-\frac12$, it suffices to show  that for any fixed test function $\phi$, uniformly in $\e>0$ and $\Ell\in(0,1)$, the following bound
\begin{equ}[e:mombnd-xi]
\E  [\tilde \xi_\e (\phi^{\Ell})^{2n} ]
\lesssim \Ell^{-n}  \;,
\end{equ}
holds for every $n\ge 1$,
where $\phi^{\Ell}:=\phi_{(0,0)}^{\Ell}$ is given in \eqref{e:phil}.
Here and in the sequel of this proof, the constant multiple that is implicit in $\lesssim$ is {\it independent} of $\e$ and $\Ell$.
Note that the bound \eqref{e:mombnd-xi} is enough to imply \eqref{e:tight1}, \eqref{e:tight2}:
by {\it stationarity}, we can drop the supremum over $z$ therein and
just consider test functions centered at $z=0$,
and our parameter $\Ell$ can be taken as $\Ell=2^{-m}$ such that our $\phi^{\Ell}$ can be taken as  $2^{3m} \psi^{(i)} (2^{2m} (\cdot ),2^{m} (\cdot ))$ in \eqref{e:tight2}.
The bound \eqref{e:mombnd-xi} does not have to be uniform in $n$, since for any fixed $\alpha<-\frac12$ we only need to bound a moment of finite  order.

We first illustrate the idea of the proof with the case $n=1$.
By \eqref{e:2pt} or Lemma~\ref{lem:xiAAAPoi}, one has $\E [\tilde\xi_\e (w) \tilde\xi_\e (z)]=\lambda P^\e (w-z)$
where $w,z \in \R\times \e\Z$ are space-time points, and
\begin{equation}\label{e:Pe}
P^\e (z) := \e^{-1}  P_{\e^{-2}|t|} (\e^{-1}x), ~~z=(t,x)\in \R\times \e\Z,
\end{equation}
 is the rescaled random walk transition kernel.
One then has
\begin{equ}[e:2ndmom-ex]
\E  [\tilde \xi_\e (\phi^{\Ell})^{2} ]
= \lambda\int_{\R^4}    
\phi^{\Ell}(w)\phi^{\Ell}(z)
P^\e (\langle w\rangle -\langle z\rangle ) \,dwdz\;.
\end{equ}
To estimate this integral,
we use the heat kernel bound
\[
P^\e (z)\lesssim (\|z\|\vee \e)^{-1}
\]
 for every $z\in\R\times \e \Z$ with $\|z\|\le 1$, where $ \|z\|$ is as defined in \eqref{e:para-dist}.
See for instance \cite{HaiMat}, in particular Lemma~5.4, Remark~5.5 and Eq.~(7.7) as well as the paragraph above Eq.~(7.14) therein.
\footnote{Note that in  \cite[Eq.~(7.7)]{HaiMat} the parameter $m$ there is zero in our case, since we do not care about the derivatives of the heat kernel; thus we do not have any issue with nondifferentiability at $t=0$ so that the $\sup_{z\notin P_0}$  in  \cite[Eq.~(7.7)]{HaiMat} is not relevant.
In general one has $P^\e (z)\lesssim (\|z\|\vee \e)^{-d}$ in $d$ spatial dimensions.  Alternatively this can be proved by the method in Section~\ref{sec:HK}.
}

Therefore
\begin{equ}[e:P-1]
P^\e (\langle w\rangle -\langle z\rangle )\lesssim \|w-z\|^{-1}\;.
\end{equ}
Since $\phi^{\Ell}$ has support of diameter $\Ell$,
\eqref{e:2ndmom-ex} is bounded by
\begin{equs}
\Ell^{-6}\int_{\|w\|,\|z\|\leq \Ell}&\|\phi\|_\infty^2
\|w-z\|^{-1} \,dwdz
\lesssim
\Ell^{-6} \int_{\|w\|\leq \Ell} \int_{\|z-w\|\leq 2\Ell}
\|w-z\|^{-1} \,dzdw
\\
&\lesssim
\Ell^{-6} \int_{\|w\|\leq \Ell} \int_{\sqrt{|t|}\lesssim\Ell, |x|\lesssim\Ell}
\frac{1}{\sqrt{|t|}+|x|}\,dtdxdw
\\
&\lesssim
\Ell^{-6} \int_{\|w\|\leq \Ell} \Ell^{2} \,dw
\lesssim \Ell^{-6} \Ell^{2} \Ell^{3}
= \Ell^{-1} \;.
\end{equs}
Note that with our parabolic distance, in estimating the above integrals, each space-time variable should be thought of as having scaling dimension 3.
%

To prove \eqref{e:mombnd-xi} for arbitrary $n$,
given a collection of $2n$ space-time points $\{(t_i,x_i)\}_{i=1}^{2n}$,
invoking
Lemma~\ref{lem:corr-exact},
one has
\begin{equ}
\E  \bigg[\prod_{i=1}^{2n}\tilde \xi (t_i,x_i)  \bigg]
=
 \!\!\!\!
\sum_{I_1, \ldots, I_k \, partition\, {[2n]} \atop |I_i|\ge  2 \,\forall\, 1\leq i\leq k}
\sum_{J_i\supset I_i \, \forall\, 1\leq i\leq k \atop J_i\neq J_{i'} \,\forall\, i\neq i'}
\prod_{i=1}^k \,
\E[\tilde \xi_{B_{J_i}}^{|I_i|}]\;.
\end{equ}
Appling Lemma~\ref{lem:mom-xiBJ}, with $[2n]$ playing the role of $[m]$ therein, and the power $n$ therein being $|I_i|$, one has
\begin{equ}[e:mom-appliedCP]
\bigg|\E  \bigg[\prod_{i=1}^{2n}\tilde \xi (t_i,x_i)  \bigg]\bigg|
\lesssim
 \!\!\!\!
\sum_{I_1, \ldots, I_k \, partition\, {[2n]} \atop |I_i|\ge  2 \,\forall\, 1\leq i\leq k}
\sum_{J_i\supset I_i \, \forall\, 1\leq i\leq k \atop J_i\neq J_{i'} \,\forall\, i\neq i'}
\prod_{i=1}^k \,
\CP(J_i)\;.
\end{equ}
 where $\CP(J)$ is defined in \eqref{e:defCP}.
 The following is a graphic illustration for \eqref{e:mom-appliedCP} for the 8th moment.

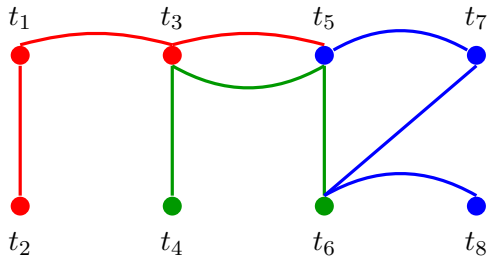
\begin{center}
\begin{tikzpicture}[thick,scale=2]

\node[dot,red] (11) at (1,1) {};
\node[dot,red] (12) at (2,1) {};
\node[dot,blue] (13) at (3,1) {};
\node[dot,blue] (14) at (4,1) {};
\node[dot,red] (21) at (1,0) {};
\node[dot,darkgreen] (22) at (2,0) {};
\node[dot,darkgreen] (23) at (3,0) {};
\node[dot,blue] (24) at (4,0) {};

\node[yshift=3ex] at (11) {$t_1$};
\node[yshift=3ex] at (12)  {$t_3$};
\node[yshift=3ex] at (13)  {$t_5$};
\node[yshift=3ex] at (14)  {$t_7$};
\node[yshift=-3ex] at (21) {$t_2$};
\node[yshift=-3ex] at (22)  {$t_4$};
\node[yshift=-3ex] at (23)  {$t_6$};
\node[yshift=-3ex] at (24)  {$t_8$};

%

\draw[very thick,red] (21) to (11);
\draw[very thick,red, bend left=15] (11.north) to (12.north);
\draw[very thick,red, bend left=15] (12.north) to (13.north);

\draw[very thick,darkgreen] (22) to (12);
\draw[very thick,darkgreen, bend right=30] (12.south) to (13.south);
\draw[very thick,darkgreen] (13.south) to (23);

\draw[very thick,blue,bend left=30] (13) to (14);
\draw[very thick,blue] (14.south) to (23.north);
\draw[very thick,blue,bend left=30] (23.north) to (24.north);
\end{tikzpicture}
\center{Figure 1}
\end{center}

Here each dot represents a space-time point (spatial coordinates not drawn) that is being integrated.
The partition has 3 blocks $I_1,I_2,I_3$, represented by 3 different colors. For instance the block $I_1$  consists of the three red points, and $J_1\supset I_1$ consists of four points  that are linked by the 3 red lines which represent the 3 heat kernels of $\CP(J_1)$. The interpretation of graphic notation for $I_2$ (green) and $I_3$ (blue) is analogous.
This graph is showing a situation where the time variables are ordered
such that $ t_2 \le t_1  \le t_3 \le t_5 $,
$t_4 \le t_3 \le t_5  \le t_6 $
and $t_5 \le t_7 \le t_6 \le t_8 $, which explains the arrangement of the heat kernels in each $\CP(J_i)$.

Next, we claim that
\begin{equ}[e:reduce-I-J]
\bigg|\E  \bigg[ \prod_{i=1}^{2n}\tilde \xi (t_i,x_i)  \bigg]\bigg|
\lesssim
\sum_{I_1, \ldots, I_k}
\prod_{i=1}^k \,
\CP(I_i)
\end{equ}
where
$I_1, \ldots, I_k$
is a partition of $[2n]$
and
$ |I_i|\ge  2$  for every  $1\leq i\leq k$.
This holds because $I_i\subset J_i$ and the observation that
\begin{equ}[e:PPbyP]
P_{s_2-s_1}(h_2-h_1)
 	\cdots
	P_{s_{\bar{k}}-s_{\bar{k}-1}}(h_{\bar{k}}-h_{\bar{k}-1})
\le
P_{s_{\bar{k}}-s_1}(h_{\bar{k}}-h_1),
\qquad
\forall \bar{k}>1\;,
\end{equ}
which holds since the left hand side is the probability of the event that a random walk conditioned on
$S_{s_1}=h_1$ satisfies $S_{s_i}=h_i$ for all $2\leq i\leq \bar{k}$, while the right hand corresponds to the event that $S_{s_{\bar{k}}}=h_{\bar{k}}$.
Note that dropping the sum $\sum_{J_1, \ldots, J_k}$ only costs a constant factor depending on $n$.
In the  example above, \eqref{e:reduce-I-J} reduces the graph to
 \begin{center}
\begin{tikzpicture}[thick,scale=2]

\node[dot,red] (11) at (1,1) {};
\node[dot,red] (12) at (2,1) {};
\node[dot,blue] (13) at (3,1) {};
\node[dot,blue] (14) at (4,1) {};
\node[dot,red] (21) at (1,0) {};
\node[dot,darkgreen] (22) at (2,0) {};
\node[dot,darkgreen] (23) at (3,0) {};
\node[dot,blue] (24) at (4,0) {};

\node[yshift=3ex] at (11) {$t_1$};
\node[yshift=3ex] at (12)  {$t_3$};
\node[yshift=3ex] at (13)  {$t_5$};
\node[yshift=3ex] at (14)  {$t_7$};
\node[yshift=-3ex] at (21) {$t_2$};
\node[yshift=-3ex] at (22)  {$t_4$};
\node[yshift=-3ex] at (23)  {$t_6$};
\node[yshift=-3ex] at (24)  {$t_8$};

%

\draw[very thick,red] (21) to (11);
\draw[very thick,red, bend left=15] (11.north) to (12.north);

\draw[very thick,darkgreen, bend right=30] (22.south) to (23.south);

\draw[very thick,blue,bend left=30] (13) to (14);
\draw[very thick,blue] (14.south) to (24.north);

\end{tikzpicture}
\center{Figure 2}
\end{center}


Now we rescale,
and pass to macroscopic space-time variables.
Let $\CP^\e$ be a rescaling of $\CP$, which is defined as in
 \eqref{e:defCP} but with each $P $ replaced by $P^\e$ given in \eqref{e:Pe}.
Recall that for a generic function $\phi\in C_c([0,\infty)\times \R)$, $\tilde \xi_\e (\phi)$ was defined in \eqref{e:tilde-xi-phi} and $\tilde \xi_\e(\phi)
  = \e^{\frac52}  \sum_{x\in \Z} \int_\R \phi(\e^2 t, \e x)  \tilde \xi(t, x) {\rm d}t$.
Applying \eqref{e:reduce-I-J}, we have
 \begin{equ}[e:2nmom-rescaled]
\E  [\tilde \xi_\e (\phi^{\Ell})^{2n}]
\lesssim
\int_{(\R \times  \R)^{2n}}
\sum_{I_1, \ldots, I_k}
 \e^{\frac52\cdot 2n}
  \e^{-3\cdot 2n}
   \e^{2n-k}
\prod_{i=1}^k \,
\CP^\e(I_i) \prod_{i=1}^{2n} \phi^{\Ell}( t_i, x_i)
\,d{\vec t} d{\vec x} \;,
\end{equ}
 where $  \e^{-3\cdot 2n} $ arises from switching from microscopic to macroscopic variables,
 $2n-k $ is the total number of heat kernels in
 $\prod_{i=1}^k \CP(I_i)$ (each heat kernel contributes a factor $\e$ when we switch from $\CP$ to $\CP^\e$).
 Note that  we have  $ \e^{\frac52\cdot 2n}
  \e^{-3\cdot 2n}
   \e^{2n-k} =\e^{n-k}$, and $\CP^\e$ is a function of the variables $(t_1, \langle x_1\rangle), \cdots, (t_{2n}, \langle x_{2n} \rangle)$.

Fixing a partition $I_1, \ldots, I_k$, we proceed as in \eqref{e:2ndmom-ex}. Note that the integral in \eqref{e:2nmom-rescaled} factors into integrals over variables in different blocks $I$ of the partition (which can be more intuitively seen from Figure 2). Within each block, if $\Ell \ge \e$,
 \begin{equs}
\Big| \int_{(\R \times \R)^{|I|}}
\CP^\e(I) \prod_{i\in I} \Big( \phi^{\Ell}( z_i)
\,dz_i \Big) \Big|
& \lesssim
\Ell^{-3|I|}\int_{\|z_i\|\lesssim\Ell\, \forall i}
\prod_{i=2}^{|I|}\|z_i - z_{i-1}\|^{-1} dz_1\cdots dz_{|I|}\\
& \lesssim
\Ell^{-|I|+1}
\end{equs}
where we applied \eqref{e:P-1}, and the last step is obtained by integrating out the
variables in the order of $z_1,z_2,\cdots$ one by one, again keeping in mind that each space-time variable $z_i$ has scaling dimension 3.
The integral on the RHS of \eqref{e:2nmom-rescaled} for a fixed partition with $k$ blocks is then bounded by
 \begin{equ}
 \e^{n-k} \Ell^{-2n + k} \lesssim  \Ell^{-n}
\end{equ}
uniformly in $0<\e \le \Ell<1$.
If $\Ell < \e$, we simply bound $P^\e$ by $\e^{-1}$
which yields
 \begin{equ}
\Big| \int_{(\R \times \R)^{|I|}}
\CP^\e(I_i) \prod_{z\in I} \phi^{\Ell}( z)
\,dz \Big|
\lesssim
\e^{-|I|+1} \;.
\end{equ}
The integral on the RHS of \eqref{e:2nmom-rescaled} for a fixed partition with $k$ blocks is then bounded by
 \begin{equ}
 \e^{n-k} \e^{-2n + k} \lesssim  \Ell^{-n}
\end{equ}
uniformly in $0<\Ell<\e$.
Summing over all possible partitions only costs a constant factor depending on $n$, which implies the desired moment bound \eqref{e:mombnd-xi}. This concludes the proof of Proposition \ref{P:xilim}.
\end{proof}

\section{Bounds on polymer partition functions} \label{S:polymer}
In this section, we prove Lemma \ref{L:Zbound}. First we recall an identity for the annealed polymer partition function $\E[Z^\xi_{T, \beta}]$ which we will need for our analysis.

\BL[Annealed partition function] \label{L:annZ}
Let $Z:=Z^\xi_{T, \beta}(0,0)$ be the quenched partition function defined as in \eqref{e:Z}, with $\beta\in\R$. Then we have
\begin{equation}\label{eq:EZ}
\E[Z] =  \bE\left[ \exp\left\{\lambda \beta \int_0^T (v_S(t, S_t) -1){\rm d}t \right\} \right],
\end{equation}
where given the random walk path $S$,
\begin{equation}\label{eq:vS}
v_S(t, y) = \bE^Y_y\left[\exp\left\{\beta \int_0^t \1_{\{Y_s=S_{t-s}\}} {\rm ds}  \right\}  \right]
\end{equation}
where $\bE^Y_y$ is for a simple symmetric random walk $Y$ starting from $y$ at time $0$.
\EL
\begin{proof}
For $\beta>0$, this is proved in \cite[Prop.~2.1 and (2.9)]{GdH06}. Their proof also applies to $\beta<0$ as shown in \cite[Section 2.1]{DGRS12}. The basic idea is to integrate out the Poisson field $\xi$.
In particular, the proof of \eqref{eq:EZ}  relies on the fact that $v_S$ satisfies the following equation
\[
\dfrac{\partial}{\partial s}v_S(s,y)
={\frac12} \Delta v_S(s,y)+ \beta\1_{\{S_{s}=y\}}v_S(s,y), \qquad v_S(0, \cdot)\equiv 1
\]
where $\Delta$ is the discrete Laplacian on $\Z$.
\end{proof}

\BL[Maximum at $S\equiv 0$] \label{L:maxS}
Given a random walk path $S:=(S_t)_{t\geq 0}$, let $v_S$ be defined as in \eqref{eq:vS}. Then for any $\beta\in\R$ and $t>0$,
\begin{equation}\label{eq:maxS}
\exp\left\{\lambda \beta \int_0^T v_S(t, S_t) {d}t \right\}  \leq \exp\left\{\lambda \beta \int_0^T v_0(t, 0) {   d}t \right\},
\end{equation}
where we  have set $S\equiv 0$ in $v_0$.
\EL
\begin{proof} For $\beta\geq 0$, \eqref{eq:maxS} follows from the inequality $v_S(t, S_t) \leq v_0(t, 0)$ for all $S$ and $t\geq 0$,
which is easily seen by Taylor expanding the exponential in \eqref{eq:vS} and using the fact that for a continuous time simple symmetric random walk,  its transition kernel satisfies $P_t(y) \leq P_t(0)$ for all $t\geq 0$ and $y\in \Z$.

The case $\beta<0$ is much more delicate and relies on the so-called Pascal Principle, which interprets the left hand side of \eqref{eq:maxS} as the probability that a random walk following the trajectory $S$ survives among the Poisson field of traps $\xi$
up to time $T$, with the right hand side of \eqref{eq:maxS} corresponding to the walk staying put at the origin. For the proof and further details, see \cite[Prop.~2.1]{DGRS12}.
\end{proof}

\begin{proof}[Proof of Lemma \ref{L:Zbound}] We first prove \eqref{e:Zbound}. Without loss of generality, we may assume $(t,x)=(0,0)$.
Applying the identity \eqref{eq:EZ} to $\CZ_{\e,\beta}:=Z^\xi_{\e^{-2}, \e^{3/2}\beta}(0,0)$ and the comparison inequality \eqref{eq:maxS}, we obtain
\begin{equation}\label{eq:Zepsbd}
\E[\CZ_{\e,\beta}] \leq  \exp\left\{\lambda \beta \e^{\frac32} \int_0^{\e^{-2}} (v_0(t, 0) -1){   d}t \right\},
\end{equation}
where as in \eqref{eq:vS},
\begin{equation}\label{eq:v0t0}
v_0(t,0) =\bE^Y_0\left[\exp\left\{\beta \e^{\frac32} \int_0^t \1_{\{Y_s=0\}} {   ds}  \right\}  \right].
\end{equation}

Let $L_t:=\int_0^t \1_{\{Y_s=0\}}ds$. Then we note that there exists $c>0$ such that for all $a>0$ and $t\geq 0$,
\begin{equ}[e:bound-Lt-eLt]
\bE^Y_0\left[ \frac{L_t}{\sqrt t} \right]  \leq c \quad \mbox{and} \quad
\bE^Y_0\left[e^{a\frac{L_t}{\sqrt{t}}}\right]\le c e^{c a^2}, \qquad t\ge0,
\end{equ}
Indeed, the first bound follows from the local limit theorem $P_s(0) \leq 1\wedge C/\sqrt{s}$, while for the second bound, by Taylor
expanding the exponential and applying the local limit theorem, we have
\begin{align*}
\bE^Y_0\left[e^{a\frac{L_t}{\sqrt{t}}}\right]  & = 1+\sum_{k=1}^\infty \Big(\frac{a}{\sqrt t}\Big)^k \idotsint\limits_{0<t_1<\cdots <t_k<t}
P_{t_1}(0)\cdots P_{t_k-t_{k-1}}(0) {   d}t_1 \cdots {   d}t_k \\
& \leq 1+ \sum_{k=1}^\infty \Big(\frac{a}{\sqrt t}\Big)^k \idotsint\limits_{0<t_1<\cdots <t_k<t} \frac{C^k}{\sqrt{t_1(t_2-t_1) \cdots (t_k-t_{k-1})}}
 {   d}t_1 \cdots {   d}t_k \\
 & \le 1+ \sum_{k=1}^\infty (aC)^k \idotsint\limits_{0<s_1<\cdots <s_k<1} \frac{{   d}s_1 \cdots {   d}s_k}{\sqrt{s_1(s_2-s_1) \cdots (s_k-s_{k-1})(1-s_k)}} \\
 & = 1 + \sum_{k=1}^\infty (aC)^k \frac{\Gamma(\frac{1}{2})^{k+1}}{\Gamma(\frac{k+1}{2})} \leq 1+ C'\sum_{l=1}^\infty \frac{a^{2l}}{l!} \leq ce^{ca^2},
\end{align*}
where we used that the integral is a multivariate Beta function. In particular, uniformly in $0\le t\le \e^{-2}$ and let $a=1$, we obtain by Markov's inequality
$$
\bP^Y_0\left(\frac{L_t}{\sqrt t} > \e^{-\frac14}\right) \leq C e^{-\e^{-\frac14}}.
$$
Recall \eqref{eq:v0t0}, uniformly in $0\le t \le \e^{-2}$, we can now bound
\begin{align}
& \Big|v_0(t, 0)-1 \Big| = \Big| \bE^Y_0\left[\exp\left\{\beta \e^{\frac32}L_t \right\} -1 \right] \Big| \notag \\
\leq \ & \bE^Y_0\left[\Big(\exp\left\{\beta \e^{\frac32}L_t \right\} +1\Big) \1_{\{\frac{L_t}{\sqrt{t}} > \e^{-1/4} \}} \right]
+ \bE^Y_0\left[\Big|\exp\left\{\beta \e^{\frac32}L_t \right\} -1\Big| \1_{\{\frac{L_t}{\sqrt{t}} \leq \e^{-1/4} \}} \right]  \notag
\\
\leq \ &
  \bP^Y_0\Big(\frac{L_t}{\sqrt t} \geq \e^{-\frac14} \Big)
  + \bP^Y_0\Big(\frac{L_t}{\sqrt t} \geq \e^{-\frac14} \Big)^{\frac12} \bE^Y_0\Big[\exp\Big\{\frac{2|\beta| \e^{\frac12}L_t}{\sqrt t}\Big\} \Big]^{\frac12}
  + C|\beta| \e^{\frac32} \E^Y_0[L_t]  \notag
\\
\le \ &  C|\beta| \e^{\frac12} + Ce^{-\frac{1}{2}\e^{-\frac14}}.   \label{e:bound-v0-1}
\end{align}
Here we have used both bounds in \eqref{e:bound-Lt-eLt}, and used
the Taylor remainder theorem
to bound the argument of the second $\bE^Y_0$ by
\[
|\beta| \e^{\frac32} L_t e^{|\beta|\e^{3/2} L_t}
\le |\beta| \e^{\frac32} L_t e^{|\beta|\e^{3/2-1/4}  \sqrt{t}}\le C |\beta| \e^{\frac32} L_t \;.
\]
Note that the bound \eqref{e:bound-v0-1} holds for any $\beta\in\R$. Substituting \eqref{e:bound-v0-1} into \eqref{eq:Zepsbd} then shows that for any $\beta\in\R$,  $\E[\CZ_{\e,\beta}]$ is uniformly bounded in $\e$.

To prove \eqref{e:Ztail}, we recall from \eqref{eq:Zexp} that given $\beta>0$,
\begin{equation}
\CZ_{\e, \beta} := Z^\xi_{\e^{-2},\beta \e^{3/2}}(0,0)  = 1+\sum_{k=1}^\infty \frac{1}{k!} \bE\left[\Big(\beta \e^{\frac32}\int_0^{\e^{-2}} \tilde \xi(s, S_s) {   d}s\Big)^k  \right] =: 1+\sum_{k=1}^\infty \CZ^{(k)}_{\e, \beta}.
\end{equation}
Denote $H_{\e}:=\e^{\frac32}\int_0^{\e^{-2}} \tilde \xi(s, S_s) {   d}s$. Similarly, replacing $\beta$ by $-\beta$, we have
\begin{equation}
\CZ_{\e, -\beta} := Z^\xi_{\e^{-2}, -\beta \e^{3/2}}(0,0)  = 1+\sum_{k=1}^\infty \frac{1}{k!} \bE\left[ (-\beta H_{\e})^k \right] =: 1+\sum_{k=1}^\infty (-1)^k \CZ^{(k)}_{\e, \beta}.
\end{equation}
Note that for any odd $k\in\N$,
\begin{equ}[e:betaH^k]
\frac{|\beta H_{\e}|^k}{k!} \leq 2 \sqrt{\frac{(\beta H_{\e})^{k-1}}{(k-1)!}} \sqrt{\frac{(\beta H_{\e})^{k+1}}{(k+1)!}} \leq \frac{(\beta H_{\e})^{k-1}}{(k-1)!} + \frac{(\beta H_{\e})^{k+1}}{(k+1)!}.
\end{equ}
Therefore for any $m\in\N$, which we assume to be even for notational simplicity,
\begin{align*}
 \E\Big[\Big|\CZ_\e(t, & x)-1  -\sum_{k=1}^m \CZ_\e^{(k)}(t,x) \Big| \Big]
 \leq
  \sum_{k=m+1}^\infty \frac{1}{k!} \E\bE\big[|\beta H_{\e}|^k  \big]
  \\
& \leq 3 \E\bE\Big[ \sum_{k=\frac{m}{2}}^\infty \frac{1}{(2k)!} (\beta H_{\e})^{2k}  \Big]
\leq 3  \E\bE\Big[ \frac{(\beta H_{\e})^m}{m!} \sum_{k=0}^\infty  \frac{(\beta H_{\e})^{2k}}{(2k)!} \Big]
\end{align*}
where the penultimate inequality used
\eqref{e:betaH^k},
and the last inequality used
$(2k)! > m! (2k-m)!$.
With $e^x+e^{-x} \ge 2x^m / m!$ for all $x\in \R$ we can bound the above quantity by
\begin{align*}
\leq \ & 3  \E\bE\left[ \frac{e^{2\beta H_{\e}}+e^{-2\beta H_{\e}}}{2^{m+1}} \cdot \frac{e^{\beta H_{\e}} + e^{-\beta H_{\e}}}{2}\right]
\\
=\ & \frac{3}{2^{m+2}}\E\bE\big[e^{3\beta H_{\e}}+e^{\beta H_{\e}} + e^{-\beta H_{\e}}+e^{-3\beta H_{\e}} \big] \\
=\ & \frac{3}{2^{m+2}} (\CZ_{\e, 3\beta} + \CZ_{\e, \beta} + \CZ_{\e, -\beta} +\CZ_{\e, -3\beta}).
\end{align*}
Since we have shown that $\CZ_{\e, \beta}$ is bounded as $\e\downarrow 0$ for all $\beta\in\R$, \eqref{e:Ztail} follows.
\end{proof}

\section{Convergence of finite order chaos} \label{S:conv}
In this section, we prove Lemma~\ref{L:termconv}.
When $\xi(t,x)$ are i.i.d.\ random variables with exponential moments, the approach used by Alberts, Khanin and Quastel~\cite{AKQ} (see also \cite{CSZ17a,CSZ17b}) was to show that the $m$-th order term converges to an $m$-fold stochastic integral with respect to a space-time white noise, where chaos of different order are mutually orthogonal in $L^2$. The proof of convergence relies on the approximation of iterated stochastic integrals (and their discrete analogues) where the integrands are approximated in $L^2$ by linear combinations of product functions and then apply It\^o isometry. However, for Stratonovich integrals with respect to a correlated noise as in our case, there is no It\^o isometry, and it is not clear how  an approximation as in \cite{AKQ} can be carried out, and in which function space. Therefore we follow a different, more functional analytic approach here, inspired by the recent developments in the study of singular SPDEs.


Before outlining the proof steps, we first recall the setup. Consider the $m$-th order term in the expansion \eqref{eq:Zexp} for the partition function 
\begin{align*}
Z_{T, \beta}^{(m)}(t_0,x_0) =
\beta^m
\sum_{x_1,\cdots,x_m} \;
\idotsint\limits_{t_0<t_1<\cdots < t_m< t_{m+1}=T} \prod_{i=1}^m P_{t_i-t_{i-1}}(x_i-x_{i-1}) \tilde \xi(t_i,x_i) {   d}t_i
\end{align*}
for $(t_0,x_0)\in (0,T)\times \Z$.
Upon scaling
$\beta \mapsto \e^{\frac32}\beta$
 as in \eqref{e:scaling} with $T=\e^{-2}$,
$x_i\mapsto \e^{-2}x_{m+1-i}$,
 and reversing and scaling the time variables by
\[
s_i = 1-\frac{t_{m+1-i}}{T}
\qquad \mbox{and} \qquad t = 1-\frac{t_{0}}{T},
\]
one has that $\CZ_{\e, \beta}^{(m)}$ defined in
\eqref{e:expand-rescaled-Z} is given by
\begin{align}
\CZ_{\e, \beta}^{(m)}(1-t,x)
 &= \beta^m
\int_{([0,1] \times \e\Z)^m} \prod_{i=1}^m P^\e_{s_{i+1}-s_i}(x_{i+1}-x_i) \tilde \xi_\e(1-s_i,x_i) \prod_{i=1}^m (dx_ids_i) \notag
\\
& \stackrel{law}{=}
 \beta^m
\int_{([0,1] \times \e\Z)^m} \prod_{i=1}^m P^\e_{s_{i+1}-s_i}(x_{i+1}-x_i) \tilde \xi_\e(s_i,x_i) \prod_{i=1}^m (dx_ids_i)
	\label{e:Zm_epsbeta}
\end{align}
where
$\int_{\e \Z} {   d} x \eqdef \e \sum_{x\in \e\Z}$,
$(s_{m+1},x_{m+1})\eqdef (t,x)$, $\tilde\xi_\e \eqdef \e^{-\frac12}\tilde \xi(\e^{-2}\cdot, \e^{-1}\cdot)$,
and
\begin{equ}[def:Peps]
P^\e (t,x) \eqdef \mathbf{1}_{t\ge 0 }\e^{-1} P(\e^{-2}t,\e^{-1}x)
\qquad
\mbox{for }(t,x)\in \R\times \e\Z \;.
\end{equ}
Here the characteristic function $ \mathbf{1}_{t\ge 0 }$ automatically  imposes the ordering $0<s_1<\cdots <s_m<s_{m+1}=t<1$ (in particular, \eqref{e:Zm_epsbeta} remains unchanged if we integrate over $([0,t] \times \e\Z)^m$).
Note that the above expression can be viewed as an $m$ times iteration of two consecutive operations: the discrete convolution $f\mapsto P^\e *_\e f$  as defined  in \eqref{e:dis-con} below and the multiplication $f\mapsto \tilde \xi_\e \cdot f$.
Namely we have the recursion $\CZ_{\e, \beta}^{(0)}\eqdef 1$ and
\begin{equ}\label{eq:rec}
\CZ_{\e, \beta}^{(m)} (1-t,x)
=\beta \int_0^t \int_{\e\Z}
P^\e_{t-s}(x-y) \tilde \xi_\e(s,y) \cdot \CZ_{\e, \beta}^{(m-1)}(1-s,y)\, dyds \;.
\end{equ}
Note that the terms in the expansion of $u(t,x)$ in \eqref{e:uexp}, which we denote by $u^{(m)}(t,x)$, satisfy a similar recursive relation  which is given in \eqref{e:u-n}, with $\tilde \xi_\e$ replaced by $\sqrt{\lambda}\Noise$ and $P^\e$ replaced by $p$.

 To prove the convergence of $\CZ_{\e, \beta}^{(m)} (1-\cdot,\cdot)$ to $u^{(m)}(\cdot,\cdot)$, we proceed inductively in $m$ as follows. The starting point is the convergence $\tilde\xi_\e\to \sqrt{\lambda}\Noise$ in $\CC_\ka^{\alpha}$ for any $\alpha<-\frac12$, proved in Proposition~\ref{P:xilim}. Assume that for a given $\alpha'>-\alpha$, we have shown for some $m\in\N$ that $\CZ_{\e, \beta}^{(m-1)}(1-t,x)$ converges
in $\CC^{\alpha'}_{(m-1)\kappa}$ to the limit $u^{(m-1)}(t,x)$. We will prove that the same holds with $m-1$ replaced by $m$. Note that the recursion \eqref{eq:rec} defining $\CZ_{\e, \beta}^{(m)}$ from $\CZ_{\e, \beta}^{(m-1)}$ consists of two operations: the first is multiplication of $\CZ_{\e, \beta}^{(m-1)} \in \CC^{\alpha'}_{(m-1)\kappa} $ by $\tilde\xi_\e \in \CC_\ka^{\alpha}$; the second is space-time convolution with $P^\e$. To prove convergence of $\CZ_{\e, \beta}^{(m)}$ in $\CC^{\alpha'}_{m\kappa}$, it suffices to show that the multiplication is continuous, and the error of replacing $P^\e$ by $p$ is negligible as $\e\to 0$. 

For the continuity of the multiplication, we use the following fact (see for instance \cite[Prop.~4.14]{Hairer14} or \cite[Lemma~B.5]{MourratWeber}, with obvious generalization to the weighted case):
for $\alpha < 0 < \alpha'$ with $\alpha + \alpha' >0$,  the classical multiplication mapping $(Z_1, Z_2) \mapsto Z_1 Z_2$ extends uniquely to a continuous bilinear mapping  $\CC^{\alpha'}_{\kappa_1} \times \CC^{\alpha}_{\kappa_2}
 \to \CC^{\alpha}_{\kappa_1+\kappa_2}$, that is
\begin{equ}[e:Young]
\| Z_1 \, Z_2 \|_{\CC^{\alpha}_{\kappa_1+\kappa_2}} \lesssim \|Z_1 \|_{\CC^{\alpha'}_{\ka_1}} \, \| Z_2\|_{\CC^{\alpha}_{\ka_2}}\;.
\end{equ}
The subsequent convolution with $P^\e$ then increases the regularity back to $\alpha'$, while controlling the error between $P^\e$ and $p$ acting on elements of $\CC^{\alpha}_{m\kappa}$ requires careful discrete Besov space analysis and local limit theorem type of estimates.

\begin{proof}[Proof of Lemma~\ref{L:termconv}]
We first apply Skorohod's representation theorem to couple the discrete and the limiting noises, such that $\tilde\xi_\e$ converges to $\sqrt{\lambda}\Noise$ almost surely in $\CC^{-\frac12-\delta}_\ka$.
Denote by
$\mathcal R\CZ_{\e, \beta}^{(m)} (t,x) \eqdef \CZ_{\e, \beta}^{(m)} (1-t,x)$ the ``time reflection'' of $\CZ_{\e, \beta}^{(m)}$. We will prove that for any $M>0$, almost surely $(\mathcal R\CZ_{\e, \beta}^{(m)})_{m=0}^M \to (u^{(m)})_{m=0}^M$
  in $\prod_{m=1}^M\CC_{m\kappa}^{\frac32-\delta}$   as $\e\to 0$,
  which implies joint convergence in finite dimensional distribution as claimed
  by the lemma.

We will proceed by induction. Note that $\CZ_{\e, \beta}^{(0)}= u^{(0)}= 1$.
Assuming that almost surely with respect to the coupling between $\tilde\xi_\e$ and $\sqrt{\lambda}\Noise$, we have shown $\mathcal R\CZ_{\e, \beta}^{(m-1)}(t,x)\to u^{(m-1)}(t,x)$ in $\CC^{\frac32-\delta}_{(m-1)\kappa}$ as $\e\to 0$, we will prove that $\mathcal R\CZ_{\e, \beta}^{(m)}(t,x) \to u^{(m)}(t,x)$ in $\CC^{\frac32-\delta}_{m\kappa}$.
By the above mentioned bound  \eqref{e:Young} (with $\alpha' = \frac32-\delta$ and $\alpha = - \frac12-\delta$ therein)
one has $\tilde\xi_\e\, \mathcal R\CZ_{\e, \beta}^{(m-1)}\to \sqrt{\lambda}\Noise\, u^{(m-1)}$ in $\CC^{-\frac12-\delta}_{m\kappa}$.
Next, we write
\begin{equs}[e:Zm-um]
\mathcal R\CZ_{\e, \beta}^{(m)} (t,x) - u^{(m)}(t,x)
&=\beta
P^\e
*_\e \Big(
\tilde \xi_\e\mathcal R\CZ_{\e, \beta}^{(m-1)}
- \sqrt{\lambda}\Noise\, u^{(m-1)}
\Big)
\\&\qquad
+ \beta
\Big(P^\e *_\e  \sqrt{\lambda}\Noise u^{(m-1)}
-p* \sqrt{\lambda}\Noise u^{(m-1)}\Big)
 \;,
\end{equs}
where we define $p_t\eqdef 0$ when $t<0$, and \footnote{Note that our notation $*$ and $*_\e$ are slightly different from the standard convolution since the time variable is integrated from $0$.}
 \begin{align}
 (p*f)(t,x) &\eqdef \int_0^\infty \int_{\R} p_{t-s}(x-y) f(s,y)dyds\notag \\
 (P^\e*_\e f)(t,x)& \eqdef \int_0^\infty \int_{\e\Z} P^\e_{t-s}(x-y) f(s,y)dyds \;. \label{e:dis-con}
 \end{align}

We note that the first term on the right hand side
of \eqref{e:Zm-um} vanishes in $\CC^{ \frac32-\delta}$, using a parabolic Schauder estimate stated in Lemma~\ref{lem:Schauder} below, which yields
$\|P^\e *_\e A_\e\|_{\CC_{m\kappa}^{\frac32-\delta}}
\lesssim
\| A_\e \|_{\CC_{m\kappa}^{-\frac12-\delta}} $,
where $A_\e \eqdef \tilde \xi_\e\mathcal R\CZ_{\e, \beta}^{(m-1)}
- \sqrt{\lambda}\Noise\, u^{(m-1)}$.
Since we have shown above that
$\|A_\e\|_{\CC_{m\kappa}^{-\frac12-\delta}} $ converges to $0$ as $\e \to 0$,
$\|\beta P^\e *_\e A_\e \|_{\CC_{m\kappa}^{\frac32-\delta}}$ converges to $0$ as $\e \to 0$.

The fact that the second term on the right hand side
of \eqref{e:Zm-um} vanishes in $\CC_{m\kappa}^{ \frac32-\delta}$
is a consequence of Lemma~\ref{lem:diffSchauder} below, 
together with the fact that
\[
\|\Noise u^{(m-1)}\|_{\CC_{m\kappa}^{-\frac12-\delta}} \lesssim \|\Noise \|_{\CC_{\kappa}^{-\frac12-\delta}} \| u^{(m-1)}\|_{\CC_{(m-1)\kappa}^{\frac32-\delta}} <\infty \;.
\]
by \eqref{e:Young}.
We thus have $\mathcal R\CZ_{\e, \beta}^{(m)}(t,x)\to u^{(m)}(t,x)$  in $\CC_{m\kappa}^{\frac32-\delta}$ as $\e\to 0$.
\end{proof}

\begin{remark}
For SHE with space-time white noise in one spatial dimension,
for instance in the context of \cite{AKQ},
 the above argument based on continuity of multiplication and heat kernel convergence
 would not work. The reason is that
 the  space-time white noise
 lies in the space $\CC^{\alpha}$ with $\alpha < -\frac32$ which is much more singular than our $\Noise$,
 and the solution is only in $C^\beta$ for $\beta<\frac12$,
 thus the condition $\alpha+\beta>0$ for applying \eqref{e:Young} would not be satisfied.
 For the same reason, the argument exploited in this section
 would not be enough to prove convergence of our model in two spatial dimensions.
\end{remark}

\begin{remark}\label{rem:multi-int}
Note that the fact that the multiple integral $u^{(n)}$ in \eqref{e:def-Jn}
is classically well-defined also follows from
Lemma~\ref{lem:Schauder} and \eqref{e:Young}.
\end{remark}

\begin{remark}
It is more difficult to control the
second term on the right hand side
of \eqref{e:Zm-um}, since it involves both the discrete kernel $P^\e$ and the continuum kernel $p$.
A ``diagonal argument'' is often used in the recent literature (e.g.\ \cite{MourratWeber,HS,MR3736653,HaiMat,CanMat}) to get away with the direct comparison between the discrete object and the continuum limit.
The idea there is to introduce an intermediate object $\CZ_{\e,\bar\e, \beta}^{(m)}(t,x)$ (for $\bar\e>\e$),
defined similarly to $\CZ_{\e,\beta}^{(m)}(t,x)$
but with  noise $\tilde\xi_\e *_\e \rho_{\bar\e}$, where $\rho_{\bar\e}$ is a `smooth' mollifier at  scale $\bar\e$. By choosing $\bar\e$ sufficiently small, one could show that the difference between $\CZ_{\e,\beta}^{(m)}$ and $\CZ_{\e,\bar\e, \beta}^{(m)}$
 is small uniformly in $\e<\bar\e$; and then fixing a small $\bar\e$ one could send $\e\to 0$ thanks to smoothness.
One contribution of the present paper is that we follow a more direct approach by directly comparing the discrete object and the limiting object by proving a variant of the local central limit theorem for the discrete heat kernel.
\end{remark}

To complete the proof of Lemma~\ref{L:termconv}, it only remains to state and prove Lemmas~\ref{lem:Schauder} and \ref{lem:diffSchauder}. We first need to introduce some notation and collect some preliminary results.
For a function $f^\e$
on $\R\times \e\Z$, $\alpha \in \R_+\backslash \Z_+$,
we define the discrete analogue of $\CC^\alpha_\ka$ (see Appendix \ref{S:Besov}) by
\begin{align}
\|f^\e\|_{\CC^\alpha_\ka}^{(\e)}
=  \sum_{ |k|\le  \lfloor\alpha\rfloor }
\sup_{z \in \R\times \e\Z}  \frac{|D_\e^k f^\e(z)|}{w_\ka (z)}
 +
  \sum_{|k|=\lfloor\alpha \rfloor} \sup_{z,\bar z \in \R\times \e\Z}
\frac{ |D_\e^k f^\e(z) - D_\e^k f^\e(\bar z)|}{w_\ka (z)(\|z-\bar z\|\vee \e)^{\alpha-\lfloor \alpha \rfloor}}
	\label{e:def-CC-alpha'}
\end{align}
where $k=(k_0,k_1) \in \Z_+^2$ with $|k|\eqdef 2k_0+k_1$,
and
$D_\e^k  = \partial_t^{k_0} \nabla_x^{k_1}$ with
$\nabla_x g (x)\eqdef \e^{-1}( g(x+\e)-g(x))$.
For $\alpha<0$,
let
$\|f^\e\|_{\CC^\alpha_\kappa}^{(\e)}$ be defined analogously as in  \eqref{e:def-CC-alpha}
\begin{equ}[e:def-CC-alpha-e]
\|f^\e\|_{\CC^\alpha_\kappa}^{(\e)}
= \sup_{\Ell\in (0,1)} \sup_{z\in \R \times \e\Z} \sup_{\phi}
\frac{\Big|\int_{ \R \times (\e\Z)^d}  f(w) \phi_z^{\Ell}(w)dw \Big|}{\Ell^{\alpha} w_\ka(z)}
\end{equ}
where $\int_{ \R \times \e\Z} g(t,x)dtdx\eqdef \e\int_{ \R} \sum_{x\in\e\Z} g(t,x)dt $, and
$\sup_\phi$ is over all functions $\phi$ such that $\|\phi\|_{\mathcal C^{r_0}} \leq 1$ and supported in a unit ball.

For functions $f^\e$ defined on $\R\times \e\Z$ and $f$ defined on $\R^{d+1}$, and for $\alpha \in \R_+\backslash \Z_+$,
we compare them in the following way \footnote{Here we write the notation as $\|f^\e;f\|_{\CC^\alpha_\kappa}^{(\e)} $ instead of $\|f^\e - f\|_{\CC^\alpha_\kappa}^{(\e)} $, because $f$ is defined in the continuum and $D^k f$ is the derivative in continuum.}
\begin{equs}[e:def-fef]
\|f^\e;f\|_{\CC^\alpha_\kappa}^{(\e)} \eqdef & \sum_{ |k|\le \lfloor\alpha\rfloor }
\sup_{z \in \R\times \e\Z}  \frac{|(D_\e^k f^\e- D^k f)(z )|}{w_\ka (z)}\\
&+ \sum_{|k|=\lfloor \alpha\rfloor} \sup_{z,\bar z \in \R\times \e\Z}
\frac{| (D_\e^k f^\e(z) - D_\e^k f^\e(\bar z) ) -  (D^k f(z) - D^k f(\bar z) ) |}{ w_\ka (z) (\|z-\bar z\|\vee \e)^{\alpha-\lfloor \alpha \rfloor}}.
\end{equs}

We will need the following technical lemma which gives decompositions for the discrete and continuum heat kernels.

\begin{lemma}
\label{lem:decompose-kernel}
Let $d=1$.
We have a decomposition of the heat kernel
$p=K+R$ where  $\|R\|_{\CC^r} <\infty$ for a sufficiently large $r>0$,
and $K=\sum_{n\ge 0} K_n$ such that
$K_n(z)$ is supported on $\{z:\|z\|\le 2^{-n}\}$, and
$|\partial_t^{k_0}\partial_x^{k_1} K_n(z)| \lesssim 2^{(d+2k_0+k_1)n}$ uniformly in $n$ and $z$.

For the discrete kernel $P^\e$ defined in \eqref{def:Peps},
 we have
 $P^\e = K^\e + R^\e$,
  where $\|R^\e\|_{\CC^r}^{(\e)} <\infty$  uniformly in $\e>0$ for a sufficiently large $r>0$,
$K^\e=\sum_{n=0}^{N-1} K^\e_n + \mathring K^\e$
such that
$K^\e_n$ has the same support and bound as $K_n$ (except $\partial_x$ is replaced by finite difference),
$\e \in [2^{-N},2^{-N+1})$,
 and $\mathring K^\e$
 has support $\{z:\|z\|\le c\e\}$ for some $c>0$ and
$|\mathring K^\e(z)| \lesssim \e^{-d}$ uniformly in $z$  and $\e>0$.
(Basically, $\mathring K^\e$ contains the ``lattice scale'' information of the kernel  $P^\e$.)


Moreover, for $n\in \{0,\cdots,N\}$, one has
\begin{equ}[e:Kn3-Kn-o1]
|(D_\e^k K^\e_n - D^k K_n)(z)| \lesssim o(1) 2^{(d+|k|)n}
\end{equ}
 where $o(1)$ is a constant that vanishes as $\e^{-1} \|z\| \to \infty$ uniformly in $n$ and $z\in \R\times\e\Z$.
Finally, $R(z)$ and $R^\e(z)$ as well as their derivatives decay faster than any power as $\Vert z\Vert \to\infty$.
\end{lemma}

\begin{proof}
Except for the bound on $D_\e^k K^\e_n - D^k K_n$,
the proof essentially follows from
\cite{Hairer14,HaiMat}.
In \cite[Lemma~5.4]{HaiMat} the proof relies on
a sequence of functions $\bar \rho_n(z)$ such that
$\sum_n \bar \rho_n(z)=1$ and each $\bar \rho_n$ is supported on  $\{z:\|z\|\in (c2^{-n},c^{-1}2^{-n}) \}$ for some $c>0$. They are constructed as follows. Let $\rho : \R_+ \to [0,1]$ be a smooth ``cutoff function'' such that $\rho(s) = 0$ if $s \notin [1/2, 2]$, and such that $\sum_{n \in \Z}  \rho(2^n s) = 1$ for all $s > 0$ (this can be done by a partition of unity). As in \cite[Proof of Lemma~5.5]{Hairer14},
one can find a smooth ``norm''  $M : \R^{d+1} \to \R_+$ that is smooth, convex, strictly positive and $M(\delta^2 t, \delta x) = \delta M(t,x)$.  We then set
$\bar \rho_n(z) \eqdef  \rho(2^{n} M(z))$.
Then one defines
\begin{equ}[e:TildeKDef]
K^\e_n(z) = \bar \varrho_n(z) P^\e(z)\;, \qquad
R^\e(z) =\sum_{n < 0} \bar \rho_n(z) P^\e (z)\;, \qquad
\mathring{K}^{\e}(z) = \sum_{n \ge N} \bar \rho_n(z) P^\e(z)\;,
\end{equ}
as well as $K_n(z) = \bar \varrho_n(z) P(z)$ and
$R(z) =\sum_{n < 0} \bar \rho_n(z) P(z)$.
Then it follows immediately that  $p=K+R = \sum_{n = 0}^{\infty} K_n + R $ and
$P^\e =K^\e+R^\e= \sum_{n = 0}^{N-1} K^{\e}_n + \mathring{K}^{\e} + R^\e $.
Since $\sum_{n < 0} \bar \rho_n(z)$ is supported away from the origin, as explained in the proofs of  \cite[Lemma~5.5]{Hairer14} and \cite[Lemma~5.4]{HaiMat},
  $\Vert R \Vert_{\CC^r}$ is bounded and $\Vert R^\e \Vert_{\CC^r}^{(\e)}$ is bounded uniformly in $\e$, and
their values and derivatives  decay faster than any power.

We now bound $|\partial_t^{k_0}\nabla_x^{k_1} K_n^\e|$.
By \cite[Lemma~5.3]{HaiMat} (which can also be derived using the techniques in Appendix~\ref{sec:HK}),
\begin{equ}[e:GHatBound]
\bigl| D_\e^k P^\e (z) \bigr| \leq C\|z\|^{-d - |k|}\quad \forall\, k \in \N^{d+1},\ |k| \leq r \;,
\end{equ}
holds uniformly over $z \in \R_+\times (\e\Z)^d$ with $\|z\| \geq c \e$ for some $c>0$. Also by construction of $\bar\rho_n$ one has $|D_\e^k \bar\rho_n(z)| \lesssim 2^{n|k|} \mathbf 1_{\|z\|\in (c2^{-n},c^{-1}2^{-n})}$.
These together give
\[
|\partial_t^{k_0}\nabla_x^{k_1} K_n^\e(z)| \lesssim 2^{(d+2k_0+k_1)n}.
\]
 The bound $|\mathring K^\e(z)| \lesssim \e^d$ follows similarly.

To bound $ |D_\e^k K_n^\e - D^k K_n|$, note that $K^\e_n (z) - K_n(z) =  \bar \varrho_n(z)  (P^\e (z) - p(z))$.
By Corollary~\ref{C:grad}
\begin{equ}[e:dffGHatBound]
\bigl| (D^k_\e P^\e - D^k p) (z) \bigr| \lesssim o(1) \|z\|^{-d - |k|} \quad \forall\, k \in \N^{d+1} ,\ |k| \leq r \;,
\end{equ}
where $o(1)$ is a constant that vanishes as $\e^{-1} \|z\| \to \infty$ uniformly in $z$. From this the claimed bound on $D_\e^k K^\e_n - D^k K_n$ for each $n$ then follows as above, using the support property and the bound for $\bar\rho_n(z)$ and its derivatives.
\end{proof}

We are now ready to state Lemma \ref{lem:Schauder}, generally known as a Schauder estimate, which roughly states that the (continuum and discrete) heat kernel convolution ``improves regularity by two''. We give a proof for this in the context of our weighted H\"older spaces $\CC_\kappa^\alpha$. We will write
$\int^{(\e)} g(x) dx := \e\sum_{x\in \e\Z} g(x)$, and for $z=(t,x)$, write $\int^{(\e)} g(z) dz : = \int_{\R}\int^{(\e)}  g(t,x)dxdt$.

\begin{lemma}\label{lem:Schauder}
Let $d=1$.
Let $\alpha\in (-\frac23,-\frac12)$,
and $f\in \CC^\alpha_\ka$ for some $\kappa\ge 0$.
Then, there is a constant $C$  such that $\|p*f \|_{\CC^{\bar\alpha}_\ka} \le C \|f\|_{\CC^\alpha_\ka}$,
 where $\bar\alpha = \alpha + 2$.
Moreover, one also has
the discrete analogue of the bound:
 $\|P^\e *_\e f^\e\|_{\CC^{\bar\alpha}_\ka}^{(\e)}
  \le C \|f^\e\|_{\CC^\alpha_\ka}^{(\e)}$.
\end{lemma}
\begin{remark}
In Lemma~\ref{lem:Schauder} and also Lemma \ref{lem:diffSchauder} below, $\bar\alpha$ can be chosen arbitrarily close to $\frac32$. It would be sufficient for our purposes even with a smaller $\bar\alpha$ (say $\bar\alpha=1$, which means that the bound \eqref{e:Schauder-aim} below is unnecessary; or actually any $\bar\alpha>\frac12$ suffices) because the condition for the bound \eqref{e:Young} would still be satisfied with our noise $\Noise$ having regularity slightly below $-\frac12$. However we state here these Schauder type estimates in the stronger form.
\end{remark}

\begin{proof}
Decompose $p$ and $P^\e$ as in Lemma~\ref{lem:decompose-kernel}.
We first note that convolution of $R$ and $R^\e$
with $f$ has arbitrarily high regularity,
which is due to the fact that the smooth functions $R$ and $R^\e$ decay faster than any power and $f$ grows polynomially;
so it is enough to prove the bound with $P$ and $P^\e$ replaced by $K$ and $K^\e$, respectively.


To simplify the notation, we write $\lesssim_\ka$ for ``less than or equal to  up to a uniform constant times the weight $w_\ka$''.
 By the definition \eqref{e:def-Ckappa-alphaU}, in order to  prove the desired bound on $\|K*f \|_{\CC^{\bar\alpha}_\ka} $, here $\bar \alpha <\frac32$,
we need to prove

\begin{equs}
  \Big|
\int K(z-w)  f(w) dw
\Big|
&=
\Big|
\sum_{n\ge 0}
\int K_n(z-w)  f(w) dw
\Big|
\lesssim_\ka 1,
	 \label{e:Schauder-aim-0'}
\\
\Big|
\int K'(z-w)  f(w) dw
\Big|
&=
\Big|
\sum_{n\ge 0}
\int K'_n(z-w)  f(w) dw
\Big|
\lesssim_\ka 1,
	\label{e:Schauder-aim-0}
\\
\Big|
\int (K'(z-w) - K'(\bar z-w)) f(w) dw
\Big|
&=
\Big|
\sum_{n\ge 0}
\int (K'_n(z-w) - K'_n(\bar z-w)) f(w) dw
\Big|
\\
&\lesssim_\ka
\|z-\bar z\|^{\alpha+1-\delta},
	\label{e:Schauder-aim}
\end{equs}
where $K'=\partial_x K$.
For the  bound \eqref{e:Schauder-aim-0},
on the support of $K_n'$ one has
$\|z-w\|\le 2^{-n}$, and by Lemma~\ref{lem:decompose-kernel} one has $|K_n'|\le 2^{n(d+1)}$, so by  Lemma~\ref{lem:rem-on-Calp} (with $2^{-n}$ and $2^n K_n'$ playing the role of $\Ell$ and $\chi$ therein respectively) and $f\in \CC_\kappa^\alpha$ one has
\begin{equ}[e:int-Knp-f]
\Big|
\int K'_n(z-w)  f(w) dw
\Big| \lesssim_\ka 2^{-n(1+\alpha)} \|f\|_{ \CC^\alpha_\ka}
\end{equ}
which is summable over $n\ge 0$ since $1+\alpha>0$.  The bound \eqref{e:Schauder-aim-0'} can be proved in a similar way.


Now we prove the  bound  \eqref{e:Schauder-aim}.
Let $n_0$ be such that $2^{-(n_0+1)} \le \|z-\bar z\|\le 2^{-n_0}$. We first consider the case $n<n_0$. In this case, by the mean value theorem and the estimate of $D^2K_n$ given in Lemma~\ref{lem:decompose-kernel}, we have
\[
|K_n'(z-w)-K_n'(\bar z-w)| 
\lesssim \|z-\bar z\|2^{n(d+2)} \;.
\]
Note that
 $\supp (K_n'(z-\cdot)-K_n'(\bar z-\cdot))\subset
 \supp K_n'(z-\cdot)\cup \supp K_n'(\bar z-\cdot)
 \subset \{w: \|w-z\|\le 2^{-n+1}\}$, because if $w\in \supp K_n'(\bar z-\cdot)$ then $\|w-z\|\le \|w-\bar z\|+\|\bar z-z\|
 \le 2^{-n}+2^{-n_0}\le 2^{-n+1}$.
Then, by Lemma~\ref{lem:rem-on-Calp}
(with $2^{-n}$ playing the role of $\Ell$ therein) and the regularity assumption  $f \in \CC_\kappa^\alpha$, we have
 \begin{align*}
\sum_{n<n_0}\Big|
\int (K'_n(z-w)   - K'_n(\bar z-w)) f(w) dw
\Big|
& \lesssim_\ka \|z-\bar z\| \sum_{n<n_0}2^{-n\alpha}\\
&\lesssim \|z-\bar z\| 2^{-n_0 \alpha}\lesssim \|z-\bar z\|^{\alpha+1},
\end{align*}
where  the last step follows from $2^{-n_0}\lesssim \|z-\bar z\|$.

%

Consider  next the case $n\ge n_0$.
In this case we bound the two terms in \eqref{e:Schauder-aim} separately,
and both $K'_n$ can be bounded by $2^{n(d+1)}$.
By Lemma~\ref{lem:rem-on-Calp}
(again with $2^{-n}$ playing the role of $\Ell$ therein)
and $f\in \CC_\kappa^\alpha$,
the part $n\ge n_0$ on the LHS of \eqref{e:Schauder-aim}
is bounded by
 \[
\sum_{n\ge n_0} 2^{-n(\alpha+1)}
\lesssim 2^{-n_0 (\alpha+1)} \lesssim \|z-\bar z\|^{\alpha+1}
\]
since $\alpha+1>0$.
Therefore \eqref{e:Schauder-aim} holds and we conclude the proof to the bound on $p*f$.

Regarding the bound on $P^\e*_\e f^\e$, recalling the decomposition of $K^\e$ given in Lemma~\ref{lem:decompose-kernel},
the discrete analogue of  \eqref{e:Schauder-aim-0'}, \eqref{e:Schauder-aim-0}  and \eqref{e:Schauder-aim} can be proved in the same way. For instance, one can prove,  for $k=(0,1)$,
\begin{equ}
\Big| ({D^k_\e} K^\e *_\e f^\e)(z)
\Big|
\le
\Big|
\sum_{n =0}^{N-1}
\int^{(\e)} \!\!\!\!  D^k_\e K_n^\e (z-w)  f^\e(w) dw
\Big|+\Big|
\int^{(\e)} \!\!\!\!  D^k_\e \mathring{K}^\e (z-w)  f^\e(w) dw
\Big|
\lesssim_\ka 1.
\end{equ}
Indeed, for $n\in \{0,1,\cdots,N-1\}$ one has
$
|D^k_\e K^\e_n *_\e f^\e
| \lesssim_\ka 2^{-n(1+\alpha)} \|f^\e\|^{(\e)}_{ \CC^\alpha_\ka}$
as in \eqref{e:int-Knp-f};
and we can  sum over $n$
up to $N-1$ which yields a finite constant uniformly in $N$.
Consider the part involving $\mathring K^\e$,
namely $\int^{(\e)}  D^k_\e \mathring K^\e (z-w)  f^\e (w)dw$.
Recall that 
$\e \in [2^{-N},2^{-N+1})$.
Now on the support of $\mathring K^\e $ one has
$\|z-w\|\lesssim 2^{-N}$, and by Lemma~\ref{lem:decompose-kernel} one has $|\mathring K^\e |\lesssim 2^{Nd}$. By the definition of finite difference, one automatically has  $|D_\e^k \mathring K^\e |\lesssim 2^{N(d+|k|)}$.
Now by  Lemma~\ref{lem:rem-on-Calp} (with $r=1$ and $2^N D^k_\e  \mathring K^\e$ playing the role of $\chi$ and $\Ell=2^{-N}$), one has
\begin{equ}[e:bound-ringK-f]
\Big|
\int^{(\e)} D^k_\e \mathring K^\e (z-w)   f^\e(w) dw
\Big| \lesssim_\ka 2^{-N(1+\alpha)} \|f^\e\|^{(\e)}_{ \CC^\alpha_\ka}
\end{equ}
and the constant $2^{-N(1+\alpha)}$ vanishes as $N\to \infty$ since $\alpha+1>0$.
Thus the part involving $\mathring{K}^\e$ is negligible and we obtain the claimed bound on $D^k_\e K^\e *_\e f$. The discrete analogue of \eqref{e:Schauder-aim-0'} and \eqref{e:Schauder-aim} can be also proved analogously as the continuum case, so we omit its details.
\end{proof}

Recall the distance  $\| f_\e ; f \|_{\CC^{\alpha}_\ka}^{(\e)} $
defined in \eqref{e:def-fef}. The following lemma is needed to bound the second term on the right hand side of \eqref{e:Zm-um}

\begin{lemma}\label{lem:diffSchauder}
Let $\alpha\in (-\frac23,-\frac12)$,$\kappa>0$
and $f\in \CC^\alpha_\kappa$.
Then one has
 $\|P^\e *_\e f ; p*f \|_{\CC^{\bar\alpha}_\ka}^{(\e)}
  \le  o(1) \|f\|_{\CC^\alpha_\ka}$
where $\bar\alpha = \alpha + 2$,
and $ o(1) $ vanishes as $\e\to 0$.
\end{lemma}

\begin{proof}
We will use similar argument as in the proof of Lemma~\ref{lem:Schauder}, but
take care in the difference between the discrete and the continuum kernels.
Recall the kernels $K^\e$ and $K$ and their decompositions as introduced in
Lemma~\ref{lem:decompose-kernel},
and we write $(K^\e)'$ for $D_\e K^\e$ and $(K_n^\e)'$ for $D_\e K_n^\e$.
We first consider
the first term in \eqref{e:def-fef} for the definition of
 $\|P^\e *_\e f ; p*f \|_{\CC^{\bar\alpha}_\ka}^{(\e)}$, namely we
aim to prove that given any small constant $\eta>0$,
as $\e$ becomes sufficiently small one has
\begin{equ}
\Big|
\int^{(\e)} ( (K^\e)^\prime - K' )(z-w)  f(w) dw
\Big|
\lesssim_\ka \eta  \|f\|_{\CC^\alpha_\ka}.
\end{equ}
Indeed, given any small constant $\eta>0$,  there exists $M>0$ so that for all $z,\e$ with $\e^{-1}\|z\|>M$, the $o(1)$ constant in \eqref{e:Kn3-Kn-o1} of Lemma~\ref{lem:decompose-kernel} is less than $\eta/C_1$.
We choose $\e$ sufficiently small such that
 \begin{equ}[e:choose e]
 \e < \min ( M^{-2} , (\eta/C_2)^{2/(1+\alpha)} ).
 \end{equ}
  Here $C_1$ and $C_2$ are universal constants which we determine below.
 Let $\bar n$ be such that $\sqrt\e \in [ 2^{-\bar n} , 2^{-\bar n+1})$.

 We first suppose that $n\le \bar n$. Then, when $\|z-w\|$ is in an annulus with radius of order $2^{-n}$, we have
 \[
 \e^{-1}\|z-w\| \ge  \e^{-1}2^{-n} \ge  \e^{-1}2^{-\bar n} >M
 \]
  by our choice of $\e$.
  By \eqref{e:Kn3-Kn-o1} of Lemma~\ref{lem:decompose-kernel} and Lemma~\ref{lem:rem-on-Calp} one has
\begin{equ}
\Big|
\sum_{n=0}^{\bar n}
\int^{(\e)}\!\!\!\! ( (K_n^\e)^\prime - K_n ')(z-w)  f(w) dw
\Big|
\le C w_\ka(z) \|f\|_{\CC^\alpha_\ka} \eta C_1^{-1}\sum_{n\le \bar n} 2^{-n(1+\alpha)}
\le \frac{\eta}{2} w_\ka(z) \|f\|_{\CC^\alpha_\ka}
\end{equ}
where $C_1$ is chosen large enough, recalling that $\alpha+1>0$.

For $n>\bar n$, we bound $D_\e K^\e$ and $K'$ separately by $2^{(d+1)n}$ on their support using Lemma~\ref{lem:decompose-kernel},
and we can deal with the $\mathring{K}^\e$ as in \eqref{e:bound-ringK-f}.
This yields
\begin{equs}
 \Big| \sum_{n =\bar n}^N
\int^{(\e)} & \!\!\!\!\!\! (K_n^{\e} )^\prime (z-w)  f(w) dw
\Big|+
\Big|
\int^{(\e)} \!\!\!\! ( \mathring{K}_n^\e )^\prime(z-w)  f(w) dw
\Big| +
\Big|\sum_{n >\bar n}
\int^{(\e)}\!\!\!\!\!\!  K_n ^\prime (z-w)  f(w) dw
\Big|
\\
&\le C w_\ka(z) \sum_{n> \bar n} 2^{-n(1+\alpha)} \|f\|_{\CC^\alpha_\ka}
\le C  w_\ka(z) 2^{-\bar n(1+\alpha)} \|f\|_{\CC^\alpha_\ka}
\le C  w_\ka(z) \e^{(1+\alpha)/2} \|f\|_{\CC^\alpha_\ka}
\\
&\le C  w_\ka(z) \eta/C_2 \|f\|_{\CC^\alpha_\ka}
\le \frac{\eta}{2} w_\ka(z) \|f\|_{\CC^\alpha_\ka}
\end{equs}
where $C_2$ is chosen sufficiently large and we used our choice of $\e$ in \eqref{e:choose e}.
(Note that  in the two cases $n\ge \bar n$ and $n<\bar n$, the bounds on the heat kernels lead to the same power of $2$ namely $2^{-\bar n(1+\alpha)}$, with the only difference being that in the first case one obtains the $o(1)$ factor.)

By the same argument above combined with
the argument for the proof of \eqref{e:Schauder-aim}, we can prove that
the second term in \eqref{e:def-fef} for the definition of
 $\|P^\e *_\e f ; p*f \|_{\CC^{\bar\alpha}_\ka}^{(\e)}$
 can be made arbitrarily small when $\e$ goes to zero.
%
\end{proof}

 \appendix

\section{Weighted Besov-H\"older spaces} \label{S:Besov}

In this section, we first define the weighted Besov-H\"older spaces and prove some basic properties, in particular, Lemma \ref{lem:rem-on-Calp}. We then formulate a tightness criterion, stated in Proposition~\ref{prop:tightness-cri}.

\subsection{Definition and basic properties} \label{S:DBP}
The functional space we need is the Besov-H\"older space of space-time distributions, introduced in \cite[Section~3]{Hairer14} (see also \cite{MR3724565}); more precisely, since we work on the entire space, we will use a weighted version of these spaces, as in e.g.~\cite{MR3358965}.
 First, we endow the space-time $\R\times \R^d$ with a {\it parabolic} distance, namely for space-time points $(t,x),(\bar t,\bar x)$
 \begin{equ}[e:para-dist]
 \|(t,x)-(\bar t,\bar x) \| := \sqrt{|t-\bar t|} + |x-\bar x| \;.
 \end{equ}
The weight functions we will take here are of the form
\begin{equation}\label{e:weight}
w_\ka(z)=(1+\|z\|)^\kappa
\end{equation} for some $\kappa\ge 0$, as used in similar contexts in \cite[Setion 5]{hhnt15} or \cite{MR3358965}. Note that the weight function has the following property
\begin{equation}\label{e:weight-prop}
C_\kappa^{-1}\le \sup_{\|z_1-z_2\|\le 1}\frac{w_\ka(z_1)}{w_\ka(z_2)} \le C_\kappa
\end{equation}
 for some constant $C_\kappa\ge1$ depending on $\kappa$ only.

Given an arbitrary open subset $U$ of $\R^{d+1}$, we now define the weighted Besov-H\"older space $\CC_\kappa^\alpha(U)$ on $U$. Let $\N_0=\N\cup\{0\}, k=(k_0,k_1,\dots, k_d)\in \N_0\times \N_0^d$ and $|k|=2k_0+\sum_{i=1}^dk_i$.

For $\alpha\in \R_+\backslash \Z_+ $, let $\CC^\alpha_\ka(U)$ be the completion of $\CC_c^\infty(U)$ under the norm
\begin{equ}[e:def-Ckappa-alphaU]
\|f\|_{\CC_\kappa^\alpha(U)}
=  \sum_{|k|\le \lfloor\alpha\rfloor }
\sup_{z \in U} \frac{ |D^k f(z)|}{w_\ka(z)}
+
\sum_{|k|=\lfloor \alpha\rfloor}\sup_{\substack{z,\bar z \in U\\ \|z-\bar z\|\le1}}
\frac{ |D^k f(z) - D^k f(\bar z)| }{w_\ka(z) (\|z-\bar z\|)^{\alpha-\lfloor \alpha \rfloor}}
\end{equ}
where
 \[D^kf(z)=D^kf(t, x):=\frac{\partial^{|k|-k_0}}{\partial t^{k_0}\partial x_1^{k_1}\dots \partial x_d^{k_d}}f(t,x),\]
  for $k=(k_0, k_1, \dots, k_d)$. For $\alpha\in \Z_+$, the norm $\|\cdot\|_{\CC_\kappa^\alpha(U)}$ is simply
\begin{align*}
\|f\|_{\CC_\kappa^\alpha(U)}
= & \sum_{|k|\le \alpha }
\sup_{z \in U} \frac{ |D^k f(z)|}{w_\ka(z)}.
\end{align*}

When $\alpha < 0$, the space $\CC^\alpha_\ka(U)$ is defined as the completion of $\mathcal C_c^\infty(U)$
with respect to
\begin{equ}[e:def-CC-alpha]
\|f\|_{\CC^\alpha_\ka(U)}
= \sup_{\Ell\in (0,1)} \sup_{z\in U} \sup_{\phi}
\frac{ |f(\phi_z^{\Ell})|}{ w_\ka(z) \Ell^{\alpha} }
\end{equ}
 where
 \begin{equation}\label{e:phil}
 \phi_{(s,y)}^{\Ell}(t,x) \eqdef \Ell^{-(d+2)} \phi\left(\Ell^{-2} (t-s ), \Ell^{-1}(x-y)\right),
 \end{equation}
 and $\sup_\phi$ is over all functions $\phi$ with $\|\phi\|_{\mathcal C^{r_0}} \leq 1$ for $r_0=-\lfloor \alpha \rfloor$ and support in the  unit ball $B(0,1)$.
We also used the notation
\[f(\phi_z^{\Ell})=\langle f,  \phi_z^{\Ell}\rangle := \int_{U} f(w)\phi_z^{\Ell}(w)dw.\]

Since $C_c^\infty(U)$ is separable, clearly $C_\kappa^\alpha(U)$ is a Polish space, i.e, a complete separable metric space.  We  write $\CC_\kappa^\alpha$ to be $\CC_\kappa^\alpha(\R^{d+1})$ with $U=\R^{d+1}$, and  $\CC^\alpha$ to be $\CC^\alpha_\ka$ with $\ka=0$. Clearly $\mathcal C_\kappa^\alpha(U)=\mathcal C^\alpha(U)$ if $U$ is bounded. Note also that $\CC^\alpha$ is essentially equivalent to the Besov space $\mathcal B^\alpha_{\infty,\infty}$ (see, e.g., \cite{MR2768550,Treibel06, Hairer14}), with respect to the  parabolic distance in space-time.

We recall the definition of the restriction of a distribution to an open set. For any open set $U\subset \R^{d+1}$ and any distribution $f\in \CC_\kappa^\alpha$, the restriction of $f$ to $U$, denoted by $f|_U$, is defined via
\[\langle f|_U, \varphi\rangle:=\langle f, \varphi\rangle, \quad \forall \varphi\in \CC_c^\infty(U).\]
Thus $\{f|_U: f\in \CC_\kappa^{\alpha}\}\subset \mathcal C_\kappa^{\alpha}(U)$,
and $\|f|_U\|_{\mathcal C_\kappa^{\alpha}(U)}\le \|f\|_{\mathcal C_\kappa^{\alpha}}$.

We say that a distribution $f$ vanishes on an open set $U$, if $\langle f, \varphi\rangle=0$ for all $\varphi\in \CC_c^\infty(U)$.  Recall that the support
$\supp(f)$ of $f$ is the complement of the largest open set where $f$ vanishes. For any compact subset $K\subset\R^{d+1}$, let $\CC_\kappa^{\alpha}(K)=\{f\in\CC_\kappa^\alpha, \,\supp(f)\subset K\}$ with $\|f\|_{\mathcal C_\kappa^{\alpha}(K)} := \|f\|_{\mathcal C_\kappa^{\alpha}}$ for $f\in \CC_\kappa^{\alpha}(K)$.
\medskip

By \eqref{e:def-CC-alpha}, we have $f\in \CC^\alpha_\ka= \CC^\alpha_\ka(\R^d)$ if $f$ integrated against a test function $\phi_z^{\Ell}$ behaves as $\Ell^{\alpha}$. On the other hand, sometimes we would like to exploit the fact $f\in \CC^\alpha_\ka$
and obtain behavior of $f(\chi)$ for some test function $\chi$ which ``behaves like'', but is not really a rescaled test function
$\phi_z^{\Ell}$. This is the content of the following lemma.

\begin{lemma}\label{lem:rem-on-Calp}
Let $f\in \CC^\alpha_\ka$ for $\alpha<0$,  $\kappa >0$, $r>-\lfloor \alpha \rfloor$ be a positive integer, and $C>0$ be a constant. Then,  there exists a constant $\bar C>0$ such that for any  test function $ \chi$
 supported on $\{u:\|u\|\le C\Ell\}$
which satisfies $|D^k\chi(w)|\le C \Ell^{-(d+2+|k|)}$ for every $|k|\le r$,
one has
\[
|f(\chi_z)| \le \bar C  \Ell^\alpha \|f\|_{ \CC^\alpha_\ka} w_\ka(z)\;, \qquad \forall \Ell\in (0,1)\;,
\]
where $\chi_z (\cdot ) \eqdef \chi(\cdot -z)$. The same bound holds for functions on $\R\times \e\Z$ with $ \| \cdot \|_{ \CC^\alpha_\ka}$ replaced by $ \| \cdot \|^{(\e)}_{ \CC^\alpha_\ka}$ defined by \eqref{e:def-CC-alpha-e}.
\end{lemma}
\begin{remark}
The statement is similar to \cite[Remark~2.21]{Hairer14} (which is a statement about the notion of {\it models} therein), and our proof here is inspired by the proof of  \cite[Prop.~3.20]{Hairer14}.
The intuition behind the proof is the following: letting $2^{-n_0}\approx \Ell$, for a test function $\chi_z$ that is ``quite smooth'' (i.e. having the above assumed bound on its derivatives up to order $r>0$), we will get a factor $2^{-(n-n_0)r}$ for its wavelet coefficients which decays fast as $n\to \infty$, so $\chi_z$ ``mainly'' consists of the scale $\sim 2^{-n_0}\approx \Ell$ bits, which explains why it should satisfy the same bound as $\phi_z^{\Ell}$ in  \eqref{e:def-CC-alpha}.
\end{remark}

\begin{proof}
Recall from standard wavelet theory (see for instance \cite{MR951745}) that we can find a finite family of functions $\Psi=\{(\psi^{(i)})_{1 \le i < 2^{d+2}}\}$ and $\phi \in \CC^r$, such that the recentered and rescaled
\footnote{Note that the factor $ 2^{\frac{(d+2)n}{2}}$ indicates a scaling preserving the $L^2$ norm, since we need an $L^2$ basis, which is different from the other notation $\phi^{\Ell}$ we use in this paper which is a scaling preserving the $L^1$ norm.}
functions $\phi_y (\cdot)\eqdef \phi(\cdot -y)$ and
$$
\psi^n_{y} (z) = 2^{\frac{(d+2)n}{2}} \psi (2^{2n} (z_0-y_0), 2^{n} (z_1-y_1)) \quad \mbox{for } \psi\in\Psi,  n\ge 0, y\in \Lambda_n=(2^{-2n}\Z)\times (2^{-n}\Z)^d
$$
form an $L^2$ basis. Here $z=(z_0,z_1)$ denote the time and space coordinates for a space-time point $z$. This $L^2$ basis is useful to characterize elements in $\CC^\alpha_\ka$: for $f\in \CC^\alpha_\ka$ one has  $|f(\psi_y^n)| \lesssim 2^{-n(d+2)/2 - n\alpha} w_\ka(y)$ and $|f(\phi_z)| \lesssim w_\ka(z)$
uniformly in $n\ge 0$, $y\in \Lambda_n$ and $z\in \Lambda_0$. This is essentially  the content of
\cite[Prop.~3.20]{Hairer14} which is easily adapted to the weighted spaces. 

 For each fixed $z\in \R^{1+d}$, $\chi_z$ has the following $L^2$- expansion 
\begin{equation*}\label{e:expan-chi}
\chi_z=\sum_{n\ge0}\sum_{\psi\in\Psi}\sum_{y\in \Lambda_n} \chi_z(\psi_y^n) \psi_y^n +\sum_{y\in\Lambda_0} \chi_z(\phi_y)\phi_y.
\end{equation*}
 Note that the above expansion also converges in $\CC^r$ (\cite[Definition 2.8 and Remark A.6]{MR3724565}).
Thus,  
we have
\[
f(\chi(\cdot -z))=f(\chi_z)
= \sum_{n\ge 0}  \sum_{\psi\in \Psi}\sum_{y\in \Lambda_n} f(\psi_y^n) \chi_z(\psi_y^n)
+ \sum_{y\in \Lambda_0} f(\phi_y) \chi_z(\phi_y)\;.
\]

Note that  by the characterization of  $\CC^\alpha_\ka$ just mentioned,  for  $y$ with $\|y-z\| \lesssim 1$,
we have  $|f(\psi_y^n)| \lesssim 2^{-n(d+2)/2 - n\alpha} w_\ka(z)$ where the proportional constant only depends on $\|f\|_{\CC^\alpha_\ka}$.
Let $n_0$ be such that $\Ell \in  [2^{-n_0}, 2^{-n_0+1})$.

We first fix $n\ge n_0$.
Then, since $\chi_z$ has support size $O(2^{-n_0})$ and $\psi_y^n$ has support diameter $O(2^{-n})$,
only $y$ such that $\|y-z\| \lesssim O(2^{-n_0})$
contribute to the sum (namely, only $O(2^{(d+2)(n-n_0)})$ terms contribute).  By our assumed bound  on $D^k \chi$, we have
$|\chi_z(\psi_y^n)| \lesssim 2^{-(n-n_0) (r+\frac{d+2}{2})} 2^{n_0 (d+2)/2}$, and thus $\sum_{y\in \Lambda_n} |\chi_z(\psi_y^n)|  \lesssim 2^{-(n-n_0) (r-\frac{d+2}{2})} 2^{n_0 (d+2)/2}$, so that $|\sum_{y\in \Lambda_n} f(\psi_y^n) \chi_z(\psi_y^n)|\lesssim 2^{-(n-n_0)r-n\alpha}w_\ka(z)$.

Next, we fix $n < n_0$.
By support sizes of $\chi_z$ and $\psi^n_y$, there are only a finite number (independent of $n_0$ and $n$) that contribute to the sum.
Using again the support properties of $\chi_z$ and $\psi^n_y$ and the assumed bound on $\chi$,
we have
$|\chi_z(\psi_y^n)| \lesssim 2^{n(d+2)/2}$
and thus
$|\sum_{y\in \Lambda_n}\chi_z(\psi_y^n)| \lesssim 2^{n(d+2)/2}$,  so that $|\sum_{y\in \Lambda_n} f(\psi_y^n) \chi_z(\psi_y^n)|\lesssim 2^{-n\alpha}w_\ka(z)$.

Summing over $n$ using the above bounds, we have
$\sum_{n\ge n_0} 2^{-(n-n_0)r-n\alpha} + \sum_{n < n_0}2^{-n\alpha}$, which is bounded by $2^{-n_0 \alpha} \lesssim \Ell^\alpha$, by  \cite[Lemma~3.21]{Hairer14}.

The discrete analog of the bound can be proved in the same way, so we omit its proof.
\end{proof}

\subsection{Tightness criterion}\label{S:T}
In this section, we formulate a tightness criterion for random variables taking values in $\CC^\alpha_\ka$, which is stated in Proposition \ref{prop:tightness-cri} below. First we give some preliminary results.

\begin{lemma}\label{lem:com-emb}
Let $K\subset \R^{d+1}$ be a compact subset. Then for $-\infty<\alpha<\alpha'<\infty$ and $\kappa',\kappa\in[0,\infty)$, $\mathcal C^{\alpha'}_{\kappa'}(K)$ is compactly embedded in $\mathcal C^{\alpha}_{\kappa}(K)$.
\end{lemma}
\begin{proof}  Note that the weighted H\"older space $\mathcal C^{\alpha}_{\kappa}(K)$ coincides with the corresponding H\"older space $\mathcal C^{\alpha}(K)$ and hence is equivalent to the Besov space $\mathcal B^{\alpha}_{\infty,\infty}(K)$. The desired result is a direct consequence of \cite[Corollary 2.96]{MR2768550} which claims that $\mathcal B^{\alpha'}_{p,\infty}(K)$ is compactly embedded in $\mathcal B^{\alpha}_{p,1}(K)$ for $p\in[1,\infty]$, noting that $\mathcal B^{\alpha}_{p,1}(K)$ is continuously embedded in $\mathcal B^{\alpha}_{p,\infty}(K)$ (see, e.g., \cite[Remark  2.12]{MR3724565}).
\end{proof}

\begin{lemma}\label{lem:compact-embedding}
For $-\infty<\alpha<\alpha'<\infty$ and $0\le\kappa'<\kappa$, $\mathcal C^{\alpha'}_{\kappa'}$ is compactly embedded in $\mathcal C^{\alpha}_{\kappa}$.
\end{lemma}

\begin{proof}
Consider a sequence $\{f_n, n\in \mathbb N\}$ with
  \begin{equ}[e:C-alpha'-M]
  \sup_{n}\|f_n\|_{\mathcal C^{\alpha'}_{\kappa'}}\le M<\infty.
  \end{equ}
To get the desired result, it suffices to show that there exists a subsequence of $\{f_n, n\in \mathbb N\}$ which converges in $\mathcal C^{\alpha}_{\kappa}$.

Consider a sequence of bounded open subsets $U_m=B(0,2m)\subset \R^{1+d}, m\in \mathbb N$. Clearly $U_m\subset U_{m+1}$ and $\bigcup_{m=1}^\infty U_m=\R^{1+d}$. Then by Lemma~\ref{lem:com-emb}, for each $U_m$, $\mathcal C^{\alpha'}_{\kappa'}(\overline{U_m})$ is compactly embedded in $\mathcal C^{\alpha}_{\kappa}(\overline{U_m})$. Since for all $n\in\N$, $f_n |_{U_m} \in \mathcal C^{\alpha'}_{\kappa'}(\overline{U_{m}})$ with
$\| f_n |_{U_m}\|_{\mathcal C^{\alpha'}_{\kappa'}(\overline{U_{m}})} \leq \| f_n\|_{\mathcal C^{\alpha'}_{\kappa'}}\leq M$, we can find a subsequence $\{f_{n_k^{(1)}}\}$ such that  $ f_{n_k^{(1)}}\big|_{U_1}  $ converges in $\mathcal C^{\alpha}_{\kappa}(\overline{U_1})$, and inductively, we can find a subsequence $\{f_{n_k^{(m+1)}}\}$ of $\{f_{n_k^{(m)}}\}$ such that $ f_{n_k^{(m+1)}} \big|_{U_{m+1}} $ converges in $\mathcal C^{\alpha}_{\kappa}(\overline{U_{m+1}})$. Thus, the diagonal sequence $ f_{n_k^{(k)}}\big|_{U_m}  $ converges as $k\to\infty$ in $\mathcal C^{\alpha}_{\kappa}(\overline{U_{m}})$ for every $m\in\mathbb N$. Assume that $ f_{n_k^{(k)}}\big|_{U_m} \to f_0^{(m)}$ in $\mathcal C^{\alpha}_{\kappa}(\overline{U_m})$ as $k\to\infty$.
Then clearly $f_0^{(m+1)}\big|_{U_m} =f_0^{(m)}$. Letting $f_0=\lim_{m\to\infty} f_0^{(m)}$, in the sense that for all $m\in \N$, $\langle f_0, \varphi\rangle=\langle f_0^{(m)}, \varphi\rangle$ for all $\varphi\in \CC_c^\infty(U_m)$,  we have that
\begin{equation}\label{e:bund-m}
\lim_{k\to\infty}\Big\|  \left(f_{n_k^{(k)}}-f_0\right) \Big|_{U_m}\Big\|_{\mathcal C^{\alpha}_{\kappa}(U_{m})}=0, \quad \forall\ m\in\N.
\end{equation}
Thus, $f_{n_k^{(k)}}$ is a Cauchy sequence in $\mathcal C^{\alpha}_{\kappa}(U_{m})$ for any $m \in \mathbb N$. In what follows, we will use the convention $\|f\|_{\CC_\kappa^\alpha(U)}:=\|f|_U\|_{\CC_\kappa^\alpha(U)}$ for any open set $U$.

Now for any fixed $\e>0$, one can find $m_0$ such that,
\[
\frac{(1+|x|)^{\kappa'}}{(1+|x|)^\kappa}< \frac{\e}{4M}, ~~\forall x\in U_{m_0}^c,
\]
and hence by \eqref{e:C-alpha'-M}
 \[\left\|f_{n_k^{(k)}}\right\|_{\mathcal C^{\alpha}_{\kappa}(\overline U_{m_0}^c)}< \frac{\e}{4M}M=\frac\e{4}.\]
Noting that $f_{n_k^{(k)}}\big|_{U_{m_0+1}}$ is a Cauchy sequence in $\mathcal C^{\alpha}_{\kappa}(U_{m_0+1})$,  one can find $k_0\in\mathbb N$ such that if $k_1, k_2>k_0$,
\[
\Big\|f_{n_{k_1}^{(k_1)}}-f_{n_{k_2}^{(k_2)}}\Big\|_{\mathcal C^{\alpha}_{\kappa}(U_{m_0+1})}<\frac\e{2},
\]
and hence
\begin{align*}
\Big\|f_{n_{k_1}^{(k_1)}}-f_{n_{k_2}^{(k_2)}}\Big\|_{\mathcal C^{\alpha}_{\kappa}}
&\le \Big\|f_{n_{k_1}^{(k_1)}}-f_{n_{k_2}^{(k_2)}}\Big\|_{\mathcal C^{\alpha}_{\kappa}(U_{m_0+1})}
+\Big\|f_{n_{k_1}^{(k_1)}}-f_{n_{k_2}^{(k_2)}}\Big\|_{\mathcal C^{\alpha}_{\kappa}(\overline{U_{m_0}}^c)}\\
&\le \frac\e{2}+\frac\e{4} \times 2 =\e,
\end{align*}
where the first step follows from \eqref{e:def-CC-alpha} and the observation that for any $\phi_z^{\Ell}$ with $\Ell\in(0,1)$, its support   $\supp (\phi_z^{\Ell})\subset B(z,1)$ must be a subset of either $U_{m+1}$ or $\overline{U_m}^c$. This concludes the proof.
\end{proof}

For any given $r > 0$, recall that one can identify a finite family of functions $\Psi=\{(\psi^{(i)})_{1 \le i < 2^{d+2}}\}$ and $\phi \in \CC^r$
(same as in the proof of Lemma \ref{lem:rem-on-Calp}) such that the following holds (with $r>\max\{|\alpha|,|\alpha'|\}$ in the next Lemma).

\begin{lemma}\label{lem:Lp-bound}
Fix $\alpha, \alpha'$, $\gamma$ and $\kappa$, such that $\alpha<\alpha'$, $\max\{|\alpha|,|\alpha'|\}<\gamma$ and $\kappa>0$. Let $(f_\e)_{\e>0}$ be a family of random linear forms on $\mathcal C_c^\gamma(\R\times \R^{d})$. Assume that there exist constants $C < \infty$ and $p>\max\{\frac{d+2}{\alpha'-\alpha}, \frac{d+2}\kappa\}$ such that for every $\e>0$, the following two statements hold:
	\begin{equ}
		\label{e:tight1}
		\sup_{z\in \R^{d+1}}
		\E \Big[ \Big| \int_{ \R^{d+1}} f_\e(w)\phi(w - z) dw \Big|^p  \Big]\le C \; ;
	\end{equ}
	and, for every $i \in \{1,\ldots, 2^{d+2}-1\}$ and $m \in \N$,
	\begin{equ}
		\label{e:tight2}
		\sup_{(t,x) \in\R^{1+d}}
		 \E \Big[ \Big| \int_{\R^{d+1}}  \!\!\!\!\! f_\e(s,y)\psi^{(i)}(2^{2m}(s - t), 2^m(y - x)) dyds  \Big|^p  \Big]
		\le C \, 2^{-mp(d+2 +\alpha')},
	\end{equ}
	where we use the notation $z=(t,x)\in\R^{1+d}$.
	Then
\[\sup_{\e>0}\E\left[\left(\|f_\e\|_{\CC^{\alpha}_{\kappa}}\right)^p\right]<\infty.\]
\end{lemma}
\begin{proof}
To prove the desired result, by \cite[Proposition~2.4]{MR3358965},
 it suffices to show
\begin{align*}
 \E \left[  \sup_{z\in \Lambda_0} \Big| \int_{ \R^{d+1}} f_\e(w)\phi(w - z) dw \Big|^p \Big/ w_\ka(z)^p \right]
\end{align*}
and
\begin{align*}
\E \left[\sup_{m\ge0}\sup_{z\in \Lambda_m}  2^{mp(d+2+\alpha)}\Big| \int_{\R^{d+1}}  \!\!\!\!\! f_\e(s,y)\psi^{(i)}(2^{2m}(s - t), 2^m(y - x))dyds   \Big|^p \Big/ w_\ka(z)^p \right]
\end{align*}
are  bounded  uniformly in $\e>0$ and $i\in \{1, \dots, 2^{d+2}-1\}$, where $\Lambda_m=(2^{-2m}\Z)\times (2^{-m}\Z)^d$.
We will show the uniform boundedness of the second term, and the boundedness of the first one is similar (and easier) and thus omitted.
Indeed,
\begin{align*}
&\E \left[\sup_{m\ge0}\sup_{z\in \Lambda_m}
 2^{mp(d+2+\alpha)}
 \Big| \int_{\R^{d+1}}  \!\!\!\!\! f_\e(s,y)\psi^{(i)}(2^{2m}(s - t), 2^m(y - x)) dyds   \Big|^p
 \Big/ w_\ka(z)^p \right]\\
& \lesssim
\sum_{m\ge0}\sum_{z\in \Lambda_m}2^{mp(d+2+\alpha)}
\E \left[  \Big| \int_{\R^{d+1}}  \!\!\!\!\! f_\e(s,y)\psi^{(i)}(2^{2m}(s - t), 2^m(y - x)) dyds   \Big|^p\right]
\Big/w_\ka(z)^p
\\
&\lesssim\sum_{m\ge0}\sum_{z\in \Z^{d+1}}2^{mp(d+2+\alpha)}2^{-mp(d+2+\alpha')} 2^{m(d+2)}
\Big/w_\ka(z)^p
\quad \lesssim 1,
\end{align*}
where the second inequality follows from the fact  that there are at most order $2^{m(d+2)}$ many elements in the restriction of $\Lambda_m$ to the unit ball of $\R^{d+1}$ and  \eqref{e:tight2}, the third step follows from the property \eqref{e:weight-prop} of the weight function, and finally the series is summable due to the assumption $p>\max\{\frac{d+2}{\alpha'-\alpha}, \frac{d+2}\kappa\}$, where $p>\frac{d+2}{\alpha'-\alpha}$ yields $\sum_{m\ge 0} 2^{m(d+2-p(\alpha'-\alpha))}<\infty$ and $p>\frac{d+2}{\kappa}$ implies $\sum_{z\in\R^{d+1}} w_\kappa(z)^{-p}\lesssim\sum_{z\in\R^{d+1}} \frac{1}{1+\|z\|^{\kappa p}}<\infty$.
\end{proof}

We are now ready to state and prove the tightness criterion.
\begin{proposition}\label{prop:tightness-cri}
Fix $\alpha, \alpha'$, $\gamma$, $\kappa$ and $\kappa'$, such that $\alpha<\alpha'$, $\max\{|\alpha|,|\alpha'|\}<\gamma$ and $0<\kappa'<\kappa$. Let $(f_\e)_{\e>0}$ be a family of random linear forms on $\mathcal C_c^\gamma(\R\times \R^{d})$ satisfying \eqref{e:tight1} and \eqref{e:tight2} with $\kappa$ being replaced by $\kappa'$ therein and $p>\max\{\frac{d+2}{\alpha'-\alpha}, \frac{d+2}{\kappa'}\}$. Then $\{f_\e\}_{\e>0}$ is a tight family in $\CC_{\kappa}^\alpha$.
\end{proposition}

\begin{remark} In the setting with usual Euclidean (not parabolic) distance, such a tightness criterion was obtained for {\em local Besov-H\"older space} in \cite[Theorem~1.1]{MR3724565}.
\end{remark}

\begin{proof}
By Lemma \ref{lem:Lp-bound}, we have
\[\sup_{\e}\E \left[\left(\|f_\e\|_{\mathcal C_{\kappa'}^{\alpha'}}\right)^p\right] \le C_0\]
for some  constant $C_0\in(0, \infty)$.
Hence, Markov's inequality yields that, for all $\e>0$,
\begin{align*}
\P\left(\|f_\e\|_{\mathcal C_{\kappa'}^{\alpha'}}>M\right)&\le M^{-p} \E \left[\left(\|f_\e\|_{\mathcal C_{\kappa'}^{\alpha'}}\right)^p\right]\\
&\le M^{-p} C_0.
\end{align*}
Thus, for any $\delta>0$, one can find a constant $M_\delta$ sufficiently large such that
\[\inf_{\e} \P\left(\|f_\e\|_{\mathcal C_{\kappa'}^{\alpha'}}\le M_\delta\right)=1-\sup_{\e} \P\left(\|f_\e\|_{\mathcal C_{\kappa'}^{\alpha'}}> M_\delta\right)\ge 1-\delta.\]
Therefore, noting that
by Lemma~\ref{lem:compact-embedding} the set $\left\{f\in \mathcal C_{\kappa}^\alpha: \|f\|_{\mathcal C_{\kappa'}^{\alpha'}}\le M\right\}$ is compact in $\mathcal C_{\kappa}^\alpha$  for any $M\in(0,\infty)$, the family
$(f_\e)_{\e>0}$ is tight in $\mathcal C_{\kappa}^\alpha$.
\end{proof}

 \section{Heat kernel estimates}\label{sec:HK}

In this section, we prove local limit theorem type of results for the random walk transition kernel and its space-time gradients, which is used in Section \ref{S:conv}. We will formulate our results in general dimensions since the proof is the same.

Let $P_t(x)$, $x\in \Z^d$, denote the transition kernel of a rate $1$ simple symmetric random walk $S$ on $\Z^d$, and let $p_t(x)$ denote the standard heat kernel with time slowed down by a factor of $d$. Then $P_t$ and $p_t$ have Fourier transforms
  \begin{equation}\label{pPhi}
\Phi_t(\theta) = {\bold E}[e^{i \langle \theta, S_t\rangle}] = e^{-\frac{t}{d}\sum_{i=1}^d(1-\cos \theta_i)}, ~~\theta \in [-\pi, \pi]^d ~ \text{ and } ~ \phi_t(\theta) = e^{-\frac{t\|\theta\|^2}{2d}}, ~~ \theta \in \R^d,
\end{equation}
where $\|\theta\|$ denotes the Euclidean norm, and
\begin{equation}\label{eq:Ptext}
\begin{aligned}
P_t(x) & = \frac{1}{(2\pi)^d} \int_{[-\pi, \pi]^d} e^{-i\langle \theta, x\rangle} \Phi_t(\theta) {   d}\theta, \qquad  x\in\Z^d,  \\
p_t(x) & = \frac{1}{(2\pi)^d} \int_{\R^d} e^{-i\langle\theta, x\rangle} \phi_t(\theta) {   d}\theta, \qquad x\in \R^d.
\end{aligned}
\end{equation}

Recall that for $k=(k_0, k_1, \ldots, k_d)$ with $k_0, \ldots, k_d \in \N_0:=\N\cup\{0\}$ and $|k|:=2k_0+\sum_{i=1}^d k_i$, and for $f:\R\times \R^d\to \R$,
\begin{equation}
D^k_\epsilon f:= \partial^{k_0}_t \nabla_{d, \epsilon}^{k_d}\cdots \nabla_{1, \epsilon}^{k_1} f(t,x), \quad \mbox{where } \nabla_{i,\epsilon} g(x) = \epsilon^{-1}(g(x+\epsilon e_i)-g(x)),
\end{equation}
with $e_i$ being the $i$-th unit vector in $\R^d$. We will also denote
$$
D^k f:= \partial^{k_0}_t \partial_{x_d}^{k_d} \cdots \partial_{x_1}^{k_1} f(t,x).
$$
We then have the following result:

\begin{lemma}\label{L:grad}
For any $k=(k_0, k_1, \ldots, k_d)$ with $k_0, k_1, \ldots, k_d\in \N_0:=\N\cup\{0\}$, we have
\begin{equation}\label{eq:gradbd}
\big|D_1^k P_t(x) - D^k p_t(x)\big| \leq  \frac{o(1)}{(\sqrt{t}+\|x\|)^{|k|+d}},
\end{equation}
where $o(1)\to 0$ as $\sqrt{t}+\|x\|\to\infty$.
\end{lemma}
\begin{proof} First we note that
$$
\partial_t P_t(x) = \frac{1}{2d}\sum_{i=1}^d\big(P_t(x+e_i)+P_t(x-e_i)-2P_t(x)\big) = \frac{1}{2d} \sum_{i=1}^d \nabla_{i,1}^2 P_t(x-e_i) =: \frac{1}{2d} \sum_{i=1}^d \nabla_{i,1}^2 T_iP_t(x),
$$
where $T_if(x) = f(x-e_i)$. Therefore
$$
D_1^k P_t(x) = \frac{1}{(2d)^{k_0}} \Big(\sum_{i=1}^d \nabla_{i,1}^2 T_i\Big)^{k_0} \prod_{i=1}^d \nabla_{i,1}^{k_i} P_t(x).
$$
Similarly, $\partial_t p_t(x) = \frac{1}{2d} \sum_{i=1}^d \partial^2_{x_i} p_t(x)$, and hence
$$
D^k p_t(x) = \frac{1}{(2d)^{k_0}} \big(\sum_{i=1}^d \partial^2_{x_i}\big)^{k_0} \prod_{i=1}^d \partial^{k_i}_{x_i} p_t(x).
$$
By expanding the products and compare $D_1^k P_t(x)$ with $D^k p_t(x)$ term by term, it suffices to show that for any
$l_1, \ldots, l_d\in \N\cup\{0\}$ with $\sum_{i=1}^d l_i = |k|$, and $z\in \Z^d$ with $|z|_1:=\sum_{i=1}^d |z_i|\leq |k|$, we have
\begin{equation} \label{eq:gradbd2}
\Big|\prod_{i=1}^d \nabla_{i, 1}^{l_i} P_t(x-z) - \prod_{i=1}^d\partial^{l_i}_{x_i} p_t(x)\Big| \leq  \frac{o(1)}{(\sqrt{t}+\|x\|)^{|k|+d}} \qquad \mbox{as } \sqrt{t}+\|x\|\to\infty.
\end{equation}

Let $\alpha=\alpha(|k|):=\frac{|k|+d+1}{2(|k|+d)}>\frac{1}{2}$. We first consider the case $\|x\|> t^\alpha$, which implies that $|x_m|\geq \|x\|/\sqrt{d} \geq t^\alpha/\sqrt{d}$ for some
$1\leq m\leq d$. Since $p_t(x) = \prod_{i=1}^d (2\pi t/d)^{-1/2}e^{-dx_i^2/2t}$, it is easily seen that
\begin{equation}\label{eq:ptbd}
\begin{aligned}
\Big|\prod_{i=1}^d \partial_{x_i}^{l_i} p_t(x)\Big| & = \prod_{i=1}^d \frac{e^{-\frac{d x_i^2}{2td^{-1}}} (|x_i|/{\sqrt{td^{-1}}})^{l_i+1}}{\sqrt{2\pi} |x_i|^{l_i+1}}  \Big|a_{l_i}\Big(\frac{x_i}{\sqrt{td^{-1}}}\Big)\Big| \\
& \leq \frac{C e^{-\frac{d x_m^2}{2td^{-1}}} (|x_m|/{\sqrt{td^{-1}}})^{l_m+1}}{\sqrt{2\pi} |x_m|^{l_m+1}}  \Big|a_{l_m}\Big(\frac{x_m}{\sqrt{td^{-1}}}\Big)\Big|  \prod_{i\neq m} \frac{1}{(\sqrt{t})^{l_i+1}}\\
& \leq \frac{C e^{-c|x_m|^{2-1/\alpha}}}{|x_m|^{|l_m|+1}}
= \frac{o(1)}{\|x\|^{|k|+d}},
\end{aligned}
\end{equation}
where $a_l$ is a polynomial of degree $l$,  the negative powers of $t$ can be neutralised by $e^{-c x_m^2/t}$ with any $c>0$, and $o(1)\to 0$ as $\sqrt{t}+\|x\|\to\infty$.

We can similarly bound
$$
\Big|\prod_{i=1}^d\nabla_{i, 1}^{l_i} P_t(x-z)\Big| \leq 2^{|k|} \max_{|\tilde z|_1\leq 2|k|} P_t(x+\tilde z) \leq C e^{-c\|x\|^{2-1/\alpha}} = \frac{o(1)}{\|x\|^{|k|+d}},
$$
where the bound on $P_t(x+\tilde z)$ for $\|x\|>t^\alpha$ and $|\tilde z|_1\leq 2|k|$ follows from an elementary large deviation bound for the simple symmetric random walk. Together with \eqref{eq:ptbd}, this implies \eqref{eq:gradbd2} when $\|x\|>t^\alpha$.

We now consider the case $\|x\|\leq t^\alpha $. We will use Fourier transform calculations. Alternatively, one can also use the Edgeworth expansion for $P_t$ (see e.g.~\cite[page 45]{Hal92}).  Note that if we extend the definition of $P_t(x)$ in \eqref{eq:Ptext} to all $x\in\R^d$, then we can write
$$
\begin{aligned}
\prod_{i=1}^d \nabla_{i, 1}^{l_i} P_t(x-z) & = \idotsint_{[0,1]^{|k|}} \prod_{i=1}^d \partial^{l_i}_{x_i} P_t\Big(x-z+\sum_{i=1}^d\sum_{j=1}^{l_i} s_{i,j} e_i\Big) \prod_{i=1}^d \prod_{j=1}^{l_i} {d}s_{i,j} \\
& =  \idotsint_{[0,1]^{|k|}} \frac{(-i)^{|k|}}{(2\pi)^d} \int_{[-\pi, \pi]^d} \!\!\!\!\!\!\!\!\!  e^{-i\langle\theta, x-z+\sum s_{i,j}e_i\rangle} \prod_{i=1}^d \theta_i^{l_i} \Phi_t(\theta) {   d}\theta  \prod_{i,j}{   d}s_{i,j}.
\end{aligned}
$$
An analogous identity holds for $\prod_{i=1}^d \nabla_{i, 1}^{l_i}p_t(x-z)$. Therefore we have
\begin{align}
& \Big|\prod_{i=1}^d \nabla_{i, 1}^{l_i} P_t(x-z) - \prod_{i=1}^d \nabla_{i, 1}^{l_i} p_t(x-z) \Big| \nonumber\\
\leq\ & \Big| \idotsint\limits_{[0,1]^{|k|}} \frac{1}{(2\pi)^d} \int\limits_{|\theta|_\infty\leq t^{-\frac13}} \!\!\!\!\!\!  e^{-i\langle\theta, x-z+\sum s_{i,j}e_i\rangle)} \prod_{i=1}^d \theta_i^{l_i} \big\{\Phi_t(\theta)-\phi_t(\theta)\big\} {   d}\theta  \prod{   d}s_{i,j}\Big| \nonumber \\
&  \ \ + \frac{1}{(2\pi)^d} \!\!\!\!\!\! \int\limits_{t^{-\frac13}<|\theta|_\infty<\pi} \!\!\!\!\!\!\!\!\!  |\theta|_\infty^{|k|} \Phi_t(\theta) {   d}\theta + \frac{1}{(2\pi)^d} \!\!\!\!\!\! \int\limits_{t^{-\frac13}<|\theta|_\infty} \!\!\!\!\!\! |\theta|_\infty^{|k|} \phi_t(\theta) {   d}\theta  \nonumber \\ \label{eq:xlt}
\leq\ & \frac{1}{(2\pi)^d} \!\!\!\!\!\!\! \int\limits_{|\theta|_\infty\leq t^{-\frac13}} \!\!\!\!\!\!\!  |\theta|_\infty^{|k|}  |\Phi_t(\theta)-\phi_t(\theta)| {   d}\theta + \frac{1}{(2\pi)^d} \!\!\!\!\!\!\!\!\! \int\limits_{t^{-\frac13}<|\theta|_\infty<\pi} \!\!\!\!\!\!\!\!\!\!\!  |\theta|_\infty^{|k|} \Phi_t(\theta) {   d}\theta + \frac{1}{(2\pi)^d} \!\!\!\!\!\!\!\! \int\limits_{t^{-\frac13}<|\theta|_\infty} \!\!\!\!\!\!\!\! |\theta|_\infty^{|k|} \phi_t(\theta) {   d}\theta.
\end{align}
Recall \eqref{pPhi}. We note that as $t\to\infty$, the second and third terms can be bounded by $Ce^{-c t^{\frac13}}$ for some $c, C$ depending only on $|k|$. For $\|x\|\leq t^\alpha$, $\sqrt{t}+\|x\|\to\infty$ implies $t\to\infty$, and it is easily seen that
$$
Ce^{-c t^{\frac13}} \leq C e^{-\frac{c}{2} (\sqrt{t}+\|x\|)^{1/3\alpha} },
$$
which is bounded by the right hand side of \eqref{eq:gradbd}.

For the first term in \eqref{eq:xlt}, we can apply Taylor expansion and bound it by
\[
\begin{aligned}
\frac{1}{(2\pi)^d} \!\!\!\!\!   \int\limits_{|\theta|_\infty\le t^{-\frac13}} \!\!\! \!\!  |\theta|_\infty^{|k|} e^{-\frac{t\|\theta\|^2}{2d}} \Big| e^{-\frac{t}{d}\sum_{i=1}^d\big(1-\cos\theta_i -\frac{\theta_i^2}{2}\big)} -1\Big| {   d}\theta \leq C \!\!\!\!\!   \int\limits_{|\theta|_\infty\le t^{-\frac13}} \!\!\! \!\!  |\theta|_\infty^{|k|} e^{-\frac{t\|\theta\|^2}{2d}} t |\theta|_\infty^4 {   d}\theta \leq  Ct^{-\frac{|k|+d+2}{2}},
\end{aligned}
\]
which is also bounded by the right hand side of \eqref{eq:gradbd} when $\|x\|\le t^\alpha=t^{\frac{|k|+d+1}{2(|k|+d)}}$.

To prove \eqref{eq:gradbd2} for $\|x\|\leq t^\alpha$, it remains to bound $|\prod_{i=1}^d \nabla_{i, 1}^{l_i} p_t(x-z) - \prod_{i=1}^d\partial^{l_i}_{x_i} p_t(x)|$. By the mean value theorem, there exists $x'\in \R^d$ with $x_i'\in [x_i-z_i, x_i-z_i+l_i]$ such that
\begin{equation}\label{eq:gradpbd}
\begin{aligned}
\Big|\prod_{i=1}^d \nabla_{i, 1}^{l_i} p_t(x-z) - \prod_{i=1}^d\partial^{l_i}_{x_i} p_t(x)\Big|
& = \Big|\prod_{i=1}^d\partial^{l_i}_{x_i} p_t(x') -\prod_{i=1}^d\partial^{l_i}_{x_i} p_t(x)\Big|  \\
& \leq \|x-x'\| \sup_{y=x+t(x'-x) \atop t\in [0,1]}\sum_{i=1}^d \Big|\partial_{x_i}\prod_{i=1}^d\partial^{l_i}_{x_i} p_t(y)\Big|.
\end{aligned}
\end{equation}
We can now apply \eqref{eq:ptbd}, where $x$ is replaced by $y$ with $|y-x|_1 \leq 2|k|$, and $l_1, \ldots, l_d$ are replaced by $l'_1, \ldots, l'_d$ with $\sum_{i=1}^d l'_i=|k|+1$. Note that for some $1\leq m \leq d$, we will have $|y_m| \geq \|y\|/\sqrt{d}$. In the product in \eqref{eq:ptbd}, for each $1\leq i\leq d$, we can bound
\begin{equation}\label{hbd1}
\frac{e^{-\frac{d y_i^2}{2td^{-1}}} (|y_i|/{\sqrt{td^{-1}}})^{l'_i+1}}{\sqrt{2\pi} |y_i|^{l'_i+1}}  \Big|a_{l'_i}\Big(\frac{y_i}{\sqrt{td^{-1}}}\Big)\Big| \leq \frac{C}{(\sqrt t)^{l'_i+1}},
\end{equation}
while for $i=m$, we can also give an alternative bound
\begin{align}
\frac{e^{-\frac{d y_m^2}{2td^{-1}}} (|y_m|/{\sqrt{td^{-1}}})^{l'_m+1}}{\sqrt{2\pi} |y_m|^{l'_m+1}}  \Big|a_{l'_i}\Big(\frac{y_m}{\sqrt{td^{-1}}}\Big)\Big|
& = \frac{e^{-\frac{d y_m^2}{2td^{-1}}} (|y_m|/{\sqrt{td^{-1}}})^{|k|+d+1}\big|a_{l'_i}\big(\frac{y_m}{\sqrt{td^{-1}}}\big)\big|}{\sqrt{2\pi} |y_m|^{|k|+d+1}}  (\sqrt{td^{-1}})^{|k|+d-l'_m} \nonumber \\
& \leq C \frac{(\sqrt{t})^{|k|+d-l'_m}}{|y_m|^{|k|+d+1}} \leq C \frac{(\sqrt{t})^{\sum_{i\neq m}(l'_i+1)}}{\|y\|^{|k|+d+1}}. \label{hbd2}
\end{align}
If $\|x\|\leq \sqrt{t}$, then we just apply \eqref{hbd1} in \eqref{eq:ptbd} and substitute the resulting bound into \eqref{eq:gradpbd}. If $\|x\|> \sqrt{t}$, then in \eqref{eq:ptbd}, we apply both \eqref{hbd1} and \eqref{hbd2} and then substitute the resulting bound into \eqref{eq:gradpbd}. Either way, since the bounds are
uniform in $|y-x|_1 \leq 2|k|$, we find that \eqref{eq:gradpbd} can be bounded by $C(\|x\|+\sqrt{t})^{-|k|-d-1}$, which is bounded by the right hand side of \eqref{eq:gradbd2}. This concludes the proof of the lemma.
\end{proof}

\begin{corollary}\label{C:grad}
Let $P^\e_t(x) := \e^{-d} P_{\e^{-2}t}(\e^{-1}x)$. For any $k=(k_0, k_1, \ldots, k_d)$, we have
\begin{equation}
\big|D_\e^k P^\e_t(x) - D^k p_t(x)\big|  \leq  \frac{o(1)}{(\sqrt{t}+\|x\|)^{k+d}},
\end{equation}
where $o(1)\to 0$ as $\e^{-1} (\sqrt{t}+\|x\|) \to \infty$.
\end{corollary}
\begin{proof}
We note that
$$
D^k p_t(x) = \e^{-|k|-d}D^kp_s(y)\big|_{s=\e^{-2}t, y=\e^{-1}x} \quad \mbox{and} \quad
D_\e^k P^\e_t(x) = \e^{-|k|-d} D_1^kP_{s}(y)\big|_{s=\e^{-2}t, y=\e^{-1}x}.
$$
The result then follows from Lemma \ref{L:grad}.
\end{proof}

%

{\bf Acknowledgement.}  We thank the referees for very helpful comments. H.~Shen is supported by NSF grant DMS-1712684 / DMS-1909525 and NSF grant DMS-1954091,
and would like to acknowledge the support by
the Trimester Program ``Randomness, PDEs and Nonlinear Fluctuations'' at
Hausdorff Research Institute for Mathematics where part of this work is done. J.~Song is supported by Shandong University grant 11140089963041.
 R.~Sun is supported by NUS grant R-146-000-253-114. L.~Xu is supported by Macao S.A.R grant FDCT  0090/2019/A2 and University of Macau grant  MYRG2018-00133-FST.

\bibliographystyle{alphaabbr}
\bibliography{Reference}

%
%
%
%
\end{document}